\documentclass[11pt]{amsart}
\usepackage{graphicx}
\usepackage{pgf,tikz,pgfplots}
\usepackage{mathrsfs}
\usepackage{mathpple}
\usepackage{ bbold }
\usepackage{tikz-cd}
\usepackage{amsmath}
\usepackage{tikz}
\usepackage{mathdots}
\usepackage{verbatim}

\usepackage{cancel}
\usepackage{color}
\usepackage{siunitx}
\usepackage{array}
\usepackage{multirow}
\usepackage{amssymb}
\usepackage{hyperref}
\usepackage{simpler-wick} 

\usepackage{latexsym,bm,amsmath,amssymb}

\usetikzlibrary{arrows}
\newcommand{\C}{\mathbb{C}}

\newcommand{\Q}{\mathbb{Q}}
\newcommand{\Z}{\mathbb{Z}}

\newcommand{\pa}{\partial}

\newcommand{\vac}{\left| 0\right\rangle}
  \newcommand{\bb}{\bullet}
\newcommand{\ww}{\circ}

\newcommand\sbullet[1][.5]{\mathbin{\vcenter{\hbox{\scalebox{#1}{$\bullet$}}}}}

\DeclareMathOperator{\Tr}{Tr}
\DeclareMathOperator{\Hom}{Hom}

\numberwithin{equation}{section}

\newtheorem{definition}{Definition}[section]
\newtheorem{theorem}{Theorem}[section]
\newtheorem{corollary}{Corollary}[section]
\newtheorem{lemma}{Lemma}[section]
\newtheorem{prop}{Proposition}[section]
\newtheorem{example}{Example}[section]

\newtheorem{remark}{Remark}[section]
\newtheorem{assumption}{Assumption}[section]

\DeclareMathAlphabet\EuScript{U}{eus}{m}{n}
\SetMathAlphabet\EuScript{bold}{U}{eus}{b}{n}

\begin{document}
	
\title{Large $N$ vertex algebras via Deligne Category}
\author{Keyou Zeng}

	\address{
		K. Zeng:
		Center of Mathematical Sciences and Applications, Harvard University, MA 02138, USA
	}
	
	\email{kzeng@cmsa.fas.harvard.edu}

	\thanks{}%

\begin{abstract}
In this paper, we propose a new construction of vertex algebras using the Deligne category. This approach provides a rigorous framework for defining the so-called large $N$ vertex algebra, which has appeared in recent physics literature. We first define the notion of a vertex algebra in a symmetric monoidal category and extend familiar constructions in ordinary vertex algebras to this broader categorical context. As an application, we consider a $\beta\gamma$ vertex algebra in the Deligne category and construct the large $N$ vertex algebra from it. We study some simple properties of this vertex algebra and analyze a certain vertex Poisson algebra limit.
\end{abstract}	

	\maketitle
	
\section{Introduction}
The goal of this paper is to provide a rigorous mathematical definition of the vertex algebra that appears in the work \cite{Costello:2018zrm} of K. Costello and D. Gaiotto. This vertex algebra can be regarded as a certain large $N$ limit of a BRST reduction of a $\beta\gamma$ system. For the finite $N$ case, this vertex algebra can also be obtained from $4d$ $\EuScript{N} = 4$ $U(N)$ super Yang-Mills theory via the $4d/2d$ duality constructed in \cite{Beem:2013sza}. Though the $4d$ super Yang-Mills theory is not mathematically well-defined, the vertex algebra that it produces is well-defined and contains a wealth of information about it (see, e.g., \cite{Arakawa:2017fdq,Arakawa:2023cki}, for a review).

For finite $N$, this vertex algebra is defined as a certain BRST cohomology of a $\beta\gamma$ system $\{Z_1(z),Z_2(z)\}$ valued in $\mathfrak{gl}_N$, along with a $bc$ system also valued in $\mathfrak{gl}_N$. While simple in definition, this vertex algebra has a complicated structure \cite{Arakawa:2023cki}. For example, its BRST cohomology is difficult to compute and is not known in most cases. The work \cite{Costello:2018zrm} considers a certain ``large $N$ limit" of this vertex algebra, which is more accessible. The BRST cohomology is related to a cyclic cohomology and the OPE structure is conjecturally related to a B-model topological string.

However, one immediately encounters some problems in defining this ``large $N$ limit". As a working definition in \cite{Costello:2018zrm}, this vertex algebra includes, for example, all $\Tr Z_1^n$ as independent generators. This only happens in the strict $N \to \infty$ limit, where no trace relations appear. On the other hand, the computation in \cite{Costello:2018zrm} also requires us to keep track of $N$ as a parameter in the OPE expansion. For example, we have the OPE $\Tr Z_1(z) \Tr Z_2(0) \sim N/z$. Such a vertex algebra cannot be defined by naively taking the limit $N \to \infty$. We overcome this problem using the Deligne category $\mathrm{Rep}(\mathrm{GL}_N)$.

Deligne categories \cite{deligne2007categories,deligne1982tannakian} are, roughly speaking, interpolations of the tensor categories of representations of the classical algebraic groups $\mathrm{GL}_n,\mathrm{O}_n, \mathrm{Sp}_n$ to non-integer rank. Their applications in mathematical physics have been explored in various contexts, such as in quantum field theory \cite{Binder:2019zqc} and their connections to integrable systems \cite{feigin2023gaudin}. The idea of using Deligne categories to define certain algebras first appears in \cite{etingof2023new,kalinov2023deformed}, where an alternative definition of the deformed double current algebra is provided. 

To extend this approach rigorously to vertex algebras, we must first understand what a vertex algebra object is in a symmetric monoidal category. This notion is not entirely new. P. Etingof discussed a Kac-Moody type vertex algebra in the Deligne category in \cite{ETINGOF2016473}. W. Niu constructed a vertex algebra object in the category of (quasi-coherent) sheaves over the Higgs branch in \cite{Niu2023}. In the language of chiral algebras \cite{beilinson2004chiral}, a definition should follow from the abstract definition in \cite{francis2012chiral,Raskin2019} using $D$-modules valued in a category. However, a general framework has not yet been established in the vertex algebra literature\footnote{A. Latyntsev proposed a definition in his thesis \cite{LatyntsevPhD}, but it is not carefully developed and contains certain issues, such as with the locality axiom.}. In the first part of this paper, we discuss the notion of a vertex algebra object in a symmetric monoidal category $(\EuScript{C}, \otimes, \mathbb{1})$. Besides being symmetric monoidal, the category $\EuScript{C}$ is required to satisfy certain additional technical assumptions such as locally presentable (see assumption~\ref{assum:vertex_cat}). We then propose the following definition.
\setcounter{section}{3}
\begin{definition}
A vertex algebra object in $\EuScript{C}$ consists of an object $V$ together with a collection of morphisms
	\begin{enumerate}
	\item A vacuum vector: $\vac \in \mathrm{Hom}_{\EuScript{C}}(\mathbb{1},V)$.
	\item A translation map: $T \in \mathrm{Hom}_{\EuScript{C}}(V,V)$.
	\item An infinite collection of maps: $\cdot_{n} \in \mathrm{Hom}_{\EuScript{C}}(V\otimes V,V)$ for $n \in \Z$.  We also require that for any compact objects $X,Y$ in $\EuScript{C}$ and $\alpha \in \Hom_{\EuScript{C}}(X,V)$, $\beta \in \Hom_{\EuScript{C}}(Y,V)$, there exists an integer $K$ such that
	\begin{equation*}
		\cdot_{n}\circ(\alpha\otimes \beta) = 0
	\end{equation*}
	for any $n \geq K$.
	We also denote $Y(z) = \sum_{n\in \Z}\cdot_{n}\; z^{-n-1}\in \mathrm{Hom}(V\otimes V,V)[[z,z^{-1}]]$.
\end{enumerate}
These morphisms are required to satisfy the following axioms:
\begin{enumerate}
	\item Vacuum. $Y(z)\circ(\vac\otimes \mathrm{id}_V) = l_V$. Furthermore, $Y(z)\circ (\mathrm{id}_V\otimes \vac) \in \mathrm{Hom}_{\EuScript{C}}(V\otimes \mathbb{1},V)[[z]]$, so that $Y(z)\circ (\mathrm{id}_V\otimes \vac)|_{z = 0} = r_V$.
	\item Translation. $T\vac = 0$. Furthermore, $T\circ Y(z) - Y(z)\circ(\mathrm{id}_V\otimes T) = \pa_z Y(z)$
	\item Locality. For any compact objects $X,Y$ and $\alpha \in \Hom_{\EuScript{C}}(X,V)$, $\beta \in \Hom_{\EuScript{C}}(Y,V)$, we can find an integer $K$ such that
	\begin{equation*}
		\begin{aligned}
			(z - w)^K\big(&Y(w)\circ(\mathrm{id}_V\otimes Y(z))\circ(\alpha \otimes \beta\otimes \mathrm{id}_V) \\
			&-  Y(z)\circ(\mathrm{id}_V\otimes Y(w))\circ(\sigma\otimes\mathrm{id}_V)\circ(\alpha \otimes \beta \otimes \mathrm{id}_V)\big) = 0\,.
		\end{aligned}
	\end{equation*}
\end{enumerate}
In the above formulas, we write $l$, $r$, and $\sigma$ for the left unit, right unit, and braiding in the category $\EuScript{C}$, respectively.
\end{definition}
\setcounter{section}{1}
In particular, when $\EuScript{C} = \mathrm{Vect}_k$, the above definition recovers the usual notion of a vertex algebra over the field $k$. 

In section \ref{sec:vertex_general}, we establish basic properties of vertex algebras in $\EuScript{C}$, parallel to those of ordinary vertex algebras, such as the Borcherds identity and the construction of the normal-ordered product. Moreover, this more general categorical framework allows us to study functorial properties of vertex algebras. For instance, given a symmetric monoidal functor $F$ that preserves filtered colimits, the image of a vertex algebra under $F$ is again a vertex algebra in the target category. As an important corollary, the functor $\mathrm{Hom}_{\EuScript{C}}(\mathbb{1},-)$ sends a vertex algebra in $\EuScript{C}$ to an ordinary vertex algebra over the field (or ring) $\mathrm{End}_{\EuScript{C}}(\mathbb{1})$. 

It is possible to construct many familiar examples of vertex algebras in this categorical framework. In this paper, however, we focus on the $\beta\gamma$ vertex algebra (also known as the symplectic boson). Given an object $X$ equipped with an antisymmetric morphism $\omega: X \otimes X \to \mathbb{1}$, we construct a vertex algebra structure on $V^{\beta\gamma} = S(Xt^{-1}[t^{-1}])$. This vertex algebra has a field denoted by $X(z)$, understood as a morphism $X\otimes V^{\beta\gamma} \to V^{\beta\gamma}((z))$, which satisfies the OPE
	\begin{equation*}
	X(z)\circ(\mathrm{id}_X\otimes X(w)) = \frac{1}{z - w}l_V\circ(\omega\otimes\mathrm{id}_{V^{\beta\gamma}}) + :X(z)X(w):\,.
\end{equation*}
For $\EuScript{C} = \mathrm{Vect}_k$, this construction reproduces the usual $\beta\gamma$ vertex algebra associated with a symplectic vector space.

In section \ref{sec:ver_Deligne}, we apply the above machinery to the example of the Deligne category $\EuScript{C} = \mathrm{Rep}(\mathrm{GL}_N)$. We consider a $\beta\gamma$–$bc$ system $\{Z_1(z), Z_2(z), b(z), c(z)\}$ in the Deligne category, apply the functor $\mathrm{Hom}(\mathbb{1},-)$, and then perform a BRST reduction. Working in the Deligne category automatically ensures that there are no trace relations among traces of arbitrary powers of the “matrices” $Z_i$, $b$, and $c$, without taking the parameter $N \to \infty$. Moreover, $N$ can be either a finite number or an indeterminate parameter. This construction precisely captures the feature desired for a definition of the large $N$ vertex algebra considered in \cite{Costello:2018zrm}.

We also study several important properties of this large $N$ vertex algebra. First, the categorical approach naturally provides a construction of a family of modules over the vertex algebra. Second, as analyzed in \cite{Costello:2018zrm}, the BRST cohomology is closely related to the cyclic cohomology of an algebra; we establish this connection using properties of the Deligne category. A key aspect of a vertex algebra is its OPE expansion, which, for the $\beta\gamma$ system, is governed by the combinatorics of Wick contractions. Notably, in the vertex algebra constructed via the Deligne category, graphs of Wick contractions correspond to fat/ribbon graphs. In section \ref{sec:chiralPoi}, we consider two Poisson vertex algebra limits of this large $N$ vertex algebra, with the most interesting limit capturing the planar part of the OPE expansion. The full vertex algebra, incorporating non-planar corrections, then naturally realizes a deformation quantization of this planar Poisson vertex algebra.

\section{Preliminaries}

\subsection{Category-theoretic preliminaries}
\label{sec:cat}

We assume the reader is familiar with the definition of a monoidal category and related concepts. Here, we briefly introduce some notations and terminologies that will be used throughout this paper and refer to \cite{etingof2015tensor} for further details.

Recall that a monoidal category is a category $\EuScript{C}$ equipped with a tensor product functor $\otimes: \EuScript{C}\times \EuScript{C} \to \EuScript{C}$, a unit object $\mathbb{1}\in \EuScript{C}$, an associativity isomorphism $a$:
$$
a_{X,Y,Z}: (X\otimes Y)\otimes Z \overset{\cong}{\to} X\otimes(Y\otimes Z)\,,
$$
a left unit isomorphism $l_X:\mathbb{1}\otimes X\overset{\cong}{\to} X$ and a right unit isomorphism $r_X: X\otimes\mathbb{1} \overset{\cong}{\to} X$. A monoidal category is normally denoted by $(\EuScript{C},\otimes,\mathbb{1},a,l,r)$. For simplicity, we also denote a monoidal category by $(\EuScript{C},\otimes,\mathbb{1})$.  

Given two monoidal categories $(\EuScript{C}_i, \otimes_i, \mathbb{1}_i)$ for $i = 1, 2$, a (lax) monoidal functor between them consists of a functor $F: \EuScript{C}_1 \to \EuScript{C}_2$ together with a morphism $\epsilon: \mathbb{1}_2 \to F(\mathbb{1}_1)$ and a natural transformation $J: F(-) \otimes F(-) \to F(- \otimes -)$ satisfying the associativity and unitality conditions. We do not assume $\epsilon$ and $J$ to be isomorphisms here. However, when they are, we call $(F, \epsilon, J)$ a strong monoidal functor.

A braided monoidal category is a monoidal category together with a natural isomorphism $\sigma_{X,Y}: X\otimes Y \to Y\otimes X$ called braiding. A symmetric monoidal category is a braided monoidal category for which the braiding satisfies $\sigma_{Y,X}\circ\sigma_{X,Y} = \mathrm{id}_{X\otimes Y}$. A braided monoidal functor between braided monoidal categories is a monoidal functor which respects the braidings.

In this paper, we will consider categories that are not abelian. However, an additive category is insufficient for our construction, so we will work with an intermediate concept called a Karoubian category, where only idempotents are required to have kernels and cokernels.

Given a category $\EuScript{C}$, an idempotent is an endomorphism $e:X\to X$, such that
\begin{equation*}
	e\circ e = e\,.
\end{equation*}

\begin{example}
	Let $(\EuScript{C},\otimes,\mathbb{1})$ be a symmetric monoidal category, such that $\Q \in \mathrm{End}_{\EuScript{C}}(\mathbb{1})$. Given an object $X$, any element $g$ in the symmetric group $S_n$ induces an endomorphism $\sigma_g$ on $X^{\otimes n}$ via the symmetric braiding. We have the following idempotent on $X^{\otimes n}$ 
	\begin{equation}\label{eqn:sym_end}
			e_{\mathrm{Sym}} = \sum_{g\in S_n} \frac{1}{n!}\sigma_g\,.
	\end{equation}
\end{example}

An idempotent $e$ is said to split if there is an object $Y$ and morphisms $p:X\to Y$, $i:Y\to X$ such that $e = i\circ p$ and $\mathrm{id}_Y = p\circ i$.
\begin{definition}
	A category $\EuScript{C}$ is called Karoubian if every idempotent splits.
\end{definition}

\begin{definition}
	A Karoubi envelope of a category $\EuScript{C}$ is a tuple $(\EuScript{C}^{kar},\iota)$ where $\EuScript{C}^{kar}$ is Karoubian and $\iota:\EuScript{C} \to \EuScript{C}^{kar}$ is a fully-faithful functor such that for any Karoubian category $\EuScript{D}$, the restriction function
	\begin{equation*}
		\mathrm{Fun}(\EuScript{C}^{kar},\EuScript{D}) \to \mathrm{Fun}(\EuScript{C},\EuScript{D})
	\end{equation*}
	is an equivalence of categories.
\end{definition}

Given a category $\EuScript{C}$, one can always construct its Karoubi envelope $\EuScript{C}^{kar}$ as follows. The objects of $\EuScript{C}^{kar}$ are pairs $(X,e)$ where $X$ is an object in $\EuScript{C}$ and $e :X\to X$ is an idempotent. The morphisms from $(X,e)$ to $(X',e')$ are given by
\begin{equation*}
	\mathrm{Hom}((X,e),(X',e')) = \{f\in \mathrm{Hom}(X,X')| f\circ e  = e'\circ f\}\,.
\end{equation*}
If a category $\EuScript{C}$ is already Karoubian, then its Karoubi envelope $\EuScript{C}^{\mathrm{kar}}$ is equivalent to $\EuScript{C}$. Hence, we can also use a pair $(X,e)$ to represent an object in a Karoubian category.

Let $\EuScript{C}$ be a Karoubian, symmetric monoidal category, and let $\mathbb{Q} \in \mathrm{End}_{\EuScript{C}}(\mathbb{1})$. The symmetric idempotent $e_{\mathrm{Sym}}$ splits, giving an object $S^n(X)$ with morphisms $i: S^n(X) \to X^{\otimes n}$, $p: X^{\otimes n} \to S^n(X)$, satisfying $i\circ p = e_{\mathrm{Sym}}$ and $\mathrm{id}_{X^{\otimes n}} = p\circ i$. $S^n(X)$ is called the $n$-th symmetric tensor power of $X$. Equivalently, it can be identified with the kernel of $\mathrm{id} - e_{\mathrm{Sym}}$.
\begin{lemma}\label{lem_Hom1Sn}
	For any objects $X$ and $Z$, we have
	\begin{equation*}
		\mathrm{Hom}_{\EuScript{C}}(Z,S^n(X)) = \mathrm{Hom}_{\EuScript{C}}(Z,X^{\otimes n})^{S_n}\,.
	\end{equation*} 
\end{lemma}
\begin{proof}
	By definition, $\mathrm{Hom}_{\EuScript{C}}(Z,S^n(X)) = \{f\in \mathrm{Hom}_{\EuScript{C}}(Z,X^{\otimes n})| e_{\mathrm{Sym}} \circ f = f\}$. For any such $f$, we have $\sigma_g\circ f = \sigma_g\circ e_{\mathrm{Sym}} \circ f = e_{\mathrm{Sym}} \circ f = f$ for any $g\in S_n$. On the other hand, suppose  $f\in \mathrm{Hom}_{\EuScript{C}}(Z,X^{\otimes n})$ that satisfies $\sigma_g\circ f = f$ for any $g\in S_n$. Clearly this implies $e_{\mathrm{Sym}} \circ f = f$. Thus we have proved  $\mathrm{Hom}_{\EuScript{C}}(Z,S^n(X)) = \{f\in \mathrm{Hom}_{\EuScript{C}}(Z,X^{\otimes n})| \sigma_g\circ f = f, \text{ for any } g\in S_n\}$.
\end{proof}

We use the following terminology in our paper, following \cite{coulembier2021monoidal}.

\begin{definition}
	\begin{enumerate}
		\item A pseudo-abelian category is an additive  Karoubian category.\footnote{Many references use the term pseudo-abelian synonymously with Karoubian. In this paper, the Karoubi envelope is always followed by the additive envelope. Hence, it is convenient to have a name for a category that is additive and Karoubian.}
		\item A pseudo-tensor category is a pseudo-abelian, rigid, symmetric monoidal category.
		\item A tensor category is an abelian, rigid, symmetric monoidal category.
	\end{enumerate}
\end{definition}
Although our terminology ``(pseudo)-tensor" hides the adjective ``symmetric"—which in vertex algebra literature may refer only to a braided monoidal category, it is important for our construction that the category be symmetric monoidal (see Remark \ref{rmk:braided}).

The notion of a compact object will also be important in our construction.

\begin{definition}
	A category $I$ is filtered if it is nonempty and satisfies
	\begin{enumerate}
		\item For every two objects  $j,j'$ in $I$, there exists an object $k$ and two morphisms $j\to k$ and $j' \to k$.
		\item For every two parallel morphisms $f,g:j\to k$, there exists an object $l$ and a morphism $h:k\to l$ such that $h\circ f= h\circ g$.
	\end{enumerate}
	
	A filtered colimit in a category $\EuScript{C}$ is the colimit of a diagram $F:I\to \EuScript{C}$, where $I$ is filtered. Such a diagram is also called an ind object in $\EuScript{C}$.
\end{definition}

Ind objects in $\EuScript{C}$ form a category $\mathrm{Ind}(\EuScript{C})$ called the ind completion of $\EuScript{C}$. Given two ind objects $X:I\to \EuScript{C}$ and $Y:J\to \EuScript{C}$, the space of morphisms from $X$ to $Y$ is defined to be
\begin{equation*}
	\mathrm{Hom}_{\mathrm{Ind}(\EuScript{C})}(X,Y) = \lim_{i\in I}\underset{j\in J}{\mathrm{colim}}\mathrm{Hom}_{\EuScript{C}}(X_i,Y_j)\,.
\end{equation*}

\begin{definition}
	An object $X$ in a locally small category $\EuScript{C}$ is called compact if the functor 
	\begin{equation*}
		\mathrm{Hom}_{\EuScript{C}}(X,-) : \EuScript{C} \to \mathrm{Set}
	\end{equation*}
	preserves filtered colimits.
	
	We denote by $\EuScript{C}^c$ the full subcategory of compact objects in $\EuScript{C}$.
\end{definition}

\begin{definition}
	A category $\EuScript{C}$ is compactly generated if $\EuScript{C} \cong \mathrm{Ind}(\EuScript{C}^c)$.
\end{definition}

By definition, objects in $\EuScript{C}$ are compact in $\mathrm{Ind}(\EuScript{C})$. The converse is not necessarily true in general. However, in this paper, we will focus on the case when $\EuScript{C}$ is Karoubian. In this case, compact objects in $\mathrm{Ind}(\EuScript{C})$ are all isomorphic to objects in $\EuScript{C}$. 
\begin{prop}[\cite{kashiwara2005categories}]
	The functor $\EuScript{C} \to (\mathrm{Ind}(\EuScript{C}))^c$ is an equivalence of categories if and only if $\EuScript{C}$ is Karoubian. 
\end{prop}
\begin{proof}
	First, we show that $A \in \mathrm{Ind}(\EuScript{C})$ is a compact object of $\mathrm{Ind}(\EuScript{C})$ if and only if there exists an object $X \in \EuScript{C}$ and morphisms $i:A\to X$ and $p:X\to A$ such that $p\circ i = \mathrm{id}_A$.
	
	For $A \in \mathrm{Ind}(\EuScript{C})^c$, let $A = \underset{i\in I}{\mathrm{colim}}A_i$. Since $A$ is compact, we have
	\begin{equation*}
		\mathrm{Hom}(A,A) = \underset{i\in I}{\mathrm{colim}}\,\mathrm{Hom}(A,A_i)\,.
	\end{equation*}
	Therefore, there exists an object $X = A_i \in \EuScript{C}$ for some $i \in I$ with morphisms $i:A\to X$ such that $\mathrm{id}_A = p\circ i$.
	
On the other hand, if there exists an object $X \in \EuScript{C}$ and morphisms $i:A\to X$, $p:X\to A$ such that $p\circ i = \mathrm{id}_A$, then for any filtered colimit $Y = \mathrm{colim}_i Y_i$, we have a map
	\begin{equation*}
		\mathrm{Hom}(A,Y) \to \mathrm{Hom}(X,Y) \overset{\cong}{\to} \mathrm{colim}_i\mathrm{Hom}(X,Y_i) \to \mathrm{colim}_i\mathrm{Hom}(A,Y_i)\,.
	\end{equation*}
	We can check that this map is an isomorphism. Therefore $A$ is compact.
	
	$\EuScript{C} \to (\mathrm{Ind}(\EuScript{C}))^c$ induces an equivalence of categories if and only if any compact object $A$ of $\mathrm{Ind}(\EuScript{C})$ is isomorphic to an object $A' \in \EuScript{C}$. We have shown that this condition is equivalent to the condition of $\EuScript{C}$ being Karoubian.
\end{proof}

When $\EuScript{C}$ is a monoidal category, $\mathrm{Ind}(\EuScript{C})$ naturally inherits a monoidal structure. If $\EuScript{C}$ is Karoubian, then for any two compact objects $X, Y \in \mathrm{Ind}(\EuScript{C})$, their tensor product $X \otimes Y$ is always isomorphic to some $X' \otimes Y'$ in $\EuScript{C}$, making it compact as well.

\begin{corollary}\label{cor:comp_tensor}
	Let $\widetilde{\EuScript{C}} = \mathrm{Ind}(\EuScript{C})$ be the Ind completion of a Karoubian category. Then the tensor product preserve compact objects of $\widetilde{\EuScript{C}}$.
\end{corollary}

\subsection{Weyl algebra in pseudo-tensor category}
\label{sec:Weyl}
A preliminary step in constructing the ordinary $\beta\gamma$ vertex algebra is to construct a Weyl algebra using a symplectic vector space. In this section, we construct a Weyl algebra object in the categorical setting.

Recall that the Weyl algebra associated with a symplectic vector space $(V,\omega)$ can be defined by the quotient
\begin{equation*}
	\EuScript{W}(V) = T(V)/(u\otimes v - v\otimes u - \omega(u,v))\,,
\end{equation*}
where $T(V)$ is the free tensor algebra. This construction extends naturally to any tensor category. However, if $\EuScript{C}$ is only a pseudo-tensor category, we can no longer use this quotient construction. In this section, we present a construction of the Weyl algebra as a deformation quantization of the symmetric algebra. We assume that $\Q \in \mathrm{End}_{\EuScript{C}}(\mathbb{1})$.

For any object $X$ in $\EuScript{C}$, the tensor algebra $T(X)$, as an ind object in $\EuScript{C}$, is defined as
\begin{equation*}
	T(X) = \bigoplus_{k\geq 0}T^k(X) = \bigoplus_{k\geq 0} X^{\otimes k}\,.
\end{equation*}
It has a natural associative product given by the tensor product $T^k(X)\otimes T^l(X) \to T^{k+l}(X)$. 

Using the symmetric idempotent \eqref{eqn:sym_end}, we define the symmetric tensor $S^k(X)$ and denote:
\begin{equation*}
	S(X) = \bigoplus_{k\geq 0}S^{k}(X)\,.
\end{equation*}
For each $k$, the splitting of the symmetric idempotent $e_{\mathrm{Sym}}$ provides two morphisms $i_k: S^k(X) \to T^k(X)$ and $p_k: T^k(X) \to S^k(X)$. Summing over all $k\geq 0$, we obtain morphisms $i:S(X) \to T(X)$ and $p:T(X) \to S(X)$. The symmetric tensor $S(X)$ then inherits a multiplication $m$ from the tensor algebra $T(X)$, defined as
\begin{equation*}
	m:S(V)\otimes S(V) \overset{i\otimes i}{\to} T(V)\otimes T(V) \overset{\otimes}{\to} T(V) \overset{p}{\to} S(V)\,.
\end{equation*}

Now we assume that $X$ is equipped with an antisymmetric two-form
\begin{equation*}
	\begin{aligned}
		\omega:X \otimes X \to \mathbb{1}\,,\\
		\omega\circ\sigma = -\omega\,.
	\end{aligned}
\end{equation*}
$\omega$ extends to a series of maps 
\begin{equation*}
	\omega_{ij}^{(2)}: T^k(X)\otimes T^l(X) \to T^{k-1}(X)\otimes T^{l-1}(X)  
\end{equation*}
by applying $\omega$ to the $i$-th and $j$-th $X$ and identities on the others. Explicitly, we can write $\omega_{ij}^{(2)}$ as the composition
\begin{equation*}
	T^k(X)\otimes T^l(X) \overset{\sigma_{(k,k-1,\dots,i)}\otimes\sigma_{(1,2,\dots,j)} }{\longrightarrow} T^{k}(X)\otimes T^{l}(X) \overset{\mathrm{id}_X^{\otimes k -1}\otimes \omega\otimes \mathrm{id}_X^{l-1}}{\longrightarrow} T^{k-1}(X)\otimes T^{l-1}(X) \,,
\end{equation*}
where $\sigma_{(k,k-1,\dots,i)}$ is the braiding map that corresponds to the permutation $(k,k-1,\dots,i)$.

We define $\boldsymbol{\omega}^{(2)}: T^k(X)\otimes T^l(X) \to T^{k-1}(X)\otimes T^{l-1}(X) $ by
\begin{equation*}
	\boldsymbol{\omega}^{(2)} = \sum_{i = 1}^k\sum_{j = 1}^l\omega_{ij}^{(2)}\,.
\end{equation*}

It induces a map on the symmetric tensor, which we also denote by $\boldsymbol{\omega}^{(2)}$:
\begin{equation*}
	\boldsymbol{\omega}^{(2)} : S(X)\otimes S(X) \to S(X)\otimes S(X)\,.
\end{equation*}

We define the Weyl algebra as the Moyal-Weyl deformation quantization of the symmetric algebra.
\begin{definition}
	The Weyl algebra $\EuScript{W}(X,\omega)$ associated to $(X,\omega)$ is the associative algebra object $(S(V),\star)$, where $\star$ is given by
	\begin{equation*}
		-\star - = m\circ \exp(\frac{\boldsymbol{\omega}^{(2)}}{2})(-\otimes -)\,.
	\end{equation*}
\end{definition}

\begin{remark}\label{rmk:star}
	The above formula, as an infinite sum, requires further clarification. More precisely, the star product $\star \in \mathrm{Hom}(S(X)\otimes S(X),S(X))$ is determined by a compatible sequence of truncated maps $\mathrm{lim}_{n,m}\mathrm{Hom}(\bigoplus_{i = 0}^n S^i(X)\otimes \bigoplus_{j = 0}^mS^j(X),S(X))$. Each map is given by a finite sum
	\begin{equation}\label{eqn:star_trancated}
		m\circ \sum_{k = 0}^{\min\{n,m\}}\frac{1}{k!}( \frac{\boldsymbol{\omega}^{(2)}}{2})^k\,.
	\end{equation}
\end{remark}

In defining the $\star$-product, we haven't used the fact that $\omega$ is antisymmetric. In fact, the star product has the property that shifting the map $\omega$ by a symmetric map $\alpha:T^2(X) \to \mathbb{1}$, $\alpha\circ\sigma = \alpha$ leads to an isomorphic star product. To construct this isomorphism we first extend $\alpha$ to a series of maps 
\begin{equation*}
	\alpha_{ij}^{(1)}: T^k(X)\otimes T^{k-2}(X) 
\end{equation*}
by applying $\alpha$ to the $i$-th and $j$-th $X$ and identities on the others. Then we define the map $\boldsymbol{\alpha}^{(1)} = \sum_{i<j}\alpha_{ij}^{(1)}: T(X) \to T(X)$. It also induces a map $\boldsymbol{\alpha}^{(1)} = S(X) \to S(X)$. Then the isomorphism is given by exponentiating $\boldsymbol{\alpha}^{(1)}$.
\begin{equation}\label{eqn_isostar}
	\exp(\frac{1}{4}\boldsymbol{\alpha}^{(1)}):S(X) \to S(X)\,.
\end{equation}

The proof of this fact can be found in \cite{nlab:star} for vector spaces and easily extends to general categorical settings.

\begin{prop}
	Let $\alpha : X\otimes X \to \mathbb{1}$ be a symmetric two form, i.e. $\alpha \circ \sigma = \alpha$. Denote $\star_{\omega}$ the star product defined by $\omega$ and $\star_{\omega+\alpha}$ the star product defined by $\omega + \alpha$. Then we have the following commutative diagram
	\begin{equation*}
		\begin{tikzcd}[sep = huge]
			S(X)\otimes S(X) \arrow[r,"e(\frac{\boldsymbol{\alpha}^{(1)}}{4})\otimes e(\frac{\boldsymbol{\alpha}^{(1)}}{4})"] \arrow[d,"\star_{\omega}"]& S(X)\otimes S(X) \arrow[d,"\star_{\omega + \alpha}"]\\
			S(X) \arrow[r,"e(\frac{\boldsymbol{\alpha}^{(1)}}{4})"]& S(X)
	\end{tikzcd}	\end{equation*}
\end{prop}

Shifting by a symmetric form $\alpha$ is related to the choice of ordering in the identification of symmetric algebra and the non-commutative Weyl algebra. The standard Moyal-Weyl product $\star_{\omega}$ is associated with the symmetric ordering. Sometimes a different ordering, called normal ordering, is also used. This ordering is more convenient in defining the module associated with a polarization. 

We call a polarization of $(X,\omega)$ a decomposition $X = L_+\oplus L_-$, so that the symplectic form also decomposes as $\omega = \omega_+ + \omega_-$, where
\begin{equation*}
	\omega_+ \in \mathrm{Hom}(L_-\otimes L_+,\mathbb{1}),\;\omega_- \in \mathrm{Hom}(L_+\otimes L_-,\mathbb{1})
\end{equation*} 
and $\omega_+ = -\omega_-\circ \sigma$. Given a polarization, we can consider the star product associated with the two form $2\omega_{-}$, which only differs from $\omega$ by a symmetric two form $\omega_+-\omega_-$.
\begin{equation*}
	\star_{2\omega_-}=\exp(\boldsymbol{\omega_-}^{(2)}) \in \mathrm{Hom}(S(X)\otimes S(X),\mathbb{1})\,.
\end{equation*}
This product corresponds to the so called normal ordering. The advantage of using $\star_{2\omega_-}$ is that we can naturally define a $(S(X),\star_{2\omega_-})$ left module structure on $S(L_-)$. The module map is given by the same formula as $\star_{2\omega_-}$, composed with $\pi_{L_-}$, where  $\pi_{L_-}$ is the natural projection 
\begin{equation*}
	\pi_{L_-}: S(L_+\oplus L_-) \to S(L_-)\,.
\end{equation*}
Using the isomorphism \eqref{eqn_isostar}, we obtain a left $\EuScript{W}(X,\omega)$ module structure on $S(L_-)$, which we denote by $\star_l$:
\begin{equation}\label{eqn_starl}
	\star_l:\EuScript{W}(X,\omega)\otimes S(L_-) \to S(L_-)\,.
\end{equation}
This map can be written as follows:
\begin{equation*}
	- \star_l - = \pi_{L_-}\circ m \circ \exp(\boldsymbol{\omega}_-^{(2)})\circ (\exp(\frac{(\boldsymbol{\omega}_+ - \boldsymbol{\omega}_-)^{(1)}}{4})-\otimes -)\,.
\end{equation*}

We will also consider a one-parameter family version of the Weyl algebra, over the polynomial ring $\C[\hbar]$. It is an ind object, defined as $\EuScript{W}(X,\omega)_{\hbar} = S(X)[\hbar] = \bigoplus_{n\geq 0}S(X)\hbar^n$. The corresponding star product $\star_{\hbar}$ is defined as 
	\begin{equation*}
	-\star_{\hbar} - = m\circ \exp(\frac{\hbar \boldsymbol{\omega}^{(2)}}{2})(-\otimes -)\,.
\end{equation*}
The precise meaning of the above exponential is similar to the explanation in Remark \ref{rmk:star}. It is defined by a sequence of truncated maps, each given by a finite sum. Defining this one-parameter version of the Weyl algebra $\EuScript{W}(X,\omega)_{\hbar}$ will be convenient for extracting its classical limit as a Poisson algebra. This will also play an important role in constructing various vertex Poisson algebras in section \ref{sec:chiralPoi}.

\subsection{The Deligne category $\mathrm{Rep}(\mathrm{GL}_N)$}
In this section, we review the definition of the Deligne category $\mathrm{Rep}_{\mathrm{f}}(\mathrm{GL}_{N})$. We refer to \cite{comes2012deligne,ETINGOF2016473} for more details. In this section, we assume that $N$ is a complex number. The Deligne category is constructed from a ``skeleton category" $\mathrm{Rep}_0(\mathrm{GL}_{N})$, followed by the additive envelope and the Karoubi envelope. We describe $\mathrm{Rep}_0(\mathrm{GL}_{N})$ first.

Objects of $\mathrm{Rep}_0(\mathrm{GL}_{N})$ consist of (possibly empty) words $w$ of two symbols $\bb,\ww$. We denote by $\mathbb{1}$ the empty word.

Given two words $w,w'$, a $(w,w')$ diagram is a graph with two rows of vertices where we set $w$ as the first row of vertices and $w'$ as the second row of vertices. We require that each vertex is connected to exactly one edge. An edge is connected to both a $\bb$ and a $\ww$ vertex if and only if the two vertices are in the same row. We define the morphism space between two words $w,w'$ to be the $\C$-linear space with basis $\{(w,w')-\text{ diagrams}\}$. 
\begin{example}
	There are six $(\bb\ww\bb\ww,\bb\ww)$ diagrams
	
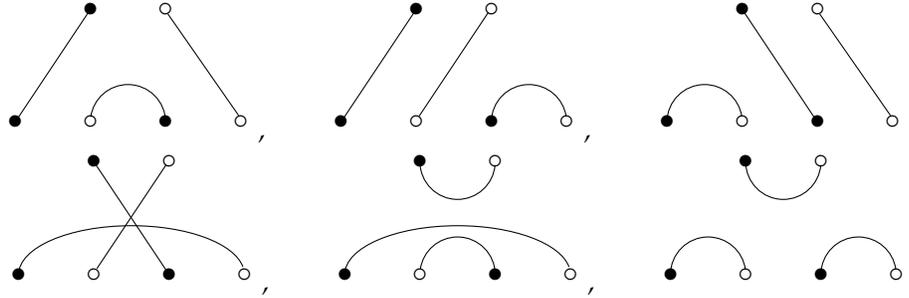
\begin{center}
		\begin{tikzpicture}
			\node (a1) at (1,0) {$\bb$};
			\node (a2) at (2,0) {$\ww$};
			\node (b1) at (0,-1.5) {$\bb$};
			\node (b2) at (1,-1.5) {$\ww$};
			\node (b3) at (2,-1.5) {$\bb$};
			\node (b4) at (3,-1.5) {$\ww$};
			\draw (1,0) -- (0,-1.5);
			\draw (2.03,-0.06) -- (2.97,-1.41);
			\draw (b3) arc (0:172:0.5);
		\end{tikzpicture},\quad\quad
		\begin{tikzpicture}
			\node (a1) at (1,0) {$\bb$};
			\node (a2) at (2,0) {$\ww$};
			\node (b1) at (0,-1.5) {$\bb$};
			\node (b2) at (1,-1.5) {$\ww$};
			\node (b3) at (2,-1.5) {$\bb$};
			\node (b4) at (3,-1.5) {$\ww$};
			\draw (1,0) -- (0,-1.5);
			\draw (1.96,-0.04) -- (1.04,-1.42);
			\draw (b3) arc (180:9:0.5);
		\end{tikzpicture},\quad\quad
		\begin{tikzpicture}
			\node (a1) at (1,0) {$\bb$};
			\node (a2) at (2,0) {$\ww$};
			\node (b1) at (0,-1.5) {$\bb$};
			\node (b2) at (1,-1.5) {$\ww$};
			\node (b3) at (2,-1.5) {$\bb$};
			\node (b4) at (3,-1.5) {$\ww$};
			\draw (1,0) -- (2,-1.5);
			\draw (2.03,-0.06) -- (2.97,-1.41);
			\draw (b1) arc (180:9:0.5);;
		\end{tikzpicture},
		\begin{tikzpicture}
			\node (a1) at (1,0) {$\bb$};
			\node (a2) at (2,0) {$\ww$};
			\node (b1) at (0,-1.5) {$\bb$};
			\node (b2) at (1,-1.5) {$\ww$};
			\node (b3) at (2,-1.5) {$\bb$};
			\node (b4) at (3,-1.5) {$\ww$};
			\draw (1,0) -- (2,-1.5);
			\draw (1.96,-0.04) -- (1.04,-1.42);
			\draw (b1) arc (180:9:1.5 and 0.65);
		\end{tikzpicture},\quad\quad
		\begin{tikzpicture}
			\node (a1) at (1,0) {$\bb$};
			\node (a2) at (2,0) {$\ww$};
			\node (b1) at (0,-1.5) {$\bb$};
			\node (b2) at (1,-1.5) {$\ww$};
			\node (b3) at (2,-1.5) {$\bb$};
			\node (b4) at (3,-1.5) {$\ww$};
			\draw (a1) arc (180:355:0.5);
			\draw (b3) arc (0:172:0.5);
			\draw (b1) arc (180:9:1.5 and 0.65);
		\end{tikzpicture},\quad\quad
		\begin{tikzpicture}
			\node (a1) at (1,0) {$\bb$};
			\node (a2) at (2,0) {$\ww$};
			\node (b1) at (0,-1.5) {$\bb$};
			\node (b2) at (1,-1.5) {$\ww$};
			\node (b3) at (2,-1.5) {$\bb$};
			\node (b4) at (3,-1.5) {$\ww$};
			\draw (a1) arc (180:355:0.5);
			\draw (b1) arc (180:8:0.5);
			\draw (b3) arc (180:8:0.5);
		\end{tikzpicture}
	\end{center}
	
\end{example}

Given a $(w, w')$-diagram $X$ and a $(w', w'')$-diagram $Y$, we first construct a graph by stacking $Y$ on top of $X$. We then define $Y \cdot X$ as the graph obtained by removing the middle row of vertices. Consequently, $Y \cdot X$ forms a $(w, w'')$-diagram. We denote by $l(X, Y)$ the number of loops that were removed in this process. The composition of morphisms in $\mathrm{Rep}_0(\mathrm{GL}_{N})$ is then defined by: 
\begin{equation*}
	\begin{aligned}
		\Hom(w',w'') \times \Hom(w,w') & \to \Hom(w,w'')\\
		(Y,X) &\mapsto N^{l(X,Y)}Y\cdot X \,.
	\end{aligned}
\end{equation*}
\begin{example}
	Let $X = 		\;	\begin{tikzpicture}[scale=0.4]
		\useasboundingbox (0,0) rectangle (1.5,-1.5);
		\node (a1) at (0,-0.5) {$\bb$};
		\node (a2) at (1,-0.5) {$\ww$};
		\node (b1) at (0,-2) {$\bb$};
		\node (b2) at (1,-2) {$\ww$}; 
		\draw (a1) arc (180:350:0.5);
		\draw  (b1) arc (180:25:0.5);
	\end{tikzpicture}$, we compute $X^2$:
	\begin{center}
		\begin{tikzpicture}[scale=0.7]
			\node (a1) at (0,0) {$\bb$};
			\node (a2) at (1,0) {$\ww$};
			\node (b1) at (0,-1.3) {$\bb$};
			\node (b2) at (1,-1.3) {$\ww$}; 
			\node (c1) at (0,-2.6) {$\bb$};
			\node (c2) at (1,-2.6) {$\ww$};
			\draw (a1) arc (180:355:0.5);
			\draw  (b1) arc (180:10:0.5);
			\draw (b1) arc (180:355:0.5);
			\draw  (c1) arc (180:10:0.5);
			\node at (2,-1.3) {$= N \times$};
			\node (d1) at (3,-0.65) {$\bb$};
			\node (d2) at (4,-0.65) {$\ww$};
			\node (e1) at (3,-1.95) {$\bb$};
			\node (e2) at (4,-1.95) {$\ww$}; 
			\draw (d1) arc (180:355:0.5);
			\draw  (e1) arc (180:8:0.5);
		\end{tikzpicture}
	\end{center}
	When $N \neq 0 $, we can set $e = \frac{1}{N}\;	\begin{tikzpicture}[scale=0.4]
		\useasboundingbox (0,0) rectangle (1.5,-1.5);
		\node (a1) at (0,-0.5) {$\bb$};
		\node (a2) at (1,-0.5) {$\ww$};
		\node (b1) at (0,-2) {$\bb$};
		\node (b2) at (1,-2) {$\ww$}; 
		\draw (a1) arc (180:350:0.5);
		\draw  (b1) arc (180:25:0.5);
	\end{tikzpicture}$, and we have $e^2 = e$.
	
\end{example}

\begin{remark}
	It is easy to check that two words $w,w'$ are isomorphic if and only if they have the same number of $\bb$ and $\ww$. Therefore, each object in $\mathrm{Rep}_0(\mathrm{GL}_{N})$ is isomorphic to an object of the form
	\begin{equation*}
		[r,s] := \underbrace{\bb,\dots,\bb}_{r}\underbrace{\ww,\dots,\ww}_{s}\,.
	\end{equation*}
\end{remark}

\begin{remark}
	$\mathrm{End}([r,s]) = B_{r,s}(N)$ is the walled Brauer algebra. Also $\mathrm{End}([r,0]) = \C[S_r]$.
\end{remark}

Now we equip $\mathrm{Rep}_0(\mathrm{GL}_{N})$ with the structure of a rigid symmetric monoidal category. The tensor functor $-\otimes-$ is defined as follows. On objects, $w_1\otimes w_2 = w_1w_2$ is simply concatenation of words. The tensor product of morphisms $X_1\otimes X_2$ is simply the diagram obtained by placing the diagram $X_1$ to the left of the diagram $X_2$. The braiding $\sigma_{w_1,w_2}:w_1\otimes w_2\to w_2\otimes w_1$ is the $(w_1w_2,w_2w_1)$ diagram that connects each letter of $w_i$ in the first row to the same letter of $w_i$ in the second row.

For every object, we can also construct its dual. For any word $w$, we define $w^*$ as the word obtained from $w$ by replacing all $\bb$ with $\ww$ and vice versa. We define the morphism $\mathrm{ev}_{w}: w^*\otimes w \to \mathbb{1}$ as the $(w^*w,\mathbb{1})$ diagram that connect the $i$-th letter in $w^*$ with the $i$-th letter in $w$. Similarly, we define $\mathrm{coev}_{w}:\mathbb{1} \to w\otimes w^*$ the $(\mathbb{1},ww^*)$ diagram that connects the $i$-th letter in $w$ with the $i$-th letter in $w^*$.

It is easy to check that the above definitions give $\mathrm{Rep}_0(\mathrm{GL}_{N})$ the structure of a rigid symmetric monoidal category.

Finally, we can define the Deligne category $\mathrm{Rep}_{\mathrm{f}}(\mathrm{GL}_{N})$
\begin{definition}
	For $N \in \C$, the Deligne category $\mathrm{Rep}_{\mathrm{f}}(\mathrm{GL}_{N})$ is the Karoubi envelope of the additive envelope of the category $\mathrm{Rep}_0(\mathrm{GL}_{N})$. The tensor structure $\mathrm{Rep}_0(\mathrm{GL}_{N})$ extends to $\mathrm{Rep}_{\mathrm{f}}(\mathrm{GL}_{N})$ in the natrual way.
\end{definition}

The most important properties of $\mathrm{Rep}_{\mathrm{f}}(\mathrm{GL}_{N})$ are listed below:
\begin{prop}(\cite{deligne1982tannakian,deligne2007categories})\label{prop:Del_fun}
	(i) For $N \notin \Z$, the category $\mathrm{Rep}_{\mathrm{f}}(\mathrm{GL}_{N})$ is a semisimple abelian category.
	
	(ii) If $N \in \Z$ and $p,q$ are nonnegative integers with $p-q = N$, the category $\mathrm{Rep}_{\mathrm{f}}(\mathrm{GL}_{N})$ (which is not abelian) admits a full non-faithful symmetric monoidal functor $$\mathrm{Rep}_{\mathrm{f}}(\mathrm{GL}_{N}) \to \mathbf{Rep}_{\mathrm{f}}(\mathrm{GL}_{p|q})\,.$$ Here $\mathbf{Rep}_{\mathrm{f}}(\mathrm{GL}_{p|q})$ is the category of finite dimensional representation of the supergroup $GL_{p|q}$. The functor sends $[1,0]$ to the supervector space $\C^{p|q}$.
	
	(iii) The category $\mathrm{Rep}_{\mathrm{f}}(\mathrm{GL}_{N})$ has the following universal property: if $\EuScript{D}$ is a rigid symmetric monoidal category then isomorphism classes of (strong) symmetric tensor functors $\mathrm{Rep}_{\mathrm{f}}(\mathrm{GL}_{N}) \to \EuScript{D}$ are in bijection with isomorphism classes of objects $X$ in $\EuScript{D}$ of dimension $N$, via $F \to F([1,0])$.
\end{prop}

In particular, we have a canonical symmetric monoidal functor $\mathrm{Rep}_{\mathrm{f}}(\mathrm{GL}_{n}) \to \mathbf{Rep}_{\mathrm{f}}(\mathrm{GL}_{n})$ for $n\in\Z_{\geq 0}$, which sends $[r,s] \to (\C^{n})^{\otimes r}\otimes (\C^{n*})^{\otimes s}$.

\subsection{Variants of Deligne Category}
\label{sec:var_Deligne}
In this paper, we work with certain variants of the Deligne category, adapted to the construction of the vertex algebra. First, most non-trivial vertex algebras in $\mathrm{Vect}_k$ are infinite-dimensional. The state–field correspondence takes an element $v$ of the vector space and associates it with a formal Laurent series whose coefficients $v_{(n)}$ act as operators on the vector space. The operators $v_{(n)}$ for $n<0$ often generate infinitely many linearly independent vectors, except when they act trivially, as is the case for commutative vertex algebras. When a vertex algebra is infinite-dimensional, it can often be successively approximated by finite-dimensional subspaces, obtained by restricting the number of fields acting on the vacuum vector. Therefore, most vertex algebras naturally live in the ind-completion of the category of finite-dimensional vector spaces. Consequently, we are primarily interested in the ind-completion of the Deligne category $\mathrm{Rep}_{\mathrm{f}}(\mathrm{GL}_{N})$, which we denoted $\mathrm{Rep}(\mathrm{GL}_{N}) = \mathrm{Ind}(\mathrm{Rep}_{\mathrm{f}}(\mathrm{GL}_{N}))$. The tensor functor $\otimes$ on $\mathrm{Rep}(\mathrm{GL}_{N})$ is defined so that it passes through filtered colimits, i.e.
\begin{equation*}
(\underset{i\in I}{\mathrm{colim}}X_i)\otimes(\underset{j\in J}{\mathrm{colim}}X_i) = \underset{i,j\in I\times J}{\mathrm{colim}}(X_i\otimes Y_j)\,.
\end{equation*}

We would also like to work with a category where $N$ is not a fixed number, but an indeterminate parameter. 
\begin{definition}
	We define $\mathrm{Rep}_0(\mathrm{GL}_{[N]})$ as a $\C[N]$-linear category. Its objects are still given by words $w$ of symbols $\bb,\ww$. The space of morphisms between two words $w,w'$ is the space of $\C[N]$-linear span of $(w,w')$ diagrams. The composition of morphisms is defined as before, except that a loop now contributes a factor of $N \in \C[N]$.
	
	We define $\mathrm{Rep}_{\mathrm{f}}(\mathrm{GL}_{[N]})$ as the Karoubi envelope of additive envelope of $\mathrm{Rep}_0(\mathrm{GL}_{[N]})$. Then $\mathrm{Rep}(\mathrm{GL}_{[N]})$ is defined as its ind completion $\mathrm{Ind}(\mathrm{Rep}_{\mathrm{f}}(\mathrm{GL}_{[N]}))$. 
\end{definition}
By construction, specializing $N$ to a complex number $n\in \C$ gives us a functor
$$F_{N= n}: \mathrm{Rep}_{\mathrm{f}}(\mathrm{GL}_{[N]})\to \mathrm{Rep}_{\mathrm{f}}(\mathrm{GL}_{n})\,,$$
as well as its ind completion $F_{N= n}: \mathrm{Rep}(\mathrm{GL}_{[N]})\to \mathrm{Rep}(\mathrm{GL}_{n})$.
\begin{remark}
	Although the construction for $\mathrm{Rep}_{\mathrm{f}}(\mathrm{GL}_{[N]})$ is basically the same as $\mathrm{Rep}_{\mathrm{f}}(\mathrm{GL}_{n})$, $\mathrm{Rep}_{\mathrm{f}}(\mathrm{GL}_{[N]})$ has far fewer objects than $\mathrm{Rep}_{\mathrm{f}}(\mathrm{GL}_{n})$. As a simple example, we have seen that $e = 1 - \frac{1}{n}\;	\begin{tikzpicture}[scale=0.4]
		\useasboundingbox (0,0) rectangle (1.5,-1.5);
		\node (a1) at (0,-0.5) {$\bb$};
		\node (a2) at (1,-0.5) {$\ww$};
		\node (b1) at (0,-2) {$\bb$};
		\node (b2) at (1,-2) {$\ww$}; 
		\draw (a1) arc (180:350:0.5);
		\draw  (b1) arc (180:25:0.5);
	\end{tikzpicture}$ is an idempotent for $n \neq 0$. Therefore $\left( \bb\ww,1 - \frac{1}{n}\;	\begin{tikzpicture}[scale=0.4]
		\useasboundingbox (0,0) rectangle (1.5,-1.5);
		\node (a1) at (0,-0.5) {$\bb$};
		\node (a2) at (1,-0.5) {$\ww$};
		\node (b1) at (0,-2) {$\bb$};
		\node (b2) at (1,-2) {$\ww$}; 
		\draw (a1) arc (180:350:0.5);
		\draw  (b1) arc (180:25:0.5);
	\end{tikzpicture}\right) $ is an object in $\mathrm{Rep}_{\mathrm{f}}(\mathrm{GL}_{n})$. However, the idempotent $1 - \frac{1}{N}\;	\begin{tikzpicture}[scale=0.4]
	\useasboundingbox (0,0) rectangle (1.5,-1.5);
	\node (a1) at (0,-0.5) {$\bb$};
	\node (a2) at (1,-0.5) {$\ww$};
	\node (b1) at (0,-2) {$\bb$};
	\node (b2) at (1,-2) {$\ww$}; 
	\draw (a1) arc (180:350:0.5);
	\draw  (b1) arc (180:25:0.5);
	\end{tikzpicture}$ does not exist over the ring $\C[N]$, hence we do not have a corresponding object in $\mathrm{Rep}_{\mathrm{f}}(\mathrm{GL}_{[N]})$.

	Nevertheless, $\mathrm{Rep}(\mathrm{GL}_{[N]})$ is still compactly generated by a pseudo tensor category by construction, which is sufficient for our later construction of the vertex algebra.
\end{remark}

\begin{remark}
	Following the above remark, we can also consider the version of the Deligne category over the field $\overline{\C(N)}$, as studied in \cite{kalinov2023deformed,etingof2023new}, which behaves analogously to $\mathrm{Rep}(\mathrm{GL}_{n})$. However, working within $\mathrm{Rep}(\mathrm{GL}_{[N]})$ has the advantage of naturally yielding polynomial coefficients in $N$ for the vertex algebra we will study. Otherwise extra effort is needed to prove this fact if we start with the field $\overline{\C(N)}$. 
\end{remark}

We also define the $\Z$ graded version of the Deligne category $\mathrm{Rep}^{\Z}(\mathrm{GL}_{N})$ (or $\mathrm{Rep}^{\Z}(\mathrm{GL}_{[N]})$). Objects of $\mathrm{Rep}^{\Z}(\mathrm{GL}_{N})$ are given by direct sum of objects in Deligne category $X = \bigoplus_{n\in \Z}X_n$. The tensor product is defined so that $(X\otimes Y)_n = \bigoplus_{p+q =n}X_p\otimes Y_q$. The braiding follows the Koszul sign rule, where $\sigma_{w_1,w_2}$ picks up a sign $(-1)^{pq}$ for $w_1$ in degree $p$ and $w_2$ in degree $q$.

It will be useful to consider a ‘multi-Deligne category’ $\mathrm{Rep}(\mathrm{GL}_{N_1,N_2})$. When $N_1$ and $N_2$ are not integers, each $\mathrm{Rep}(\mathrm{GL}_{N_i})$ is an abelian category, and we define $\mathrm{Rep}(\mathrm{GL}_{N_1,N_2})$ as their Deligne tensor product: $\mathrm{Rep}(\mathrm{GL}_{N_1})\boxtimes\mathrm{Rep}(\mathrm{GL}_{N_2})$. However, the categories $\mathrm{Rep}(\mathrm{GL}_{N_i})$ or $\mathrm{Rep}(\mathrm{GL}_{[N_i]})$ will not always be abelian, and therefore we will adopt a different definition.

\begin{definition}
		We first define $\mathrm{Rep}_0(\mathrm{GL}_{[N_1,N_2]})$ as a $\C[N_1,N_2]$-linear category. An object is given by a pair of two words $(w_1,w_2)$ labeled by $1$ and $2$ respectively. The word labeled by $1$ consist of symbol $\overset{1}{\bb},\overset{1}{\ww}$ and the word labeled by $2$ consist of symbol $\overset{2}{\bb},\overset{2}{\ww}$.
		
		A diagram between two objects $w_1w_2,w_1'w_2'$ is simply a pair of a $(w_1,w_1')$ diagram and a $(w_2,w_2')$ diagram. The space of morphisms between two objects $w_1w_2,w_1'w_2'$ is the space of $\C[N_1,N_2]$-linear span of $(w,w')$ diagrams. The composition of morphisms is defined separately for words labeled by $1$ and $2$. A loop labeled by $1$ contribute a factor of $N_1 \in \C[N_1,N_2]$, and a loop labeled by $2$ contribute a factor of $N_2 \in \C[N_1,N_2]$.
	
	We define $\mathrm{Rep}_{\mathrm{f}}(\mathrm{GL}_{[N_1,N_2]})$ as the Karoubi envelope of additive envelope of $\mathrm{Rep}_0(\mathrm{GL}_{[N_1,N_2]})$. We also denote its ind completion $\mathrm{Ind}(\mathrm{Rep}_{\mathrm{f}}(\mathrm{GL}_{[N]}))$ by $\mathrm{Rep}(\mathrm{GL}_{[N_1,N_2]})$. 
\end{definition}
The diagonal embedding of the algebraic group $\mathrm{GL}_{n_1}\times \mathrm{GL}_{n_2} \to \mathrm{GL}_{n = n_1+n_2}$ induces a functor of the corresponding representation categories. For the Deligne category, we have a similar functor $\Delta: \mathrm{Rep}(\mathrm{GL}_{N}) \to \mathrm{Rep}(\mathrm{GL}_{N_1,N_2})$, with $N= N_1 + N_2$, as well as a functor $\Delta: \mathrm{Rep}(\mathrm{GL}_{[N]}) \to \mathrm{Rep}(\mathrm{GL}_{[N_1,N_2]})$. 
It is easy to check that $\overset{1}{\bb}+ \overset{2}{\bb}$ has dimension $N$. Proposition \ref{prop:Del_fun} then implies the existence of the strong symmetric monoidal functor $\Delta: \mathrm{Rep}(\mathrm{GL}_{N}) \to \mathrm{Rep}(\mathrm{GL}_{N_1,N_2})$. We also provide an explicit construction, which also works for the power series case.

 On objects, $\Delta$ maps empty word to empty word. For a non-empty word $w$, we first consider all possible $\{1,2\}$ labeling of the word. Then we reorder every labeled word into the standard form $w_1w_2$ and sum them all. e.g. it maps $\Delta \ww = \overset{1}{\ww}+ \overset{2}{\ww}$, $\Delta \bb = \overset{1}{\bb}+ \overset{2}{\bb}$ and $\Delta (\ww\ww )= \overset{1}{\ww}\overset{1}{\ww} +\overset{1}{\ww}\overset{2}{\ww} + \overset{2}{\ww}\overset{1}{\ww}+ \overset{2}{\ww}\overset{2}{\ww} =\overset{1}{\ww}\overset{1}{\ww} +2\overset{1}{\ww}\overset{2}{\ww} + \overset{2}{\ww}\overset{2}{\ww}$. For a $(w,w')$ diagram in $\mathrm{Hom}(w,w')$, we first consider all possible $\{1,2\}$ labeling of the edges of the diagram. Note that a labeling of edges also induces a labeling of the vertices, thus correspond to a morphism in $\mathrm{Hom}(\Delta w,\Delta w')$. Then we sum over all of them. We can check that the compatibility of the morphism with the composition requires that $N$ is sent to $N_1 + N_2$:
	\begin{center}
	\begin{tikzpicture}[scale=0.8]
		\node at (-1,0) {$N = $};
		\node (b1) at (0,0) {$\bb$};
		\node (b2) at (1,0) {$\ww$}; 
		\draw  (b1) arc (180:10:0.5);
		\draw (b1) arc (180:355:0.5);
		\node at (1.6,0) {$\to$};
		\node (c1) at (2.2,0) {$\bb$};
		\node (c2) at (3.2,0) {$\ww$}; 
		\node at (2.1,0.3) {{\tiny 1}};
		\node at (3.3,0.3) {{\tiny 1}}; 
		\draw  (c1) arc (180:10:0.5);
		\draw (c1) arc (180:355:0.5);
		\node at (3.8,0) {$+$};
		\node (d1) at (4.4,0) {$\bb$};
		\node (d2) at (5.4,0) {$\ww$}; 
		\node at (4.3,0.3) {{\tiny 2}};
		\node at (5.5,0.3) {{\tiny 2}}; 
		\draw  (d1) arc (180:10:0.5);
		\draw (d1) arc (180:355:0.5);
		\node at (7,0) {$= N_1 + N_2\,.$};
	\end{tikzpicture}
\end{center}

\section{Vertex algebra in symmetric monoidal category}
\label{sec:vertex_general}
\subsection{Basic definitions}
\label{sec:def}
First, we recall the definition of a vertex algebra. A vertex algebra consists of a vector space $V$ together with the following data:
\begin{enumerate}
	\item The vacuum vector: $\vac \in V$.
	\item The translation map: $T: V \to V$.
	\item An infinite collection of bilinear maps: $\cdot_{n}: V\otimes V \to V$ for $n \in \Z$.  We require that for any $u,v \in V$, there exists an integer $K$ such that $u\cdot_n v = 0$ for any $n \geq K$. We usually collect these maps into a power series and write $$Y(u,z)v = \sum_{n\in \Z}u\cdot_{n}v\; z^{-n-1}\,.$$
\end{enumerate}
These data satisfy the following axioms
\begin{enumerate}
	\item Vacuum. $Y(\vac,z)u = u$, and $Y(u,z)\vac \in u + zV[[z]]$ for any $u \in V$.
	\item Translation. $T\vac = 0$. Furthermore, $[T,Y(u,z)] = \pa_z Y(u,z)$
	\item Locality. For any $u,v \in V$, there exists an integer $K$ such that
	\begin{equation*}
		(z - w)^K[Y(u,z),Y(v,w)] = 0\,.
	\end{equation*}
\end{enumerate}

Most part of the above definition generalize easily to any braided monoidal category. The aspect requiring particular attention arises in defining the product $Y(-,z)$ and the locality axiom, where the bound $K$ depend on the choice of elements in $V$. In the categorical setting, a possible replacement for an 'element in an object' could be a morphism mapping to the object. Rather than allowing arbitrary morphisms, we should restrict to morphisms from objects that can be characterized as 'finite-dimensional'\footnote{The definition in \cite{LatyntsevPhD} only considers morphisms from $\mathbb{1}$ in the locality axiom, which we believe is not strong enough.}. This leads to the notion of compactness. Furthermore, the vertex algebra $V$, which is typically infinite-dimensional, need to be approximated by 'finite-dimensional' objects. Thus, in our definition, we consider symmetric monoidal category $\EuScript{C}$ that is also compactly generated. To establish a reasonable framework, we impose the following conditions in this section and the following ones.
\begin{assumption}\label{assum:vertex_cat}
	\begin{enumerate}
		\item\label{con:t_presen} ($\otimes$-presentable) $\EuScript{C}$ is compactly generated, which is also called locally presentable, and the symmetric monoidal structure distributes over colimits. In other words, $\EuScript{C} \cong \mathrm{Ind}(\EuScript{C}_0)$ for some symmetric monoidal category $\EuScript{C}_0$ and the tensor product satisfies: $$(\mathrm{colim}X_\alpha)\otimes(\mathrm{colim}Y_\beta) = \mathrm{colim}(X_\alpha\otimes Y_{\beta}).$$
		\item $\EuScript{C}_0$ is a pseudo tensor category, i.e. a rigid symmetric monoidal category and is idempotent complete.
		\item $\Q \subset \mathrm{Hom}_{\EuScript{C}}(\mathbb{1},\mathbb{1})$.
	\end{enumerate}
\end{assumption}

As we will discuss later, condition $(2)$ and $(3)$ are not really necessary for the definition of vertex algebra. However, we have seen in section \ref{sec:cat} and \ref{sec:Weyl} that these conditions guarantee the nice properties that compact objects are closed under the monoidal structure, and we can define Weyl algebra in this setting. This will allow us to construct interesting example of vertex algebra object in these categories.

We define vertex algebra in $\EuScript{C}$ as follows:
\begin{definition}
	A vertex algebra in $\EuScript{C}$ consist of the data of an object $V$ in $\EuScript{C}$ together with morphisms
	\begin{enumerate}
		\item A vacuum vector: $\vac \in \mathrm{Hom}_{\EuScript{C}}(\mathbb{1},V)$.
		\item A translation map: $T \in \mathrm{Hom}_{\EuScript{C}}(V,V)$.
		\item An infinite collection of maps: $\cdot_{n} \in \mathrm{Hom}_{\EuScript{C}}(V\otimes V,V)$ for $n \in \Z$. We usually collect these maps together and get a map $Y(z) = \sum_{n\in \Z}\cdot_{n}\; z^{-n-1}\in \mathrm{Hom}(V\otimes V,V)[[z,z^{-1}]]$. We also require that for any compact objects $X,Y$ in $\EuScript{C}$ and $\alpha \in \Hom_{\EuScript{C}}(X,V)$, $\beta \in \Hom_{\EuScript{C}}(Y,V)$, there exists an integer $K$ such that
		\begin{equation}\label{cond_finite}
			\cdot_{n}\circ(\alpha\otimes \beta) = 0
		\end{equation}
		for any $n \geq K$.
	\end{enumerate}
	
	These morphisms are required to satisfy the following axioms
	\begin{enumerate}
		\item Vacuum. $Y(z)\circ(\vac\otimes \mathrm{id}_V) = l_V$. Furthermore, $Y(z)\circ (\mathrm{id}_V\otimes \vac) \in \mathrm{Hom}_{\EuScript{C}}(V\otimes \mathbb{1},V)[[z]]$, so that $Y(z)\circ (\mathrm{id}_V\otimes \vac)|_{z = 0} = r_V$.
		\item Translation. $T\vac = 0$. Furthermore, $T\circ Y(z) - Y(z)\circ(\mathrm{id}_V\otimes T) = \pa_z Y(z)$
		\item Locality. For any compact objects $X,Y$ and $\alpha \in \Hom_{\EuScript{C}}(X,V)$, $\beta \in \Hom_{\EuScript{C}}(Y,V)$, we can find an integer $K$ such that \footnote{For simplicity, the associative isomorphism $a$ in our monoidal category $\EuScript{C}$ is omitted in the formula, as it can be easily recovered and is not essential for the discussion in this paper.}
		\begin{equation*}
			\begin{aligned}
				(z - w)^K\big(&Y(w)\circ(\mathrm{id}_V\otimes Y(z))\circ(\alpha \otimes \beta\otimes \mathrm{id}_V) \\
				&-  Y(z)\circ(\mathrm{id}_V\otimes Y(w))\circ(\sigma\otimes\mathrm{id}_V)\circ(\alpha \otimes \beta \otimes \mathrm{id}_V)\big) = 0\,.
			\end{aligned}
		\end{equation*}
	\end{enumerate}
\end{definition}

\begin{remark}
	By our assumption on $\EuScript{C}$ and Corollary \ref{cor:comp_tensor}, the tensor product of two compact objects is still compact. We can show that the collection of maps $Y(z)$ satisfying condition \eqref{cond_finite} is equivalent to say that it is an element in
	\begin{equation*}
		\mathrm{Hom}_{\EuScript{C}}(V\otimes V,V((z)))\,.
	\end{equation*}
	
	To see this, we let $V = \underset{i\in I}{\mathrm{colim}}V_i$ with $V_i$ compact. To simplify the discussion, we focus on the singular part of $Y(z)$. By definition, we have
	\begin{equation*}
		\mathrm{Hom}_{\EuScript{C}}(V\otimes V,Vz^{-1}[z^{-1}]) = \lim_{(i,j)\in I\times I}\underset{k\in \Z}{\mathrm{colim}}\, \mathrm{Hom}_{\EuScript{C}}(V_i\otimes V_j,\oplus_{i=0}^kVz^{-i-1})\,.
	\end{equation*}
	This gives us, for each $i,j\in I\times I$, an integer $N_{i,j}$ and maps $\mu_{i,j;n}:V_{i}\otimes V_j \to Vz^{-n-1}$ for $n\leq N_{i,j}$. These maps are compatible in the obvious way. So we get a series of maps $\mu_n$, and on each $V_i\otimes V_j$, $\mu_n$ vanish for $n > N_{i,j}$. This is equivalent to the condition we give in our definition, because any map $X\to V$ from a compact object $X$ must factor through a map $X \to V_i$ for some $i$. Our definition has the advantage of being independent of the presentation of $V$ as a colimit.
\end{remark}
\begin{remark}\label{rmk:braided}
	The definition above also works for a braided monoidal category $\EuScript{C}$ that only satisfies the $\otimes$-presentable condition \eqref{con:t_presen}. However, most examples considered in this paper require the existence of symmetric powers in the category, defined via the action of the symmetric group $S_n$ on $X^{\otimes n}$. For a braided monoidal category that is not symmetric, only a braid group action exists, and it is unclear whether the examples considered in this paper can be extended to the braided monoidal setting.
\end{remark}

\begin{remark}
Compact objects in $\mathrm{Vect}_k$ are finite-dimensional vector spaces. Thus, for $\EuScript{C} = \mathrm{Vect}_k$, the above definition recovers the usual notion of a vertex algebra. Similarly, we may consider the category of $\mathbb{Z}_2$-graded super vector spaces $\mathrm{sVect}_k$ or $\mathbb{Z}$-graded super vector spaces $\mathrm{sVect}^{\mathbb{Z}}_k$. In these categories, the braiding morphism $\sigma$ acquires a Koszul sign determined by the grading. As a result, the locality axiom (and hence the Borcherds identity) for vertex algebras in these categories also include a grading-dependent sign. This recovers the standard notions of $\mathbb{Z}_2$- or $\mathbb{Z}$-graded vertex superalgebras. One often considers the category of $\mathbb{Z}$-graded vector spaces $\mathrm{Vect}^{\mathbb{Z}}_k$, where the braiding does not involve a sign. In this case, one recovers the usual notion of a vertex algebra, equipped with an additional $\mathbb{Z}$-grading.
\end{remark}

The following statement is a simple corollary of the vacuum and translation axioms.
\begin{lemma}\label{lem_vac}
	We have the following identity
	\begin{equation*}
		Y(z)\circ (\mathrm{id}_V\otimes \vac) = e^{zT}\circ r_V
	\end{equation*}
	in $\mathrm{Hom}_{\EuScript{C}}(V\otimes \mathbb{1},V)[[z]]$.
\end{lemma}
\begin{proof}
	By the translation axiom, we have 
	\begin{equation*}
		\partial_zY(z)\circ (\mathrm{id}_V\otimes \vac)  = T\circ Y(z)\circ (\mathrm{id}_V\otimes \vac)\,.
	\end{equation*}
	This implies, by induction, that 
	\begin{equation*}
		\partial_z^nY(z)\circ (\mathrm{id}_V\otimes \vac)  = T^n\circ Y(z)\circ (\mathrm{id}_V\otimes \vac)\,.
	\end{equation*}
	By the vacuum axiom, we have 
	\begin{equation*}
		\partial_z^nY(z)\circ (\mathrm{id}_V\otimes \vac)|_{z = 0}  = T^n\circ r_V\,.
	\end{equation*}
	Since $Y(z)\circ (\mathrm{id}_c\otimes \vac)\in \mathrm{Hom}_{\EuScript{C}}(V,V)[[z]] $, we have 
	\begin{equation*}
		Y(z)\circ (\mathrm{id}_V\otimes \vac) = \sum_{n = 0}^\infty\frac{z^n}{n!}\partial_z^nY(0)\circ (\mathrm{id}_V\otimes \vac) = e^{zT}\circ r_V	\,.
	\end{equation*}
\end{proof}

We also define morphism between two vertex algebras in $\EuScript{C}$.
\begin{definition}
	Let $(V,\vac_V,T_V,Y_V)$ and $(W,\vac_W,T_W,Y_W)$ be two vertex algebras in $\EuScript{C}$. A morphism $\phi: V \to W$ is called a morphism of vertex algebras if $\phi\circ\vac_V=\vac_W$, $\phi\circ T_V = T_W\circ\phi $ and $\phi \circ Y_V = Y_W\circ(\phi\otimes\phi)$.
\end{definition}

\subsection{Associativity}
	There are several equivalent formulations of the locality axiom for a vertex algebra in $\mathrm{Vect}_k$. An often used one is the so called Borcherds identity. In this section, we study associativity property of vertex algebra in $\EuScript{C}$ and different formulation of the locality axiom. Most part of this section follows from the same results as ordinary vertex algebra \cite{frenkel2004vertex,kac1998vertex}. 
\begin{lemma}\label{lem_trans}
	We have the following identity
	\begin{equation*}
		e^{wT}\circ Y(z)\circ(\mathrm{id}_V\otimes e^{-wT}) = Y(z+w)
	\end{equation*}
	in $\mathrm{Hom}_{\EuScript{C}}(V\otimes V,V)[[z^{\pm 1},w^{\pm 1}]]$.
\end{lemma}
\begin{proof}
	By the translation axiom, we have 
	\begin{equation*}
		e^{wT}\circ Y(z)\circ(\mathrm{id}_V\otimes e^{-wT}) = \sum_{n=0}^{\infty}\frac{w^n}{n!}\partial_z^nY(z) = Y(z+w)\,.
	\end{equation*}
\end{proof}

\begin{prop}
	(skew-symmetry) For any compact objects $X,Y$ in $\EuScript{C}$ and morphisms $\alpha \in \Hom_{\EuScript{C}}(X,V)$, $\beta \in \Hom_{\EuScript{C}}(Y,V)$, the following identity holds
	\begin{equation}\label{eqn_skew}
		Y(z)\circ (\alpha\otimes \beta) = e^{zT}\circ Y(-z)\circ\sigma\circ (\alpha\otimes \beta)
	\end{equation}
	in $\mathrm{Hom}_{\EuScript{C}}(X\otimes Y,V)((z))$
\end{prop}
\begin{proof}
	We apply the locality axiom and find
	\begin{equation*}
		\begin{aligned}
			(z - w)^K\Big(&Y(w)\circ(\mathrm{id}_V\otimes Y(z))\circ(\alpha\otimes \beta\otimes \vac) \\
			&-  Y(z)\circ(\mathrm{id}_V\otimes Y(w))\circ(\sigma\otimes\mathrm{id}_V)\circ(\alpha\otimes \beta \otimes \vac)\Big) = 0
		\end{aligned}
	\end{equation*}
	for large enough $K$. Using Lemma \ref{lem_vac}, we find
	\begin{equation*}
		(z - w)^K\Big(Y(w)\circ(\mathrm{id}_V\otimes e^{zT}) -  Y(z)\circ(\mathrm{id}_V\otimes e^{wT})\circ \sigma\Big)\circ(\alpha\otimes \beta) = 0	\,.
	\end{equation*}
	Then we apply Lemma \ref{lem_trans}, which gives us
	\begin{equation*}
		(z - w)^K Y(w)\circ(\mathrm{id}_V\otimes e^{zT})\circ(\alpha\otimes \beta) =   (z - w)^Ke^{wT}\circ Y(z - w)\circ \sigma\circ(\alpha\otimes \beta) = 0	\,.
	\end{equation*}
	We can choose $K$ large enough so that the right hand side does not contain any negative power of $(z-w)$. Then we set $z = 0$ and the identity becomes $w^KY(w)\circ(\alpha\otimes \beta) =   w^Ke^{wT}\circ Y(- w)\circ \sigma\circ(\alpha\otimes \beta) = 0	$. We divide both sides by $w^K $, which gives the formula \ref{eqn_skew}.
\end{proof}

\begin{theorem}\label{thm_ass}
	For any compact objects $X_i$, $i = 1,2,3$ and morphisms $\alpha\in \mathrm{Hom}_{\EuScript{C}}(X_1,V)$, $\beta\in \mathrm{Hom}_{\EuScript{C}}(X_2,V)$, $\gamma \in \mathrm{Hom}_{\EuScript{C}}(X_3,V)$ the three expansions
	\begin{equation*}
		\begin{aligned}
			&Y(z)\circ(\mathrm{id}_V \otimes Y(w))\circ(\alpha\otimes\beta\otimes \gamma) \in \mathrm{Hom}_{\EuScript{C}}(\mathbf{X},V)((z))((w))\,,\\
			& Y(w)\circ(\mathrm{id}_V \otimes Y(z))\circ(\sigma\otimes \mathrm{id}_V)\circ(\alpha\otimes\beta\otimes \gamma) \in \mathrm{Hom}_{\EuScript{C}}(\mathbf{X},V)((w))((z))\,,\\
			&Y(w)\circ(Y(z-w)\otimes\mathrm{id}_V)\circ(\alpha\otimes\beta\otimes\gamma) \in \mathrm{Hom}_{\EuScript{C}}(\mathbf{X},V)((w))((z-w))\,,
		\end{aligned}
	\end{equation*}
	where we denote $\mathbf{X} = X_1\otimes X_2\otimes X_3$, are the expansion of the same element of 
	\begin{equation*}
		\mathrm{Hom}_{\EuScript{C}}(\mathbf{X},V)[[z,w]][z^{-1},w^{-1},(z-w)^{-1}]\,.
	\end{equation*}
\end{theorem}
\begin{proof}
	By the locality axiom, $Y(z)\circ(\mathrm{id}_V \otimes Y(w))\circ(\alpha\otimes\beta\otimes \gamma) $ and $Y(w)\circ(\mathrm{id}_V \otimes Y(z))\circ(\sigma\otimes \mathrm{id}_V)\circ(\alpha\otimes\beta\otimes \gamma)$ are the expansions of the same element. Therefore we only need to prove that the first and last expression are the expansions of the same element.
	
	By the skew symmetry property, we have the following identity
	\begin{equation*}
		\begin{aligned}
			&Y(z)\circ(\mathrm{id}_V \otimes Y(w))\circ(\alpha\otimes\beta\otimes \gamma)\\
			&  = Y(z)\circ(\mathrm{id}_V \otimes e^{wT}\circ Y(-w)\circ \sigma)\circ(\alpha\otimes\beta\otimes \gamma)\,.
		\end{aligned}
	\end{equation*}
	Then we use Lemma \ref{lem_trans} and find
	\begin{equation*}
		\begin{aligned}
			&Y(z)\circ(\mathrm{id}_V \otimes Y(w))\circ(\alpha\otimes\beta\otimes \gamma)\\
			& = e^{wT}\circ Y(z-w)\circ(\mathrm{id}_V\otimes Y(-w)\circ\sigma)\circ(\alpha\otimes\beta\otimes \gamma) \,.
		\end{aligned}
	\end{equation*}
	On the other hand, $Y(z-w)\circ(\alpha\otimes\beta) = \sum_{n\in \Z}(z-w)^{-n-1}\cdot_{n}\circ(\alpha\otimes \beta)$ by definition. The composition $\cdot_{n}\circ(\alpha\otimes \beta)$ is still a map from a compact object $X_1\otimes X_2$. Therefore we can apply the skew symmetry property to $Y(w)\circ(\cdot_{n}\circ(\alpha\otimes \beta)\otimes \gamma)$, which gives us
	\begin{equation*}
		\begin{aligned}
			&Y(w)\circ(Y(z-w)\otimes\mathrm{id}_V)\circ(\alpha\otimes\beta\otimes\gamma) \\
			&= e^{wT}\circ Y(-w)\circ  \sigma \circ (Y(z-w)\otimes \mathrm{id}_V)\circ(\alpha\otimes\beta\otimes\gamma)\,.
		\end{aligned}
	\end{equation*}
	By applying the locality axiom again, we find that  $Y(z)\circ(\mathrm{id}_V \otimes Y(w))\circ(\alpha\otimes\beta\otimes \gamma) $ and $Y(w)\circ(Y(z-w)\otimes\mathrm{id}_V)\circ(\alpha\otimes\beta\otimes\gamma)$ are expansions of the same elements.
\end{proof}

\begin{theorem}(Borcherds identity)
	For any compact objects $X_i$, $i = 1,2,3$ and morphisms $\alpha\in \mathrm{Hom}_{\EuScript{C}}(X_1,V)$, $\beta\in \mathrm{Hom}_{\EuScript{C}}(X_2,V)$, $\gamma \in \mathrm{Hom}_{\EuScript{C}}(X_3,V)$, we have the following identity
	\begin{equation}\label{eqn_Bor}
		\begin{aligned}
			\sum_{n\geq 0}&\binom{m}{n}\cdot_{m+k-n}\circ (\cdot_{n+l}\otimes \mathrm{id}_V)\circ(\alpha\otimes\beta\otimes\gamma) = \\
			\sum_{j \geq 0 }&\binom{l}{j}(-1)^j\Big(\cdot_{m+l-j}\circ(\mathrm{id}_L\otimes\cdot_{k+j}) -\\
			&(-1)^l\cdot_{k+l-j}\circ(\mathrm{id}_V\otimes\cdot_{m+j})\circ(\sigma\otimes\mathrm{id}_V) \Big)\circ(\alpha\otimes\beta\otimes\gamma)\,.
		\end{aligned}
	\end{equation}
	for any integers $m,k,l$.
\end{theorem}
\begin{proof}
	By Theorem \ref{thm_ass}, the three expressions $Y(z)\circ(\mathrm{id}_V \otimes Y(w))(\alpha\otimes\beta\otimes \gamma)$, $Y(w)\circ(\mathrm{id}_V \otimes Y(z))\circ(\sigma\otimes \mathrm{id}_V)(\alpha\otimes\beta\otimes \gamma)$ and $Y(w)\circ(Y(z-w)\otimes\mathrm{id}_V)(\alpha\otimes\beta\otimes\gamma)$ are the expansion of the same elements $A(z,w)$ in $\mathrm{Hom}_{\EuScript{C}}(\mathbf{X},V)[[z,w]][z^{-1},w^{-1},(z-w)^{-1}]$, where $\mathbf{X} = X_1\otimes X_2 \otimes X_3$. Let $f(z,w)$ be a rational function which has poles only at $z = 0,w=0$ and $z=w$. Let $R>\rho>r>0$, we consider the contour integral
	\begin{equation}\label{eqn_contour12}
		\begin{aligned}
			&\oint_{C_w^{\rho}}\oint_{C_z^R}Y(z)\circ(\mathrm{id}_V \otimes Y(w))(\alpha\otimes\beta\otimes \gamma)f(z,w)dzdw \\
			&- \oint_{C_w^{\rho}}\oint_{C_z^r}Y(w)\circ(\mathrm{id}_V \otimes Y(z))\circ(\sigma\otimes \mathrm{id}_V)(\alpha\otimes\beta\otimes \gamma)f(z,w)dzdw \,.
		\end{aligned}
	\end{equation} 
	This contour integral can be written as $\oint_{C_w^{\rho}}\oint_{C_z^R - C_z^r}X(z,w)f(z,w)dzdw$. We can further replace $C_z^R - C_z^r$ by a circle $C_{z}^{\delta}(w)$ of radius $\delta < \rho$ around $w$. In this region, $A(z,w)$ is expanded as $Y(w)\circ(Y(z-w)\otimes\mathrm{id}_V)(\alpha\otimes\beta\otimes\gamma)$. We find that \ref{eqn_contour12} is equal to
	\begin{equation*}
		\oint_{C_w^{\rho}}\oint_{C_{z}^{\delta}(w)}Y(w)\circ(Y(z-w)\otimes\mathrm{id}_V)(\alpha\otimes\beta\otimes\gamma)f(z,w)dzdw\,.
	\end{equation*}
	If we choose $f(z,w) = z^mw^k(z-w)^{l}$, the above identity gives us \ref{eqn_Bor}.
\end{proof}

For a vertex algebra $V$ in $\mathrm{Vect}_k$, the Borcherds identity can often be used to construct Lie algebras or associative algebras from the vertex algebra. For instance, $V / T V$ naturally carries a Lie algebra structure defined via the operation $\cdot_{(0)}$, and the space $(V\otimes \C[t,t^{-1}])/\mathrm{Im}(T\otimes \mathrm{id} + \mathrm{id}\otimes \partial_t)$ also carries a Lie algebra structure defined via the operations $\cdot_{(n)}$ for $n \geq 0$. When $\EuScript{C}$ is abelian, i.e., when it admits arbitrary quotients, these constructions should also extend to the categorical setting, producing Lie algebra or associative algebra objects in the category.\footnote{The associative algebra of modes associated with a vertex algebra is more subtle, as it involves a completion with respect to the natural topology on $\mathbb{C}[t, t^{-1}]$ (see e.g., \cite{frenkel2004vertex}). Such a construction yields a pro object (or a Tate object) in the category $\EuScript{C}$.}

\subsection{Fields and OPE}
An important concept in the study of vertex algebras is that of field (also called operator in some literature). In fact, the map $Y(z)$ is commonly called the state-field correspondence. However, in the categorical context, the space of states $V$ becomes an object, and it no longer makes sense to refer to an individual state or field. Instead, we propose the following definition:
\begin{definition}
	Given a compact object $X$, we call a collection of maps $A(z) = \sum_{n}A_nz^{-n-1}$
	\begin{equation*}
		A_n: X\otimes V \to V
	\end{equation*}
	a field labeled by $X$ if for any compact object $X'$ and morphism $\beta: X' \to V$, there exists an integer $K$ such that $A_n\circ(\mathrm{id}_{X}\otimes \beta) = 0$ for all $n > K$. In other words, a field $A(z)$ labeled by $X$ is a morphism
	\begin{equation*}
		A(z) \in \mathrm{Hom}_{\EuScript{C}}(X\otimes V,V((z)))\,.
	\end{equation*}
\end{definition}

Though it is not always possible to talk about ``a state" in $V$, we can instead consider a morphism $\alpha:X \to V$ from a compact object $X$. By construction, for any morphism $\alpha:X\to V$, the map $Y(z)(\alpha\otimes\mathrm{id}_V)$ is a field labeled by $X$. In the following, we will also denote 
\begin{equation*}
	Y(z)(\alpha\otimes\mathrm{id}_V) = Y(\alpha,z)\,.
\end{equation*}
This construction can be understood as a generalization of the state-field correspondence.
\begin{remark}
	For $\EuScript{C} = \mathrm{Vect}_{\C}$, a field labeled by $\C^n$ is simply a collection of $n$ fields. For $n = 1$, let $u \in V$ and $\alpha: \C \to \{\C u\} \subset V$, $Y(\alpha,z)$ is the field $Y(u,z)$ as in the usual vertex algebra notation.
\end{remark}

\begin{theorem}(Goddard's uniqueness theorem)
	Let $V$ be a vertex algebra in $\EuScript{C}$, and $A(z)$ a field labeled by $X$. Suppose there exists a map $\alpha:X\to V$ such that
	\begin{equation*}
		A(z)\circ(\mathrm{id}_X\otimes\vac) = Y(z)\circ(\alpha\otimes\vac)
	\end{equation*}
	and $A(z)$ is local with respect to the field $Y(z)\circ(\beta\otimes\mathrm{id}_{X'})$ for any $\beta: X' \to V$. Then $A(z) = Y(z)\circ(\alpha\otimes\mathrm{id}_V)$.
\end{theorem}
\begin{proof}
	Let $V = \underset{i\in I}{\mathrm{colim}}V_i$ with $V_i$ compact, and denote $s_i:V_i \to V$ the inclusion map. By the locality, we have, for large enough $N$
	\begin{equation*}
		\begin{aligned}
			(z-w)^NA(z)\circ(\mathrm{id}_X\otimes &(Y(w)\circ (s_i\otimes\vac)))\\
			& = (z-w)^N Y(w)\circ(s_i\otimes A(z))\circ (\sigma_{X,X'}\otimes\vac)\,.
		\end{aligned}
	\end{equation*}
	Since $A(z)(\mathrm{id}_X\otimes\vac) = Y(z)\circ(\alpha\otimes\vac)$, we further have
	\begin{equation*}
		\begin{aligned}
			(z-w)^NA(z)\circ(\mathrm{id}_X&\otimes (Y(w)\circ (s_i\otimes\vac)))\\
			& = (z-w)^N Y(w)\circ(s_i\otimes Y(z))\circ (\alpha\otimes\vac)\circ(\sigma_{X,X'}\otimes id_{V})\\
			& = (z-w)^N Y(z)\circ(\mathrm{id}_X\otimes Y(w))\circ(\alpha\otimes s_i\otimes\vac)\,.
		\end{aligned}
	\end{equation*}
	By the vacuum axiom, both sides of the above equation are well-defined at $w = 0$, and $Y(w)\circ(s_i\otimes \vac) = s_i$. Setting $w = 0$, and divide both sides by $z^N$, we obtain
	\begin{equation*}
		A(z)(\mathrm{id}_{X}\otimes s_i) = Y(z)\circ(\alpha\otimes s_i)\,.
	\end{equation*}
	This equation hold for any $i \in I$, therefore we have
	\begin{equation*}
		A(z) = Y(z)\circ(\alpha\otimes \mathrm{id}_V)\,.
	\end{equation*}
\end{proof}

\begin{corollary}\label{eqn_TDer}
	We have the identity
	\begin{equation*}
		Y(z)\circ(T\otimes\mathrm{id}_V) = \partial_zY(z)\,.
	\end{equation*}
\end{corollary}
\begin{proof}
	Let $V = \underset{i\in I}{\mathrm{colim}}V_i$ with $V_i$ compact, and denote $s_i:V_i \to V$ the inclusion map. We define the field $A(z) = \partial_z Y(z)\circ(s_i\otimes\mathrm{id}_V)$. Since $Y(z)\circ(s_i\otimes\mathrm{id}_V)$ satisfies the locality condition with any other $Y(z)\circ(\beta\otimes\mathrm{id}_V)$, $A(z)$ also satisfies the locality condition with any other $Y(z)\circ(\beta\otimes\mathrm{id}_V)$. We also have
	\begin{equation*}
		A(z)\circ(\mathrm{id}_{V_i}\otimes \vac) = \partial_zY(z)\circ(s_i\otimes\vac) = \partial_z e^{zT}s_i = e^{zT}Ts_i = Y(z)\circ(Ts_i\otimes\vac)\,.
	\end{equation*}
	By the Goddard's uniqueness theorem, we have $\partial_z Y(z)\circ(s_i\otimes\mathrm{id}_V) = Y(z)\circ(Ts_i\otimes\mathrm{id}_{V})$. This holds for any $i\in I$, which implies $Y(z)\circ(T\otimes\mathrm{id}_V) = \partial_zY(z)$.
\end{proof}

We define the notion of normally ordered product
\begin{definition}
	Let $X_1,X_2$ be two compact objects, and $A(z) = \sum_{n}A_nz^{-n-1}$, $B(z) = \sum_{n}B_nz^{-n-1}$ two fields labeled by $X_1$ and $X_2$ respectively. The normally ordered product $:A(z)B(w):$ is defined as the formal power series
	\begin{equation*}
		\begin{aligned}
			\sum_{n \in \Z}&\left(\sum_{m<0}A_m\circ(\mathrm{id}_X\otimes B_n)z^{-m-1}\right. \\
			&\left. + \sum_{m\geq0}B_n\circ(\mathrm{id}_{X'}\otimes A_m)\circ(\sigma\otimes \mathrm{id}_V)z^{-m-1} \right)w^{-n-1} 
		\end{aligned}
	\end{equation*}
	as an element in $\mathrm{Hom}_{\EuScript{C}}(X_1\otimes X_2\otimes V,V)[[z^{\pm},w^{\pm}]]$. Equivalently, we have
	\begin{equation}\label{eq:nor_ord}
		:A(z)B(w): = A(z)_+\circ (\mathrm{id}_{X_1}\otimes B(w)) + B(w)\circ(\mathrm{id}_{X_2}\otimes A(z)_-)\circ(\sigma\otimes \mathrm{id}_V)\,.
	\end{equation}
\end{definition}
\begin{lemma}
	The specialization of $:A(z)B(w):$ at $w = z$ is a well defined field labeled by $X_1\otimes X_2$. Moreover,
	\begin{equation*}
		\begin{aligned}
			:A(w)B(w): = \mathrm{Res}_{z = 0}(&\delta(z-w)_- A(z)\circ (\mathrm{id}_{X_1}\otimes B(w))  \\
			&+ \delta(z-w)_+B(w)\circ(\mathrm{id}_{X_2}\otimes A(z))\circ(\sigma\otimes \mathrm{id}_V) ) \,.
		\end{aligned}
	\end{equation*}
\end{lemma}
\begin{proof}
	Let us denote $C(z) = :A(z)B(z):$. As a formal expression, $C(z) = \sum_{l\in\Z}C_{l}z^{-l-1}$ with
	\begin{equation*}
		C_{l} = \sum_{n>l-1}A_{l-1-n}\circ(\mathrm{id}_{X_1}\otimes B_n) + \sum_{n\leq l-1}B_n\circ(\mathrm{id}_{X_2}\otimes A_{l-1-n})\circ(\sigma\otimes \mathrm{id}_V)\,.
	\end{equation*}
	To show that each $C_l$ is a well defined map $\mathrm{Hom}_{\EuScript{C}}(X_1\otimes X_2\otimes V,V)$, we write $V = \underset{i\in I}{\mathrm{colim}}V_i$. We show that each $C_l\circ(\mathrm{id}_{X_1\otimes X_2}\otimes s_i)$ is well defined. Since both $A(z)$ and $B(w)$ are fields, we can find an integer $K$ such that $A_n(\mathrm{id}_{X_1}\otimes s_i) = 0$, $B_n(\mathrm{id}_{X_2}\otimes s_i) = 0$ for $n > K$. Therefore,
	\begin{equation*}
		\begin{aligned}
			C_{l}\circ(\mathrm{id}_{X_1\otimes X_2}\otimes s_i) = &\sum_{n =l}^{K}A_{l-1-n}\circ(\mathrm{id}_{X_1}\otimes B_n\circ(\mathrm{id}_{X_2}\otimes s_i))\\
			& + \sum_{n=l-1-K}^{l-1}B_n\circ(\mathrm{id}_{X_2}\otimes A_{l-1-n})\circ(\sigma\otimes s_i)
		\end{aligned}
	\end{equation*}
	which is a well-defined finite sum. 
	
	For $l > K$, the summation in the first line vanishes. For the second line, we find that each $A_{m}\circ(\mathrm{id}_{X_1}\otimes s_i)$ is a map from the compact object $X_1\otimes V_i$ to $V$. We can further find another integer $K_m$ such that $B_n\circ(\mathrm{id}_{X_2}\otimes A_{m}\circ(\mathrm{id}_{X}\otimes s_i)) = 0$ for $n > K_m$. If we take $\bar{K} = \max\{K_m\}_{0\leq m \leq K}$, we find that $B_n\circ(\mathrm{id}_{X_2}\otimes A_{l-1-n})\circ(\sigma\otimes s_i) = 0$ for $n > \bar{K}$ in the summation range $l-1-K\leq n \leq l-1$. As a result, we find that $C_{l}\circ(\mathrm{id}_{X_1\otimes X_2}\otimes s_i) = 0$ for $l > K+\bar{K} + 1$. This prove that $C(z)$ is a field.
	
	The last identity is a simple computation
	\begin{equation*}
		\mathrm{Res}_{z = 0}\delta(z-w)_- A(z)\circ (\mathrm{id}_{X_1}\otimes B(w)) = A(w)_+\circ (\mathrm{id}_{X_1}\otimes B(w))\,.
	\end{equation*}
	Similarly
	\begin{equation*}
		\mathrm{Res}_{z = 0} \delta(z-w)_+B(w)\circ(\mathrm{id}_{X2}\otimes A(z)) = B(w)\circ(\mathrm{id}_{X2}\otimes A(z)_-)\,.
	\end{equation*}
\end{proof}

\begin{corollary}
	We have the identity
	\begin{equation*}
		\partial_w:A(w)B(w): = :\partial_wA(w) B(w): + :A(w) \partial_wB(w):\,.
	\end{equation*}
\end{corollary}

\begin{prop}
	Let $X_1,X_2$ be two compact objects, and $A(z) = \sum_{n}A_nz^{-n-1}$, $B(z) = \sum_{n}B_nz^{-n-1}$ two fields labeled $X_1$ and $X_2$ respectively. The following are equivalent:
	\begin{enumerate}
		\item $A(z)$ is local with respect to $B(z)$
		\item We have an identity in $\mathrm{Hom}(X_1\otimes X_2\otimes V,V)[[z^{\pm},w^{\pm}]]$
		\begin{equation}\label{eq:field_com}
			A(z)\circ(\mathrm{id}_{X_1}\otimes B(w)) - B(w)\circ(\mathrm{id}_{X_2}\otimes A(z))\circ(\sigma\otimes \mathrm{id}_V) = \sum_{j=0}^{K-1}\frac{C_j(w)}{j!}\partial^j_w\delta(z-w)\,,
		\end{equation}
		where $C_j(w)$ are fields labeled by $X_1\otimes X_2$.
		\item We have the identity
		\begin{equation}\label{eq:field_OPE}
			A(z)\circ(\mathrm{id}_{X_1}\otimes B(w)) = \sum_{j=0}^{K-1}\frac{C_j(w)}{(z-w)^{j+1}} + :A(z)B(w):\,,
		\end{equation}
		where $1/(z-w)$ is expanded in positive powers of $w/z$.
	\end{enumerate}
\end{prop}
\begin{proof}
	The equivalence between (1) and (2) follows immediately from the property of formal power series (See Corollary 2.2 in \cite{kac1998vertex}). Now suppose (2) hold. From \eqref{eq:nor_ord} we obtain that
	\begin{equation*}
\begin{aligned}
			&A(z)_-\circ(\mathrm{id}_{X_1}\otimes B(w)) - B(w)\circ(\mathrm{id}_{X_2}\otimes A(z)_-)\circ(\sigma\otimes \mathrm{id}_V)\\
			 &= A(z)\circ(\mathrm{id}_{X_1}\otimes B(w)) - :A(z)B(w):\,.
\end{aligned}
	\end{equation*}
	Then by taking the negative power coefficients of \eqref{eq:field_com} with respect to $z$, we obtain (3). To prove (2) from (3), we simply take the commutator of \ref{eq:field_OPE} and use the definition of the delta function $\delta(z-w)$.
\end{proof}
The equation \eqref{eq:field_OPE} is also called the operator product expansion (OPE) of the two fields. We will also use the expression
		\begin{equation*}
	A(z)\circ(\mathrm{id}_{X_1}\otimes B(w)) \sim \sum_{j=0}^{K-1}\frac{C_j(w)}{(z-w)^{j+1}} 
\end{equation*}
to keep track of the singular terms only.

Now, given two maps $\alpha:X_1 \to V$ and $\beta: X_2 \to V$, we have mutually local fields $Y(\alpha,z)$ and $Y(\beta,z)$. From the Theorem \ref{thm_ass}, it's easy to check that their OPE takes the following form
		\begin{equation*}
	Y(\alpha,z)\circ(\mathrm{id}_{X_1}\otimes Y(\beta,w)) = \sum_{j=0}^{K-1}\frac{Y(\cdot_{j}\circ(\alpha\otimes\beta),w)}{(z-w)^{j+1}} + :Y(\alpha,z)Y(\beta,z):\,.
\end{equation*}

\subsection{Basic examples}\label{sec:voa_ex}
In this section, we consider some basic examples of vertex algebra in $\EuScript{C}$. 
\subsubsection{Commutative vertex algebra}
The simplest example of a vertex algebra is the commutative vertex algebra.
\begin{definition}
	A vertex algebra $(V,\vac,T,Y)$ in $\EuScript{C}$ is called commutative if the morphisms $\cdot_n:V\otimes V \to V$ vanish for all $n \geq 0$.
\end{definition}

\begin{definition}
	A differential algebra in $\EuScript{C}$ is a commutative algebra $(A,m)$ equipped with a derivation $T$, i.e. a map $T:A\to A$ that satisfies the Leibniz rule $T\circ m =  m\circ(T\otimes \mathrm{id} + \mathrm{id}\otimes T)$.  
\end{definition}

\begin{prop}
	There is a one to one correspondence between commutative vertex algebra in $\EuScript{C}$ and unital differential algebra in $\EuScript{C}$.
\end{prop}
\begin{proof}
	Given a commutative vertex algebra $(V,\vac,T,Y)$, we define a unital commutative algebra structure on $V$. We define $m = \cdot_{-1}:V\otimes V\to V$.
	Let $V = \underset{i\in I}{\mathrm{colim}}V_i$ with $V_i$ compact, and denote $s_i: V_i \to V$ the canonical inclusion map. By the skew symmetry property \ref{eqn_skew}, we have
	\begin{equation*}
		m\circ(s_i\otimes s_j) = m\circ\sigma\circ(s_i\otimes s_j)
	\end{equation*}
	for any $i,j \in I$. Notice that the map $m$ is defined as an element in 
	\begin{equation*}
		m\in \mathrm{Hom}_{\EuScript{C}}(V\otimes V,V) = \lim_{(i,j)\in I\times I}\mathrm{Hom}_{\EuScript{C}}(V_i\otimes V_j,V)\,.
	\end{equation*}
	Therefore, $m\circ(s_i\otimes s_j)$ for all $i,j\in I$ determines $m$. We have $m = m\circ\sigma$. 
	
	By the Borcherds identity, we have
	\begin{equation*}
		m\circ(m\otimes\mathrm{id}_V)\circ(s_i\otimes s_j\otimes  s_k) = m\circ (\mathrm{id}_V\otimes m)\circ(s_i\otimes s_j\otimes  s_k)
	\end{equation*}
	for any $i,j,k \in I$. This implies the associativity of $m$. Therefore, $m$ defines a commutative algebra structure on $V$. $\vac$ is a unit follows from the vacuum axiom. 
	
	On the other hand, given a unital commutative algebra $(V,\vac,m)$ equipped with a derivation $T$, we define the vertex algebra structure by
	\begin{equation*}
		Y(z) = m\circ(e^{zT}\otimes \mathrm{id}_V)\,.
	\end{equation*}
	Then we can check that all axioms of vertex algebra are satisfied.
\end{proof}
Given a compact object $X$ in $\EuScript{C}$, we explicitly construct the commutative vertex algebra ``generated" by $X$. 

First we define $L_-X = \bigoplus_{n<0} Xt^n$ and define a operator $T:L_-X \to L_-X$ by $-\partial/\partial_t$. More precisely, $L_-X =  \underset{n\in \Z_{\geq 0}}{\mathrm{colim}}\bigoplus_{i = 0}^{n}X_{-i-1}$, where each $X_{-n - 1} \cong X$ is a copy of $X$. The operator $T$ by definition is an element in $\prod_{n=0}^{\infty}\mathrm{Hom}(X_{-n-1},X_{-n-2})$, which is given by
\begin{equation}\label{der_LX}
	T = ((n+1)\cdot\mathrm{id}_{X} \in \mathrm{Hom}(X_{-n-1},X_{-n-2}))_{n\in \Z_{\geq 0}}\,.
\end{equation}

Then we consider the symmetric algebra $S(L_-X)$ and extend $T$ to $S(L_-X)$ by Leibniz rule. More precisely, we first define $T: T^k(L_-X) \to T^k(L_-X)$ by
\begin{equation}\label{der_SLX}
	\sum_{i=0}^{k-1}\mathrm{id}_{L_-X}^{\otimes i}\otimes T\otimes \mathrm{id}_{L_-X}^{\otimes k-i-1}:	(L_-X)^{\otimes k} \to (L_-X)^{\otimes k}\,.
\end{equation}
We also define $T$ to be $0$ for $k = 0$. By abuse of notation we used the same symbol $T$ here. Its easy to check that $T$ commutes with the symmetric idempotent $e_{\mathrm{Sym}}$  \eqref{eqn:sym_end}: $e_{\mathrm{Sym}}\circ T = T = T\circ e_{\mathrm{Sym}}$. Therefore, $T$ defines an map $S(L_-X) \to S(L_-X)$ and it is a derivation with respect to the commutative product.

We have constructed $(S(L_-X),\cdot,T)$ as a unital commutative algebra with a derivation $T$, which is equivalent to a commutative vertex algebra structure.

\subsubsection{$\beta\gamma$ vertex algebra}
\label{sec:ex_betaga}
Next, we consider a class of vertex algebra called $\beta\gamma$ vertex algebra, which is also called the chiral Weyl algebra. We start with a compact symplectic object $X$, i.e. a compact object equipped with a symplectic form
\begin{equation*}
	\begin{aligned}
		\omega: X\otimes X \to \mathbb{1}\,,\\
		\omega\circ \sigma = - \omega\,.
	\end{aligned}
\end{equation*}

An important result that allows us to construct vertex algebra from generators is the Dong's lemma.
\begin{lemma}(Dong's Lemma) 
	If $A(z), B(z), C(z)$ are mutually local fields, then $:A(z)B(z):$ and $C(z)$ are also mutually local. 
\end{lemma}
\begin{proof}
The proof follows the same argument as Dong's Lemma for ordinary vertex algebras (see, e.g. \cite{frenkel2004vertex}).
\end{proof}

We define $LX = X[t^{\pm}]  := \underset{n\in \Z_{\geq 0}}{\mathrm{colim}}\bigoplus_{i = -n-1}^{n}X_{i}$, where as before each $Xt^n = X_{n} \cong X$ is a copy of $X$. We can  equip $LX$ with a symplectic form, given by
\begin{equation*}
	\delta_{n,-m-1}\omega \in \mathrm{Hom}(X_n\otimes X_m,\mathbb{1})\,.
\end{equation*}
Using the above structure, we can use the construction in section \ref{sec:voa_ex} and define the Weyl algebra $\EuScript{W}(LX) = (S(LX),\star)$. The Lagrangian decomposition $LX = L_+X \oplus L_-X$ induces a left $\EuScript{W}(LX)$ module structure on $S(L_-X)$.

As an object in $\EuScript{C}$ we set $V = V^{\beta\gamma} = S(L_-X)$ and define the map $T$ in a same way as \ref{der_LX}, \ref{der_SLX}. The vacuum is the natural map $\vac: \mathbb{1} \to S(L_-X)$. The only nontrivial part is to define the vertex algebra map $Y(z)$. By definition, it suffice to construct a series of maps
\begin{equation*}
	Y_k(z):	S^k(L_-X)\otimes V \to V((z)) \text{ for any } k \geq 0\,.
\end{equation*}
For $k = 0$, we define it by the vacuum axiom $Y_0(z) = \mathrm{id}_V \in \mathrm{Hom}(\mathbb{1}\otimes V,V)$. 

For $k = 1$, we can define it as a collection of fields $Y(X_{n},z)\in \mathrm{Hom}(X_n\otimes V,V((z)))$ for $n\leq -1$. We denote $X(z) := Y(X_{-1},z)\in \mathrm{Hom}(X_{-1}\otimes V,V((z)))$. Let $t^k = \mathrm{id}_X\in \mathrm{Hom}(X_n,X_{n+k})$ be the identity map that only shift the index by $k$. We define $X(z)$ as follows
\begin{equation}\label{eqn_defX}
	X(z) = \sum_{n\in \Z} z^{-n-1}  \star_l \circ(t^{n+1}\otimes \mathrm{id}_V) \in \mathrm{Hom}(X_{-1}\otimes V,V((z)))\,,
\end{equation}
where $\star_l$ is the module map defined in \eqref{eqn_starl}. To check that $X(z)$ is indeed an element in $\mathrm{Hom}(X_{-1}\otimes V,V((z)))$, we notice that $\star_l$ restricted to $\mathrm{Hom}(X_{n}\otimes S^{k}(\bigoplus_{i = -m-1}^{m}X_{i}),V)$ vanish when $k,m$ is fixed and $n$ large enough (See remark \ref{rmk:star}). Then we define the fields $Y(X_{-n-1},z)$ by
\begin{equation*}
	Y(X_{-n-1},z) = \frac{1}{n!}\partial^{n}_zX(z)\circ(t^n\otimes\mathrm{id}_V)  \in \mathrm{Hom}(X_{-n-1}\otimes V,V((z)))\,.
\end{equation*}

For higher $k$, we first define a collection of fields
\begin{equation*}
	Y(X_{-n_{1}-1}\dots X_{-n_k-1},z) \in \mathrm{Hom}(X_{-n_{1}-1}\otimes \dots \otimes X_{-n_k-1}\otimes V,V((z)))
\end{equation*} 
by normally ordered product. Namely, we set
\begin{equation}\label{eq:bg_ver_nor}
	Y(X_{-n_{1}-1}\dots X_{-n_k-1},z) = \frac{1}{n_1!\dots n_k!}:\pa_z^{n_1}X(z) \dots \pa_z^{n_k}X(z):\,.
\end{equation}
Such collection of fields defines a map $T^{k}(L_-X)\otimes V \to V((z))$. We use the inclusion $S^{k}(L_-X) \to T^k(L_-X)$ and find a map $S^{k}(L_-X)\otimes V \to V((z))$.

\begin{prop}
	The data $(V,\vac,T,Y)$ defined above is a vertex algebra in $\EuScript{C}$.
\end{prop}
\begin{proof}
	First we check the vacuum axiom.  $Y(z)\circ(\vac\otimes \mathrm{id}_V) = \mathrm{id}_V$ holds by definition. For the remaining part $Y(z)\circ (\mathrm{id}_V\otimes \vac)|_{z = 0} = \mathrm{id}_V$, we check this on each $S^k(L_-X)$. For $k = 1$, we find by definition of $X(z)$ \ref{eqn_defX} that $X(z)\circ(\mathrm{id}_{X_{-1}}\otimes \vac)|_{z = 0}$ is well defined and gives $\mathrm{id}_{X_{-1}}$. Similarly,  $Y(X_{-n-1},z)\circ(\mathrm{id}_{X_{-n-1}}\otimes \vac)|_{z = 0} = \mathrm{id}_{X_{-n-1}}$. Then we can prove by induction that 
	\begin{equation*}
		Y(X_{-n_{1}-1}\dots X_{-n_k-1},z)\circ(\mathrm{id}_{X_{-n_{1}-1}\otimes\dots X_{-n_k-1}}\otimes\vac)|_{z} =  \mathrm{id}_{X_{-n_{1}-1}\otimes\dots X_{-n_k-1}}\,.
	\end{equation*}
	
	Then we check the translation axiom. $T\circ\vac = 0$ by definition. For the remaining part of the translation axiom, we also check it on each $S^k(L_-X)$. Notice that $T|_{X_0} = 0$. Also $T$ is a derivation of the Weyl algebra $\EuScript{W}(LX)$. For $k = 1$, we find 
	\begin{equation*}
		\begin{aligned}
			T\circ X(z) - X(z)\circ(\mathrm{id}_{X_{-1}}\otimes T) & = \sum_{n\in \Z}  z^{-n-1}  \star_l \circ(T\circ t^{n+1}\otimes \mathrm{id}_V)\\
			& = \sum_{n\in \Z}  -nz^{-n-1} \star_l \circ(T\circ t^n\otimes \mathrm{id}_V)\\
			& = \pa_z X(z)\,.
		\end{aligned}
	\end{equation*}
	Similarly, we can check that $T\circ 	Y(X_{-n-1},z) - 	Y(X_{-n-1},z)\circ(\mathrm{id}_{X_{-n-1}}\otimes T) = 	\partial_z Y(X_{-n-1},z)$. We prove by induction and that $\partial_z$ satisfies the Leibniz rule with respect to the normally ordered product.
	
	Finally, we check the locality axiom. As implied by the Dong's lemma, we only need to check that the field $X(z)$ is local with respect to itself. We compute
	\begin{equation}\label{eqn_XX}
		\begin{aligned}
			X(z)\circ(\mathrm{id}_{X_{-1}}\otimes X(w))& = \sum_{n,m \in \Z} z^{-n-1}w^{-m-1} \star_l \circ (t^{n+1}\otimes \star_l \circ (t^{m+1} \otimes \mathrm{id}_V))\\
			& = \sum_{n,m \in \Z} z^{-n-1}w^{-m-1} \star_l \circ (\star\circ (t^{n+1}\otimes t^{m+1})\otimes \mathrm{id}_{V})\,.
		\end{aligned}
	\end{equation}
	Therefore, the commutator is given by
	\begin{equation*}
		\begin{aligned}
			&X(z)\circ(\mathrm{id}_{X_{-1}}\otimes X(w)) - X(w)\circ(\mathrm{id}_{X_{-1}}\otimes X(z))\circ  (\sigma\otimes\mathrm{id}_V)\\
			=& \sum_{n,m \in \Z} z^{-n-1}w^{-m-1} \star_l \circ ([-,-]_{\star}\circ (t^{n+1}\otimes t^{m+1})\otimes \mathrm{id}_{V})\,,
		\end{aligned}
	\end{equation*}
	where $[-,-]_{\star} = \star - \star\circ\sigma$. We have, by definition, 
	\begin{equation}
		[-,-]_{\star}\circ (t^{n+1}\otimes t^{m+1}) = \delta_{n,-m-1}\omega \in \mathrm{Hom}(X_{-1}\otimes X_{-1},\mathbb{1})\,.
	\end{equation}
	This implies
	\begin{equation*}
		\begin{aligned}
			&X(z)\circ(\mathrm{id}_{X_{-1}}\otimes X(w)) - X(w)\circ(\mathrm{id}_{X_{-1}}\otimes X(z))\circ  (\sigma\otimes\mathrm{id}_V)\\
			=& \delta(z-w)l_V\circ(\omega\otimes\mathrm{id}_{V})\,.
		\end{aligned}
	\end{equation*}
	Hence  $X(z)$ is local with respect to itself.
\end{proof}

\begin{remark}
	From \ref{eqn_XX} we can easily check that 
	\begin{equation}\label{eq:OPE_symp}
		X(z)\circ(\mathrm{id}_{X_{-1}}\otimes X(w)) = \frac{1}{z - w}l_V\circ(\omega\otimes\mathrm{id}_{V}) + :X(z)X(w):\,.
	\end{equation}
	This is the OPE of the field $X(z)$ with itself. From this, we see that for $\EuScript{C} = \mathrm{Vect}_k$, the above construction reproduces the usual $\beta\gamma$ vertex algebra associated with a symplectic vector space.
\end{remark}

We also consider a one parameter family version of the $\beta\gamma$ vertex algebra over $\C[\hbar]$. As an object, we define $V^{\beta\gamma}_{\hbar} = S(L_-X)[\hbar]$. It is a module over the Weyl algebra $\EuScript{W}_{\hbar}(LX)$. The definition of the vertex algebra structure on $V^{\beta\gamma}_{\hbar}$ is similar. But now the OPE is multiplied by  $\hbar$.
\begin{equation*}
	X(z)\circ(\mathrm{id}_{X_{-1}}\otimes X(w)) = \frac{\hbar}{z - w}l_V\circ(\omega\otimes\mathrm{id}_{V}) + :X(z)X(w):\,.
\end{equation*}
It is easy to see that the specialization $V^{\beta\gamma}_{\hbar = 0}$ gives us a commutative vertex algebra.

\subsection{Functorial properties}
So far, we have only considered properties parallel to ordinary vertex algebra. The categorical construction allows for establishing more general functorial properties, which we analyze in this section. Recall that a monoidal functor between two monoidal categories consists of a functor $F$ together with a natural transformation $J: F(-) \otimes F(-) \to F(- \otimes -)$ and a unit morphism $\epsilon$.\footnote{We do not assume $\epsilon$ and $J$ to be isomorphisms here without further assumption.} In this section, we further assume the functor $(F,J,\epsilon)$ to preserve filtered colimit, i.e.
\begin{equation*}
	F(\underset{i\in I}{\mathrm{colim}}X_i) = \underset{i\in I}{\mathrm{colim}}F(X_i)\,.
\end{equation*}

\begin{example}
	Let $\EuScript{C}$ and $\EuScript{D}$ two categories as per our requirements in section \ref{sec:def}, i.e. $\EuScript{C} \cong \mathrm{Ind}(\EuScript{C}_0), \EuScript{D} \cong \mathrm{Ind}(\EuScript{D}_0)$ and $\EuScript{C}_0,\EuScript{D}_0$ are pseudo tensor category. Let $(F_0,J,\epsilon)$ be a monoidal functor between $\EuScript{C}_0$ and $\EuScript{D}$. Then, $F_0$ has an extension
	\begin{equation}
		F:\EuScript{C} \to \EuScript{D}
	\end{equation}
	that preserve filtered colimits  \cite{kashiwara2005categories}. 
\end{example}

We first show that the image of a vertex algebra under such a functor is still a vertex algebra.

\begin{prop}
\label{prop:mfunctor_vertex}
	Let $(F,J,\epsilon):\EuScript{C} \to \EuScript{D}$ be a symmetric monoidal functor that preserve filtered colimit. Given a vertex algebra $(V,\vac,T,Y(z))$ in $\EuScript{C}$, we define 
	\begin{enumerate}
		\item $V_{F} = F(V)$.
		\item $\vac_F = F(\vac)\circ\epsilon$.
		\item   $T_F = F(T)$.
		\item $Y_F(z) = \sum_{n\in\Z}z^{-n-1}F(\cdot_{n})\circ J_{V,V}$.
	\end{enumerate}
	Then $(V_{F}, \vac_F, T_F, Y_F(z))$ is a vertex algebra in $\EuScript{D}$.
\end{prop}
\begin{proof}
	First we show that $Y_F(z)$ is indeed a map to the formal Laurent series. Let $V = \underset{i\in I}{\mathrm{colim}}V_i$ with $V_i \in \EuScript{C}_0$ and denote $s_i:V_i \to V$ the inclusion map. Take any compact object $X$ in $\EuScript{D}$ and $\alpha,\beta \in \mathrm{Hom}_{\EuScript{D}}(X,F(V))$. We have
	\begin{equation*}
		\mathrm{Hom}_{\EuScript{D}}(X,F(V)) = \mathrm{Hom}_{\EuScript{D}}(X,\underset{i\in I}{\mathrm{colim}}\,F(V_i)) = \underset{i\in I}{\mathrm{colim}}\,\mathrm{Hom}_{\EuScript{D}}(X,F(V_i))\,.
	\end{equation*}
	As a result, we can always find a $i \in I$ and a map $\alpha' \in \mathrm{Hom}_{\EuScript{D}}(X,F(V_i))$ such that $\alpha = F(s_i)\circ\alpha'$. Similarly we can find a $j \in I$ and a map $\beta' \in \mathrm{Hom}_{\EuScript{D}}(X,F(V_j))$ such that $\beta = F(s_j)\circ\beta'$. Therefore, to check that $F(\cdot_{n})\circ J_{V,V}\circ (\alpha\otimes\beta) = 0$ for large enough $n$, we only need to check that $F(\cdot_{n})\circ J_{V,V}\circ (F(s_i)\otimes F(s_j)) = 0$ for large enough $n$. 
	
	Since $J$ is a natural transformation, we have $J_{V,V}\circ (F(s_i)\otimes F(s_j)) = F(s_i\otimes s_j)\circ J_{V_i,V_j}$. Therefore $F(\cdot_{n})\circ J_{V,V}\circ (F(s_i)\otimes F(s_j)) = F(\cdot_n\circ(s_i\otimes s_j))\circ J_{V_i,V_j}$, which vanishes for large enough $n$.
	
	Next, we check the vacuum axiom. By an abuse of notation, we write $Y_F(z) = F(Y(z))\circ J_{V,V}$. We have
	\begin{equation*}
		\begin{aligned}
			Y_F(z)\circ(\vac_F\otimes \mathrm{id}_{V_F}) & = F((z))\circ J_{V,V}\circ (F(\vac)\circ\epsilon \otimes \mathrm{id}_{V_F})\\
			& = 	F(Y(z)\circ (\vac\otimes \mathrm{id}_V) )\circ J_{\mathbb{1},V}\circ(\epsilon\otimes \mathrm{id}_{V_F})\\
			& = F(l_V)\circ J_{\mathbb{1},V}\circ(\epsilon\otimes \mathrm{id}_{V_F})\\
			& = l_{V_F}\,.
		\end{aligned}
	\end{equation*}
	Similarly, $Y_F(z)\circ(\mathrm{id}_{V_F} \otimes \vac_F) = F(Y(z)\circ (\mathrm{id}_V\otimes\vac))\circ J_{V,\mathbb{1}}\circ( \mathrm{id}_{V_F}\otimes\epsilon)$, which has no pole and specialize to $F(r_V)\circ J_{V,\mathbb{1}}\circ( \mathrm{id}_{V_F}\otimes\epsilon) = r_{V_F}$ at $z = 0$.
	
	Then we check the translation axiom. $T_F\circ \vac_F = F(T\circ\vac)\circ\epsilon = 0$. By definition, $T_F\circ  Y_F(z) = F(T\circ Y(z))\circ J_{V,V}$. On the other hand, 
	\begin{equation*}
		\begin{aligned}
			Y_F(z)\circ(\mathrm{id}_{V_F}\otimes T_F) &= F(Y(z))\circ J_{V,V}\circ (\mathrm{id}_{V_F}\otimes F(T))\\
			& =  F(Y(z)\circ(\mathrm{id}_{V} \otimes T))\circ J_{V,V}\,.
		\end{aligned}
	\end{equation*}
	Therefore, $T_F\circ \vac_F - Y_F(z)\circ(\mathrm{id}_{V_F}\otimes T_F) = F(\pa_zY(z))\circ J_{V,V} = \pa_zY_F(z)$.
	
	Finally, we check the locality axiom. As before, it suffice to check the locality for maps $F(s_i),F(s_j)$. We have 
	\begin{equation}\label{eqn:Floc1}
		\begin{aligned}
			&Y_F(w)\circ(\mathrm{id}_{V_F}\otimes Y_F(z))\circ(F(s_i) \otimes F(s_j)\otimes \mathrm{id}_{V_F})\\
			=&F(Y(w))\circ J_{V,V} \circ(\mathrm{id}_{V_F}\otimes F(Y(z))\circ J_{V,V})\circ(F(s_i) \otimes F(s_j)\otimes \mathrm{id}_{V_F})\\
			= &F(Y(w))\circ J_{V,V} \circ (F(s_i) \otimes F(Y(z)\circ (s_j\otimes \mathrm{id}_V))\circ (\mathrm{id}_{F(V_i)}\otimes J_{V_j,V})\\
			= & F(Y(w)\circ(\mathrm{id}_V\otimes Y(z))\circ (s_i\otimes s_j\otimes \mathrm{id}_V)))\circ J_{V_i,V_j\otimes V}\circ(\mathrm{id}_{F(V_i)}\otimes J_{V_j,V})\,.
		\end{aligned}
	\end{equation}
	Similarly, 
	\begin{equation*}
		\begin{aligned}
			&Y_F(z)\circ(\mathrm{id}_{V_F}\otimes Y_F(w))\circ(\sigma\otimes\mathrm{id}_{V_F})\circ(F(s_i) \otimes F(s_j) \otimes \mathrm{id}_{V_F})\\
			= & F(Y(w)\circ(\mathrm{id}_V\otimes Y(z))\circ (s_j\otimes s_i\otimes \mathrm{id}_V)))\circ J_{V_j,V_i\otimes V}\circ(\mathrm{id}_{F(V_j)}\otimes J_{V_i,V})\circ(\sigma\otimes\mathrm{id}_{V_F})\\
			=&  F(Y(w)\circ(\mathrm{id}_V\otimes Y(z))\circ (s_j\otimes s_i\otimes \mathrm{id}_V)))\circ J_{V_j\otimes V_i,V}\circ(J_{V_j,V_i}\otimes \mathrm{id}_{V_F})\circ(\sigma\otimes\mathrm{id}_{V_F})\,.\\
		\end{aligned}
	\end{equation*}
	Since $F$ is braided, $J_{V_j,V_i}\circ \sigma = F(\sigma)\circ J_{V_i,V_j}$. Moreover, $J_{V_j\otimes V_i,V}\circ (F(\sigma)\otimes F(\mathrm{id}_{V})) = F(\sigma\otimes \mathrm{id}_V)\circ J_{V_i\otimes V_j,V}$. We find
	\begin{equation}\label{eqn:Floc2}
		\begin{aligned}
			&Y_F(z)\circ(\mathrm{id}_{V_F}\otimes Y_F(w))\circ(\sigma\otimes\mathrm{id}_{V_F})\circ(F(s_i) \otimes F(s_j) \otimes \mathrm{id}_{V_F})\\
			=&  F(Y(w)\circ(\mathrm{id}_V\otimes Y(z))\circ (s_j\otimes s_i\otimes \mathrm{id}_V)\circ (\sigma\otimes\mathrm{id}_V) )\circ J_{V_i,V_j\otimes V}\circ(\mathrm{id}_{F(V_i)}\otimes J_{V_j,V})\,.
		\end{aligned}
	\end{equation}
	In summary, \eqref{eqn:Floc1} - \eqref{eqn:Floc2} $ = F([\dots])\circ J_{V_i,V_j\otimes V}\circ(\mathrm{id}_{F(V_i)}\otimes J_{V_j,V})$, which vanishes when we multiply it by $(z-w)^K$ for $K$ large enough.
\end{proof} 

Morphism of vertex algebra is also preserved by such functors.
\begin{prop}\label{prop:fun_mor_vec}
Let $(F, J, \epsilon)$ be a symmetric monoidal functor as above. Let $(V, \vac_V, T_V, Y_V)$ and $(W, \vac_W, T_W, Y_W)$ be two vertex algebras in $\EuScript{C}$ and $\phi: V \to W$ a morphism of vertex algebras. We have that $F(\phi): V_F \to W_F$ is a morphism of vertex algebras.
\end{prop}
\begin{proof}
	By definition $F(\phi)\circ\vac_{V_F} = F(\phi\circ\vac_V)\circ\epsilon = \vac_{W_{F}}$ and $F(\phi)\circ T_{V_F} = F(\phi\circ T_{V}) = F(T_{W}\circ\phi) = T_{F_W}\circ F(\phi)$. Moreover, we have $F(\phi)\circ Y_{V_F} = F(\phi)\circ F(Y_V)\circ J_{V,V} = F(Y_W\circ(\phi\otimes\phi))\circ J_{V,V} = F(Y_W)\circ F(\phi\otimes\phi)\circ J_{V,V} = F(Y_W)\circ J_{W,W}\circ( F(\phi)\otimes F(\phi)) = Y_{W_F}\circ( F(\phi)\otimes F(\phi))$.
\end{proof}

For a category $\EuScript{C} = \mathrm{Ind}(\EuScript{C}_0)$ as above, the endomorphism of identity $R = \mathrm{End}_{\EuScript{C}}(\mathbb{1}) $ is a commutative ring. The functor
\begin{equation*}
	\mathrm{Hom}_{\EuScript{C}}(\mathbb{1},-): \EuScript{C} \to R\mathrm{-mod}
\end{equation*}
extend to a symmetric monoidal functor. The natural transformation $J_{X,Y}:	\mathrm{Hom}_{\EuScript{C}}(\mathbb{1},X)\otimes_R	\mathrm{Hom}_{\EuScript{C}}(\mathbb{1},Y) \to \mathrm{Hom}_{\EuScript{C}}(\mathbb{1},X\otimes Y)$ is given by the tensor product $\alpha\otimes\beta$ composed with the isomorphism $\mathbb{1}\cong \mathbb{1}\otimes\mathbb{1}$.  It commute with the permutation $\sigma$ by construction. This functor also preserves filtered colimit since $\mathbb{1}\in \EuScript{C}_0$ is a compact object.

As a corollary, we have an easy way to construct a vertex algebra over the ring $R$ from a vertex algebra in $\EuScript{C}$.

\begin{corollary}\label{cor:inv_voa}
	Let $(V,\vac,T,Y)$ be a vertex algebra in $\EuScript{C}$. We write $R = \mathrm{End}_{\EuScript{C}}(\mathbb{1})$. Then $\widetilde{V}:= \mathrm{Hom}_{\EuScript{C}}(\mathbb{1},V)$ have the structure of vertex algebra over $R$, defined as follows
	\begin{enumerate}
		\item The vacuum vector is simply $\vac \in \widetilde{V}$.
		\item The translation operator $\widetilde{T}:\widetilde{V} \to \widetilde{V}$ is constructed as composition with $T$: for any $\alpha \in \widetilde{V}$, $\widetilde{T}\alpha$ is the composite map $\mathbb{1}\overset{\alpha}{\to} V \overset{T}{\to}  V$.
		\item $\widetilde{Y}(z):\widetilde{V}\otimes \widetilde{V} \to \widetilde{V}((z))$ is constructed as composition with $Y(z)$: for any $\alpha,\beta \in \widetilde{V}$, $\alpha\cdot_{n}\beta$ is the composite map $\mathbb{1}\overset{\alpha\otimes \beta}{\to} V\otimes V \overset{\cdot_n}{\to}  V$. 
	\end{enumerate} 
\end{corollary}

\begin{remark}
Although the notion of a vertex algebra over a ring (rather than a field) is not commonly considered in the literature, we will later apply the above construction to the category $\mathrm{Rep}(\mathrm{GL}_{[N]})$, which produces a vertex algebra over the ring $\C[N]$. Keeping track of $N$ as an indeterminate parameter instead of a fixed number allows us to define certain vertex Poisson algebra limits in section \ref{sec:chiralPoi}.
\end{remark}

Given a strong monoidal functor $(F,J,\epsilon):\EuScript{C} \to \EuScript{D}$ that preserves filtered colimits, for a vertex algebra $V$ in $\EuScript{C}$, we also get a vertex algebra $F(V)$ in $\EuScript{D}$. We can further apply the functor $\mathrm{Hom}(\mathbb{1},-)$ on both sides. This gives us two vertex algebras $\widetilde{V} = \mathrm{Hom}_{\EuScript{C}}(\mathbb{1}_{\EuScript{C}},V)$ and $\widetilde{F(V)}=\mathrm{Hom}_{\EuScript{D}}(\mathbb{1}_{\EuScript{D}},F(V))$. Note that the functor $F$ also provides us a map $\widetilde{F}: \widetilde{V} \to \widetilde{F(V)}$, we will show that this map preserve the vertex algebra structure.

We consider the general case when the ground ring $R_{\EuScript{C}} = \mathrm{End}_{\EuScript{C}}(\mathbb{1})$ and $R_{\EuScript{D}} = \mathrm{End}_{\EuScript{D}}(\mathbb{1})$ are not the same. It's important to require $(F,J,\epsilon)$ to be a strong monoidal functor here. In this case, $\epsilon:\mathbb{1}_{\EuScript{D}} \to F(\mathbb{1}_{\EuScript{C}})$ is an isomorphism, and we get a ring map $F:R_{\EuScript{C}} \to R_{\EuScript{D}}$. $\widetilde{V}$ is a vertex algebra over $R_{\EuScript{C}}$ and $\widetilde{F(V)}$ is a over $R_{\EuScript{D}}$. For simplicity, we denote the extension $\widetilde{V}\otimes_{R_{\EuScript{C}}}R_{\EuScript{D}} = \widetilde{V}_{R_{\EuScript{D}}}$. Then we have the following
\begin{prop}\label{prop:functor_vertex_vect}
	The map $\widetilde{F}: \widetilde{V}_{R_{\EuScript{D}}} \to \widetilde{F(V)}$ is a morphism of vertex algebra (over $R_{\EuScript{D}}$).
\end{prop}
\begin{proof}
	More precisely, the map $\widetilde{F}: \widetilde{V}_{R_{\EuScript{D}}} \to \widetilde{F(V)}$ is defined as $\widetilde{F}(\alpha\otimes s) = F(\alpha)\circ\epsilon\circ s$ for $\alpha :\mathbb{1}_{\EuScript{C}}\to V$ and $s\in R_{\EuScript{D}}$. $\widetilde{F}$ is a morphism of $R_{\EuScript{D}}$ module by construction. We check that the vertex algebra structure is preserved under $\widetilde{F}$. First, it maps the vacuum vector to $\widetilde{F}(\vac) = F(\vac)\circ\epsilon$, which is the vacuum in $\widetilde{F(V)}$ by definition. Then, for any $\alpha \in \widetilde{V}$ and $s\in R_{\EuScript{D}}$, we have $\widetilde{F}(T\alpha\otimes s) = F(T\circ\alpha)\circ\epsilon\circ s = F(T)\circ\widetilde{F}(\alpha)$. Therefore $\widetilde{F}$ intertwines the translation maps. Finally, we check that it is compatible with the product $\cdot_n$. Recall that $\alpha\cdot_n\beta$ is defined as the composite map $\mathbb{1}\overset{\alpha\otimes \beta}{\to} V\otimes V \overset{\cdot_n}{\to}  V$. We find that $\widetilde{F}(\alpha\cdot_{n}\beta)$ is given by the composition $\mathbb{1}\overset{\epsilon}{\to} F(\mathbb{1}) \overset{F(\alpha\otimes \beta)}{\to} F(V\otimes V) \overset{F(\cdot_n)}{\to}  F(V)$. Using the coherence map $J$ we find that this expression can be written as $\mathbb{1}\overset{\epsilon\otimes \epsilon}{\to} F(\mathbb{1})\otimes F(\mathbb{1}) \overset{F(\alpha)\otimes F(\beta)}{\to} F(V)\otimes (V)\overset{J_{V,V}}{\to}F(V\otimes V) \overset{F(\cdot_n)}{\to}  F(V)$ this is $\widetilde{F}(\alpha)\cdot_{n}\widetilde{F}(\beta)$ by definition. For more general $(\alpha\otimes s_1)\cdot_n(\beta\otimes s_2)$, simply note that $\cdot_n$ on  $\widetilde{V}_{R_{\EuScript{D}}} $ is defined to be bilinear over $R_{\EuScript{D}}$.
\end{proof}

\subsection{Vertex algebra module}
\label{sec:module}
Given a vertex algebra object $V$ in $\EuScript{C}$, we can define its module in $\EuScript{C}$ in the same way as for an ordinary vertex algebra. However, we will not discuss that in this paper. In this section, we consider modules for the vertex algebra $\widetilde{V} = \mathrm{Hom}_{\EuScript{C}}(\mathbb{1},V)$. It turns out that we also have a family of vertex algebra modules over $\widetilde{V}$ labeled by compact objects of $\EuScript{C}$.

\begin{prop}\label{prop:ver_mod}
	The functor $\mathrm{Hom}_{\EuScript{C}}(-,V): (\EuScript{C}^c)^{op} \to R-\mathrm{mod}$ factor through the forgetful functor $\widetilde{V}-\mathrm{mod} \to R-\mathrm{mod}$.
\end{prop}
\begin{proof}
	For any compact object $X$ in $\EuScript{C}$, we denote $M_X = \mathrm{Hom}_{\EuScript{C}}(X,V)$. We construct a $\widetilde{V}$ module structure $Y_M: \widetilde{V}\otimes M_X \to M_X((z))$ on $M_X$. For any $\alpha \in \widetilde{V}$ and $v \in M_X$, we define $\alpha\cdot_{n,M}v$ as the composition
	\begin{equation*}
		\alpha\cdot_{n,M}v: X \cong \mathbb{1}\otimes X \overset{\alpha\otimes v}{\to} V\otimes V \overset{\cdot_{n}}{\to} V \,.
	\end{equation*}
	And we set $Y_M(\alpha,z)v = \sum_{n\in \Z} \alpha\cdot_{n,M}v \,z^{-n-1}$. Since $X$ is a compact object, for any $\alpha \in \widetilde{V}$ and $v \in M_X$, we can find a $K$ such that $\alpha\cdot_{n,M}v = 0$ for $n>K$. Therefore, $Y_M(\alpha,z)v$ is indeed a Laurent series. 
	
	By the identity axiom for $V$, we have $\vac\cdot_{n,M}v = \delta_{n,-1}v$. This implies $Y_M(\vac,z)v = v$ for any $v \in M_X$.
	
	The Borcherds identity \ref{eqn_Bor} immediately implies the Jacobi identity for $Y_M$.
	
	For any two compact objects $X_1,X_2$ and a morphism $f\in\mathrm{Hom}_{\EuScript{C}}(X_1,X_2)$. We can check that the map $f^*:M_{X_2} \to M_{X_1}$, defined by $f^*(v) = v\circ f$, intertwine with the module map, i.e. $\alpha\cdot_{n,M_1}f^*(v) = f^*(\alpha\cdot_{n,M_2}v)$.
\end{proof}

Given a compact object $X\in\EuScript{C}$, we can construct a collection of compact objects: $\{X^{\otimes n}\}_{n\in \Z_{\geq 0}}$. Then Proposition \ref{prop:ver_mod} gives a collection of vertex algebra modules $\{M_{X^{\otimes n}} = \mathrm{Hom}_{\EuScript{C}}(X^{\otimes n},V)\}_{n\in \Z_{\geq 0}}$ over $\widetilde{V}$. We have the following lattice-type construction.
\begin{prop}\label{prop:ver_mod_lat}
	Let us denote
	\begin{equation*}
		V_{X^{\Z_{\geq 0}}} = \bigoplus_{n\geq0}M_{X^{\otimes n}}\,.
	\end{equation*}
	We have that $V_{X^{\Z_{\geq 0}}}$ is a vertex algebra over $R$.
\end{prop}
\begin{proof}
	The vacuum vector is $\vac \in \widetilde{V}\subset V_{X^{\Z_{\geq 0}}}$. The translation map  $T:V\to V$  induces a translation map on $T^{\Z_{\geq 0}}:V_{X^{\Z_{\geq 0}}} \to V_{X^{\Z_{\geq 0}}} $ by composition $T^{\Z_{\geq 0}}(\alpha) = X^{\otimes n} \overset{\alpha}{\to} V \overset{T}{\to} V$. The vertex algebra map $\cdot_n$ is defined as $\alpha\cdot_n\beta = X^{\otimes (n+m)} \overset{\alpha\otimes \beta}{\to}V\otimes V\overset{\cdot_n}{\to}V$ for $\alpha:X^{\otimes n} \to V$ and $\beta:X^{\otimes m} \to V$. The series $Y^{\Z_{\geq 0}}(\alpha,z)\beta := \sum_{n\in \Z}z^{-n-1}\alpha\cdot_n\beta$ is indeed a formal Laurent series because of the condition \eqref{cond_finite}. All the vertex algebra axioms for $(V_{X^{\Z_{\geq 0}}},\vac,T^{\Z_{\geq 0}},Y^{\Z_{\geq 0}})$ follows immediately from the corresponding axioms for $(V,\vac,T,Y)$.
\end{proof}

More generally, given a countable\footnote{We can also consider other regular cardinals as long as we can define filtered colimits, but we will not discuss this in this paper.} collection $\EuScript{S}$ of compact objects that is closed under tensor product, we can define 
\begin{equation*}
	V_{\EuScript{S}} = \bigoplus_{X\in \EuScript{S}}\mathrm{Hom}_{\EuScript{C}}(X,V)\,.
\end{equation*}
The same argument can show that $V_{\EuScript{S}}$ is also a vertex algebra over $R$. For $\EuScript{S} = \{\mathbb{1}\}$, this reproduces Corollary \ref{cor:inv_voa}. It is also clear from the definition that $V_{\EuScript{S}}$ is an extension of the vertex algebra $\widetilde{V}$.

\subsection{Differential and BRST reduction}
An important part of the construction in our later example will be the BRST reduction. 
Given a vertex algebra $V$ in $\EuScript{C}$, we call a map $D:V\to V$ a derivation of the vertex algebra if $D$ satisfies 
\begin{equation*}
	D\circ Y(z) = Y(z)\circ(D\otimes \mathrm{id}_V) + Y(z)\circ(\mathrm{id}_V\otimes D)\,.
\end{equation*}
By definition and Corollary \ref{eqn_TDer}. $T$ is a derivation for the vertex algebra $V$. If, moreover, $V$ is a graded object and $D$ is of degree $1$ and satisfies $D^2 = 0$, $D$ is called a differential. For an ordinary vertex algebra $V$, a natural way to construct a derivation or differential is to take an element $A \in V$ and consider the mode $A_{(0)} = \oint dz Y(A,z)$. We have a similar construction here.
\begin{prop}
	For any map $\EuScript{J}:\mathbb{1} \to V$, the map $\EuScript{J}_{(0)} := \cdot_{0}\circ(\EuScript{J}\otimes \mathrm{id}_V)\circ l_V^{-1}: V \cong \mathbb{1}\otimes V\to V$ is a derivation.
	
	Suppose moreover that $\EuScript{J}$ has degree $1$, and that the field $\EuScript{J}(z) = Y(\EuScript{J},z)$ has a regular OPE with itself, i.e. \begin{equation*}
		\EuScript{J}(z)\EuScript{J}(w) = :\EuScript{J}(z)\EuScript{J}(w):\,,
	\end{equation*}
	then the mode $\EuScript{J}_{(0)}$ is a differential.
\end{prop}
\begin{proof}
	By the Borcherds identity, we have
	\begin{equation*}
		\EuScript{J}_{(0)}\circ \cdot_{k}\circ(s_i\otimes s_j) = \cdot_{k}\circ(\EuScript{J}_{(0)}\circ s_i\otimes s_j) +  \cdot_{k}\circ( s_i\otimes \EuScript{J}_{(0)}\circ s_j) 
	\end{equation*}
	where $s_i:V_i \to V$ is the inclusion. This identity hold for any $s_i,s_j$, which implies that $\EuScript{J}_{(0)}$ is a derivation. The second statement also follows from the Borcherds identity.
\end{proof}

If we can find an element $\EuScript{J}$ that satisfies the above condition, this construction naturally gives us a complex in $\EuScript{C}$. When the category $\EuScript{C}$ is abelian, we can consider its cohomology, which will also be a vertex algebra object in $\EuScript{C}$. However, we will also consider the case when $\EuScript{C}$ is not abelian. Since we are ultimately interested in the vertex algebra $\widetilde{V} = \mathrm{Hom}(\mathbb{1},V)$, we can circumvent the problem of taking cohomology by first applying the functor $\mathrm{Hom}(\mathbb{1},-)$. In fact, $\EuScript{J}_{(0)}$ also acts on $\widetilde{V}$ via composition. It is easy to check that this action coincides with the action obtained by regarding $\EuScript{J}$ as an element of $\widetilde{V}$. It is known that the cohomology $H^{\bullet}(\widetilde{V},\EuScript{J}_{0})$ also carries a graded vertex algebra structure. Moreover, this construction of differential is compatible with the module construction in section \ref{sec:module}. In other words, $\EuScript{J}_{0}$ also acts on the vertex algebra module $M_X = \mathrm{Hom}(X,V)$ via composition. This action is compatible with the vertex algebra module structure:
\begin{equation*}
	\begin{aligned}
			\EuScript{J}_{(0)}(\alpha\cdot_{n,M}v) &= \EuScript{J}_{(0)}\circ\cdot_{n}\circ(\alpha\otimes v) \\
			& = \cdot_n\circ((\EuScript{J}_{(0)}\circ \alpha)\otimes v) + \cdot_{n}\circ(\alpha\otimes(\EuScript{J}_{(0)}\circ v))\\
			& = (\EuScript{J}_{(0)}\alpha)\cdot_{n,M}v + \alpha\cdot_{n,M}\EuScript{J}_{(0)}(v)\,,
	\end{aligned}
\end{equation*}
where we used the Borcherds identity in the second line. Therefore, when $\EuScript{J}_{(0)}$ is a differential, $H^{\sbullet}(M_X,\EuScript{J}_{(0)})$ is a vertex algebra module of $H^{\sbullet}(\widetilde{V},\EuScript{J}_{(0)})$.

One natural question is whether, when $\EuScript{C}$ is abelian, applying the functor $\mathrm{Hom}_{\EuScript{C}}(\mathbb{1},-)$ commutes with taking BRST cohomology. In general, this cannot be true. However, the Deligne category $\mathrm{Rep}(\mathrm{GL}_n)$ is abelian and semisimple when $n$ is not an integer. In this case, the functor $\mathrm{Hom}_{\mathrm{Rep}(\mathrm{GL}_n)}(\mathbb{1},-)$ is automatically exact, so it commutes with taking BRST cohomology.

\section{Vertex algebra in Deligne Category}
\label{sec:ver_Deligne}

\subsection{A general construction}
Before we go into the details of the large $N$ vertex algebra in \cite{Costello:2018zrm}, we first introduce a general construction studied in \cite{Gaiotto:2024dwr}. This construction assigns a vertex algebra $V$, together with a map $\EuScript{J}: \mathbb{1} \to V$, to a compact $2d$ Calabi-Yau algebra $A$. In the subsequent sections, we consider a special case of $A$ that corresponds to the vertex algebra studied in \cite{Costello:2018zrm}, in which $\EuScript{J}$ also induces a differential.

 Due to \cite{kontsevich2006notes}, we can take $A$ to be a (finite-dimensional) unital cyclic $A_\infty$ algebra, i.e., an $A_\infty$ algebra $(A, m_1, m_2, \dots)$ equipped with a symmetric and non-degenerate pairing $(-,-): A \otimes A \to \mathbb{C}[2]$, such that the expression $(a_0, m_n(a_1, \dots, a_n))$ is cyclically symmetric in the graded sense.

We denote the inverse of the pairing by $\eta \in A\otimes A$. We can choose a basis $\{e_1,\dots,e_n\}$ of $A$ and express $\eta$ as $\sum \eta^{ij}e_{i}\otimes e_{j}$. By definition, we have
\begin{equation}\label{eq:eta_id}
	(a,e_i)\eta^{ij} e_j = a\,.
\end{equation}
This element also defines a symmetric pairing on the linear dual $A^\vee $, which is given by
\begin{equation*}
	(f,g) = (f\otimes g)(\eta)\,.
\end{equation*}
If we choose the dual basis $\EuScript{B} = \{f^1,\dots,f^n\}$ of $A^\vee $, then $(f^i,f^j) = \eta^{ij}$. 
We define the following object in the graded version of Deligne category $\mathrm{Rep}^{\Z}(\mathrm{GL}_{[N]})$
\begin{equation*}
	X_{A^\vee } = \bigoplus_{f_i \in \EuScript{B}}\phi^i\,,
\end{equation*}
where each $\phi^i$ is a copy of $[1,1] = \bb\ww$ in degree $|f_i|+ 1$. It's easy to check that different choices of basis give rise to objects that are isomorphic. 
\begin{remark}
	More abstractly, we can define $X_{A^\vee } = \bb\ww\otimes A^{\vee}[-1]$ as an object in the category  $\mathrm{Rep}(\mathrm{GL}_{[N]})\boxtimes \mathrm{sVect}^{\Z} \cong \mathrm{Rep}^{\Z}(\mathrm{GL}_{[N]})$. In the abelian case, the tensor product $\boxtimes$ is the Deligne tensor product \cite{deligne2007categories}. However, in our case $\mathrm{Rep}(\mathrm{GL}_{[N]})$ is not abelian and we should use the (Ind completion of) Kelly tensor product \cite{Kelly1982,kelly1982basic}. Since this categorical construction is not essential in our paper, we will not go through the details here.
\end{remark}

The object $X_{A^\vee }$ is equipped with a degree $0$ symplectic form
\begin{equation*}
	\Omega = \sum_{ij} \Omega^{ij} \in \bigoplus_{i,j}\mathrm{Hom}(\phi^i\otimes\phi^j,\mathbb{1})
\end{equation*}
given by
\begin{equation}\label{eq:symp_gen}
	\Omega^{ij} = \eta^{ij}		\begin{tikzpicture}[scale = 0.6]
		\node (a1) at (0,0) {$\bb$};
		\node (a2) at (1,0) {$\ww$};
		\node (a3) at (2,0) {$\bb$};
		\node (a4) at (3,0) {$\ww$};
		\draw (a1) arc (180:7: 1.5 and 1.1);
		\draw (a3) arc (0:164: 0.5 and 0.4);
	\end{tikzpicture}\,.
\end{equation}
Then one can use our general construction in section \ref{sec:ex_betaga} to define the $\beta\gamma$ vertex algebra $V^{\beta\gamma}(A)$ generated by the object $X_{A^\vee }$. We denote $\phi^i(z)$ the field that correspond to the inclusion $\phi^{i}_{-1} \to V^{\beta\gamma}(A)$. From \eqref{eq:OPE_symp}, we can write the OPE of these fields as
\begin{equation}\label{eq:OPE_gen_model}
	\phi^i(z)\circ(\mathrm{id}_{\phi^{i}_{-1}}\otimes\phi^j(w)) = \frac{1}{z - w}l_V\circ(\Omega^{ij}\otimes\mathrm{id}_{V}) + :\phi^i(z)\phi^j(w):\,,
\end{equation}
where we denote $V = V^{\beta\gamma}(A)$ in the above formula.
\begin{remark}
Recall that $\mathrm{Rep}^{\mathbb{Z}}(\mathrm{GL}_{[N]})$ is defined so that the braiding picks up a degree-dependent sign. The pairing $(-,-): A^{\otimes 2} \to A$ has even degree, so on the even-degree part of $A^{\vee}$, $\eta^{ij}$ is symmetric, while on the odd-degree part, $\eta^{ij}$ is anti-symmetric. The fields $\phi^i$ have a degree shift from $A^\vee$ by $1$. Combining these facts with the OPE \eqref{eq:OPE_gen_model}, we see that the even-degree fields $\phi^i$ generate a $\beta\gamma$ system in the Deligne category, while the odd-degree fields generate a vertex algebra analogous to the $bc$ system.
\end{remark}

We further apply the construction in Corollary \ref{cor:inv_voa} and consider 
\begin{equation*}
	\EuScript{A}_N(A) = \mathrm{Hom}_{\mathrm{Rep}^{\Z}(\mathrm{GL}_{[N]})}(\mathbb{1},V^{\beta\gamma}(A))
\end{equation*}
as a vertex algebra over $\C[N]$. Before we study properties of $\EuScript{A}_N(A)$, we introduce some notation. For a graded vector space $V$, let $$CC^{n}(V) = \mathrm{Hom}_{\C}(V^{\otimes n+1},\C )^{\Z_{n+1}}$$ 
denote the space of cyclic invariant maps from $V^{\otimes n+1}$ to $\C$. Elements of $CC^{n}(V)$ are assigned a degree such that $|f| = \deg f + n $,\footnote{This grading is compatible with the grading of cyclic cohomology, which we will consider later.} where $\deg f$ is the degree of $f$ as a map $\mathrm{Hom}_{\mathbb{C}}(V^{\otimes n+1}, \mathbb{C})$. Therefore, the direct sum $CC^{\sbullet}(V) = \bigoplus_n CC^{n}(V)$ is a graded vector space. The grading is given such that we can identify $$CC^{\sbullet}(V) = \bigoplus_n \mathrm{Hom}_{\mathbb{C}}((V[1])^{\otimes n}, \mathbb{C})^{\Z_n}[1]\,.$$

\begin{lemma}\label{lem_isocyc}
	We have an isomorphism 
	\begin{equation*}
		\EuScript{A}_N(A) \cong S(CC^{\sbullet}(A[[t]])[-1])\otimes \C[N]\,.
	\end{equation*}
\end{lemma}
\begin{proof}
	From the definition of $X_{A^\vee }$ and the construction of $V^{\beta\gamma}$, we can identify $V^{\beta\gamma}(A) = S([1,1]\otimes A^\vee t^{-1}[t^{-1}][-1])$. Using Lemma \ref{lem_Hom1Sn} we have
	\begin{equation*}
		\begin{aligned}
			\mathrm{Hom}_{\mathrm{Rep}^{\Z}(\mathrm{GL}_{[N]})}(\mathbb{1},V^{\beta\gamma}(A)) &\cong \bigoplus_{n\geq 0}\big(\mathrm{Hom}(\mathbb{1},T^n([1,1]\otimes A^\vee t^{-1}[t^{-1}][-1]))\big)^{S_n}\\
			& = \bigoplus_{n\geq 0}(\C[N][S_n]\otimes (A^\vee t^{-1}[t^{-1}][-1])^{\otimes n})^{S_n}\\
			& = \bigoplus_{n\geq 0}(\C[S_n]\otimes  \mathrm{Hom}(A[[t]][1]^{\otimes n},\C))^{S_n}\otimes \C[N]\,.
		\end{aligned}
	\end{equation*}
	In the above, we used the fact that $$\mathrm{Hom}_{\mathrm{Rep}(\mathrm{GL}_{[N]})}(\mathbb{1},T^n[1,1]) = \mathrm{Hom}_{\mathrm{Rep}(\mathrm{GL}_{[N]})}(\mathbb{1},[n,n]) = \C[N][S_n]\,.$$
	
By tracing through the isomorphism, we observe that the action of $S_n$ on $\mathbb{C}[S_n]$ is by conjugation, while its action on $\mathrm{Hom}(A[[t]][1]^{\otimes n},\mathbb{C})$ is by permuting $A[[t]][1]^{\otimes n}$.

	Let $U_k\subset S_k$ be the conjugacy class of the cycle $(12\dots k)$. Since any permutation is a product of cycles, the natural map
	\begin{equation*}
		S(\bigoplus_{n\geq 0}(\C[U_n]\otimes  \mathrm{Hom}(A[[t]][1]^{\otimes n},\C))^{S_n} ) \to \bigoplus_{n\geq 0}(\C[S_n]\otimes  \mathrm{Hom}(A[[t]][1]^{\otimes n},\C))^{S_n} 
	\end{equation*}
is an isomorphism. Then the proof follows from the observation that 
\begin{equation*}
	\C[U_n]\otimes  \mathrm{Hom}(A[[t]][1]^{\otimes n},\C)^{S_n} \cong CC^{\sbullet}(A[[t]])[-1]\,.
\end{equation*}
\end{proof}
	
Under the above isomorphism, a cyclic map $f: A[[t]]^{\otimes n} \to \C$ is identified with the following element \footnote{We need to further symmetrize this expression to obtain an element in $\EuScript{A}_N(A)$.}
\begin{equation}\label{eq:cyc_to_op}
	\sum_{i_j,k_j} f(e_{i_1}t^{k_1},\dots,e_{i_n}t^{k_n})\Tr(\phi^{i_1}_{-k_1} \dots \phi^{i_n}_{-k_n})
\end{equation}
where
\begin{equation}\label{eq:Tr_general}
	\Tr(\phi^{i_1}_{-k_1} \dots \phi^{i_n}_{-k_n}) := \begin{tikzpicture}[scale = 0.6]
		\node (a1) at (0,0) {$\bb$};
		\node (a2) at (1,0) {$\ww$};
		\node (a3) at (2,0) {$\bb$};
		\node (a4) at (3,0) {$\ww$};
		\node (a5) at (4,0) {$\dots$};
		\node (a6) at (5,0) {$\bb$};
		\node (a7) at (6,0) {$\ww$};
		\draw (a1) arc (180:353: 3 and 1.1);
		\draw (a3) arc (360:190: 0.5 and 0.4);
		\draw (3.8,0) arc (360:190: 0.4 and 0.4);
	\end{tikzpicture}\in \mathrm{Hom}(\mathbb{1},\phi^{i_1}_{-k_1}\otimes \dots \phi^{i_n}_{-k_n})\,.
\end{equation}

Our next step is to define a BRST reduction for $\EuScript{A}_N(A)$. We specify an element $Q \in \EuScript{A}_N(A)$. By the previous lemma, we can define it as an element in $CC^{\sbullet}(A[[t]])$. We consider the following cyclic map
\begin{equation*}
	\sum_{n\geq 1} \Big( a_0(t)\otimes a_1(t)\otimes  \dots \otimes a_n(t) \to \frac{1}{(n+1)!}\left(a_0(0),m_n(a_1(0),\dots a_n(0) ) \right) \Big) \,,
\end{equation*}
where $a_i(t) \in A[[t]]$.

By an abuse of notation, we denote the corresponding map $Q_{(0)}:\EuScript{A}_N(A)\to \EuScript{A}_N(A)$ also by $Q$. Then this operator can be written as follows
\begin{equation*}
	Q =\sum_{n\geq 1} \sum_{i_0,\dots,i_{n}} \frac{1}{(n+1)!}\oint dz (e_{i_0},m_n(e_{i_1},\dots ,e_{i_n}))\Tr(:\phi^{i_0}(z)\dots \phi^{i_n}(z):)\,.
\end{equation*}

\begin{remark}
Here, we use the same notation as in \eqref{eqn_defX} and \eqref{eq:Tr_general}. In other words, we identify $\phi^i(z)$ with the field that corresponds to the natural inclusion $\phi^i_{-1} \to V^{\beta\gamma}(A)$. Thus, the normal ordering $:\phi^{i_0}(z)\dots \phi^{i_n}(z):$ is a field labeled by $\phi^{i_0}_{-1}\otimes\dots \otimes\phi^{i_n}_{-1}$. By composing with $\Tr$, understood here as a map $\mathbb{1} \to \phi^{i_0}_{-1}\otimes\dots \otimes\phi^{i_n}_{-1}$, we obtain a field of $\EuScript{A}_N(A)$.
\end{remark}

We emphasize that the mere condition that $A$ is a cyclic $A_\infty$-algebra is not sufficient to ensure that $Q$ is a differential. We also provide a counterexample in section \ref{sec:chiralSYM}. As a result, we cannot perform BRST reduction at this step. We will study a special case when $Q$ is a differential later.

\subsection{General properties}
In this section, we collect some properties of the vertex algebra coming from the functorial construction. 

In the above, we constructed the vertex algebra using the one-parameter version of the Deligne category $\mathrm{Rep}(\mathrm{GL}_{[N]})$. We can also consider the Deligne category $\mathrm{Rep}(\mathrm{GL}_n)$ for any $n \in \mathbb{C}$. In fact, we have a functor $\mathrm{Rep}(\mathrm{GL}_{[N]}) \overset{|_{N = n}}{\longrightarrow} \mathrm{Rep}(\mathrm{GL}_n)$, which is defined by specializing $N = n$. By Proposition \ref{prop:mfunctor_vertex}, the image of $V^{\beta\gamma}(A)$ under this functor, denoted by $V^{\beta\gamma}_{N = n}(A)$, is a vertex algebra in $\mathrm{Rep}(\mathrm{GL}_n)$. By construction, $V^{\beta\gamma}_{N = n}(A)$ can also be defined as the $\beta\gamma$ vertex algebra in $\mathrm{Rep}(\mathrm{GL}_n)$, generated by $\bb\ww\otimes A^{\vee}[-1]$. 

For $n$ a positive integer, we further have a functor $\mathrm{Rep}(\mathrm{GL}_n) \to \mathbf{Rep}(\mathrm{GL}_n)$ from the Deligne category to the actual representation category of the group $\mathrm{GL}_n$. Let us denote the composite functor $F_n:\mathrm{Rep}(\mathrm{GL}_{[N]}) \overset{|_{N = n}}{\to} \mathrm{Rep}(\mathrm{GL}_n) \to \mathbf{Rep}(\mathrm{GL}_n)$. The image $F_n(V^{\beta\gamma}(A))$ is a vertex algebra in $\mathbf{Rep}(\mathrm{GL}_n)$, which is the same as an ordinary vertex algebra together with an action of the group $GL_n$.

By construction, $F_n(V^{\beta\gamma}_{N}(A))$ is simply the $\beta\gamma$ vertex algebra generated by $\mathrm{Mat}_{n}\otimes A^{\vee}[-1]$. More specifically, this vertex algebra is generated by the fields $\phi^i_{kl}(z)$, which have degree $|f^i|+1$ for $i = 1, \dots, \dim A$ and $k, l = 1, \dots, n$, associated with the basis ${f^i}$ of $A^{\vee}$. Their operator product expansion is given by
\begin{equation*}
	\phi^i_{k_1l_1}(z)\phi^j_{k_2l_2}(w) \sim \frac{\eta^{ij}\delta_{l_1k_2}\delta_{k_1l_2}}{z-w}\,.
\end{equation*}

By Corollary \ref{cor:inv_voa}, we also have a vertex algebra $\mathrm{Hom}(\mathbb{1},F_n(V^{\beta\gamma}(A)))$. In the category $\mathbf{Rep}(\mathrm{GL}_n)$, $\mathrm{Hom}(\mathbb{1},F_n(V^{\beta\gamma}(A)))$ is the $\mathrm{GL}_n$ invariant vertex subalgebra $F_n(V^{\beta\gamma}(A))^{\mathrm{GL}_n}$. 

Then, as a corollary of Proposition \ref{prop:functor_vertex_vect}, we have the following
\begin{corollary}
	Let $n$ be a positive integer. Consider the ring map $\C[N] \to \C$ given by $N \to n$. We have a morphism of vertex algebra
	\begin{equation*}
		\widetilde{F}_n: \EuScript{A}_{N}(A)\otimes_{\C[N]}\C \to F_n(V^{\beta\gamma}(A))^{\mathrm{GL}_n}\,.
	\end{equation*}
\end{corollary}

Later, we will consider the special case where \( A = \mathbb{C}[\epsilon_1, \epsilon_2] \) and show that \( Q \) defines a differential in this setting. In fact, the BRST reduction of $ F_n(V^{\beta\gamma}(A))^{\mathrm{GL}_n} $ in this case gives us the vertex algebra arising from the $4d$ $\EuScript{N} = 4$ $U(n)$ super Yang-Mills theory via the $4d/2d$ duality constructed in \cite{Beem:2013sza}. Since the functor $F_n$ is full (but not faithful), the morphism $\widetilde{F}_n$ above is surjective (but not injective).

There is another interesting symmetric monoidal functor from the Deligne category to the ``multi"-Deligne category we defined in \ref{sec:var_Deligne}:  $\Delta:\mathrm{Rep}(\mathrm{GL}_{[N]}) \to \mathrm{Rep}(\mathrm{GL}_{[N_1,N_2]})$. By Proposition \ref{prop:mfunctor_vertex}, we obtain a vertex algebra $\Delta(V^{\beta\gamma}(A))$ in $\mathrm{Rep}^{\Z}(\mathrm{GL}_{[N_1,N_2]})$. 

By construction, $\Delta(V^{\beta\gamma}(A))$ is a $\beta\gamma$ vertex algebra in $\mathrm{Rep}^{\Z}(\mathrm{GL}_{[N_1,N_2]})$ generated by $\Delta(X_{A^{\vee}})$. We can check that $\Delta(X_{A^{\vee}})$ can be identified with $(\overset{1}{\bb}\overset{1}{\ww} + \overset{1}{\bb}\overset{2}{\ww} + \overset{2}{\bb}\overset{1}{\ww} + \overset{2}{\bb}\overset{2}{\ww})\otimes A^{\vee}[-1]$.

We denote  $\EuScript{A}_{N_1,N_2} = \mathrm{Hom}(\mathbb{1},\Delta(V^{\beta\gamma}(A)))$. $\EuScript{A}_{N_1,N_2}$ is a vertex algebra over $\C[N_1,N_2]$.

Then as a corollary of Proposition \ref{prop:functor_vertex_vect}, we have the following
\begin{corollary}
	Consider the ring map $\C[N] \to \C[N_1,N_2]$ given by $N \to N_1 + N_2$. We have a morphism of vertex algebra (over $\C[N_1,N_2]$)
	\begin{equation*}
		\widetilde{\Delta}: \EuScript{A}_{N}(A)\otimes_{\C[N]}\C[N_1,N_2] \to  \EuScript{A}_{N_1,N_2}(A)\,.
	\end{equation*}
\end{corollary}

There is another vertex algebra in $\mathrm{Rep}(\mathrm{GL}_{[N_1,N_2]})$ of interest. It is the $\beta\gamma$ vertex algebra generated by $(\overset{1}{\bb}\overset{1}{\ww} +\overset{2}{\bb}\overset{2}{\ww})\otimes A^\vee[-1]$. Because of the identity $S(V\oplus W) = S(V)\otimes S(W)$, we can identify this vertex algebra with $V_{\beta\gamma,N_1}(A)\otimes V_{\beta\gamma,N_2}(A)$. Its easy to check that $\mathrm{Hom}(\mathbb{1},V_{\beta\gamma,N_1}(A)\otimes V_{\beta\gamma,N_2}(A))  = \EuScript{A}_{N_1}(A)\otimes_{\C}\EuScript{A}_{N_2}(A)$. By construction, we have a natural inclusion $i: V_{\beta\gamma,N_1}(A)\otimes V_{\beta\gamma,N_2}(A) \to \Delta(V^{\beta\gamma}(A))$ of vertex algebra in $\mathrm{Rep}(\mathrm{GL}_{[N_1,N_2]})$. As a corollary of Proposition \ref{prop:fun_mor_vec}, we have
\begin{corollary}
	We have a morphism of vertex algebra
	\begin{equation*}
		i: \EuScript{A}_{N_1}(A)\otimes_{\C}\EuScript{A}_{N_2}(A)\to \EuScript{A}_{N_1,N_2}(A)\,.
	\end{equation*}
\end{corollary}

As a corollary of the module construction \ref{prop:ver_mod},\ref{prop:ver_mod_lat}, we have the following
\begin{corollary}
	\begin{enumerate}
		\item  For any object $X \in \mathrm{Rep}_{\mathrm{f}}(\mathrm{GL}_{[N]})$, we have a vertex algebra module $M_X = \mathrm{Hom}(X,V^{\beta\gamma}(A))$ over $\EuScript{A}_N(A)$.
		\item  The sum $V_{X^{\Z_{\geq 0}}} = \bigoplus_{n\geq0}M_{X^{\otimes n}}$ is a vertex algebra over $\C[N]$. In particular, $V_{[1,1]^{\Z_{\geq 0}}} = \bigoplus_{n\geq0}\mathrm{Hom}([n,n],V^{\beta\gamma}(A))$ is a vertex algebra over $\C[N]$.
	\end{enumerate}
\end{corollary}
From our discussion in section \ref{sec:module}, the BRST operator $Q$ constructed in the previous section also acts on $M_X$ and $V_{X^{\mathbb{Z}_{\geq 0}}}$. When $Q$ is a differential, the cohomology $H^{\sbullet}(M_X,Q)$ and $H^{\sbullet}(V_{X^{\mathbb{Z}_{\geq 0}}},Q)$ are also modules for the vertex algebra $H^{\sbullet}(\EuScript{A}_N(A),Q)$. We hope this construction is useful in studying the representation theory of $H^{\sbullet}(\EuScript{A}_N(A),Q)$ and its ``finite $N$" counterpart.

\subsection{The main example}
\label{sec:chiralSYM}
In this section, we consider the large $N$ vertex algebra studied in \cite{Costello:2018zrm}. It is a special case of our previous construction that corresponds to $A = \C[\epsilon_1,\epsilon_2]$, with both $\epsilon_i$ in degree $1$. $A$ is a graded-commutative algebra equipped with a trace map $(-) :A \to \C$ given by $(\epsilon_1\epsilon_2) = 1 = -(\epsilon_2\epsilon_1)$ and $0$ otherwise. This map further induces a non-degenerate pairing given by $(a,b) = (ab)$. Then we can use our previous construction to define the vertex algebra $V^{\beta\gamma}(\C[\epsilon_1,\epsilon_2])$ and $\EuScript{A}_N := \mathrm{Hom}(\mathbb{1},V^{\beta\gamma}(\C[\epsilon_1,\epsilon_2]))$.

We introduce some notation that will be used in this section. We denote $c$, $Z_1$, $Z_2$, $b$ as four copies of the object $[1,1]$, in degrees $-1$, $0$, $0$, $1$ respectively. The sum $X =  c \oplus Z_1 \oplus Z_2 \oplus b$ can be identified with $X_{A^\vee }$ we introduced in the last section with $A = \C[\epsilon_1,\epsilon_2]$. $X$ is equipped with a symplectic form induced by the pairing $(-,-)$.

Using the same notation as in the previous section, we denote $c(z) \in \mathrm{Hom}(c_{-1}\otimes V^{\beta\gamma},V^{\beta\gamma}((z)))$ the field that correspond to the inclusion $c_{-1} \to V_{\beta\gamma}$. Similarly, we define the fields $b(z),Z_1(z),Z_2(z)$. According to \eqref{eq:OPE_gen_model}, the OPE between these fields can be written as follows
\begin{equation}\label{eq:OPE_cano_ex}
\begin{aligned}
		&b(z)\circ(\mathrm{id}_{b_{-1}}\otimes c(w)) = \frac{1}{z-w}(\begin{tikzpicture}[scale = 0.6]
		\node (a1) at (0,0) {$\bb$};
		\node (a2) at (1,0) {$\ww$};
		\node (a3) at (2,0) {$\bb$};
		\node (a4) at (3,0) {$\ww$};
		\draw (a1) arc (180:7: 1.5 and 1.1);
		\draw (a3) arc (0:164: 0.5 and 0.4);
	\end{tikzpicture}\otimes \mathrm{id}_{V^{\beta\gamma}}) + :b(z)c(w):\\
	& Z_1(z)\circ(\mathrm{id}_{Z_{1,-1}}\otimes Z_2(w))= \frac{1}{z-w}(\begin{tikzpicture}[scale = 0.6]
		\node (a1) at (0,0) {$\bb$};
		\node (a2) at (1,0) {$\ww$};
		\node (a3) at (2,0) {$\bb$};
		\node (a4) at (3,0) {$\ww$};
		\draw (a1) arc (180:7: 1.5 and 1.1);
		\draw (a3) arc (0:164: 0.5 and 0.4);
	\end{tikzpicture}\otimes \mathrm{id}_{V^{\beta\gamma}}) + :Z_1(z)Z_2(w):\,.
\end{aligned}
\end{equation}

Then the BRST charge can be written as follows
\begin{equation*}
	Q = \oint dz\Tr(:b(z)c(z)c(z):) + \Tr(:c(z)[Z_{1}(z),Z_{2}(z)]:)\,.
\end{equation*}

\begin{prop}
	For $Q$ given as above, we have $Q^2 = 0$.
\end{prop}
\begin{proof}
	Let us denote $\EuScript{J}(z) = \Tr(:b(z)c(z)c(z):) + \Tr(:c(z)[Z_{1}(z),Z_{2}(z)]:)$. To prove $Q^2 = 0$, it suffices to show that the OPE between $\EuScript{J}(z)$ and itself vanish. As an illustration, we compute in a diagrammatic way the OPE between $\Tr(:b(z)c(z)c(z):)$ and itself. This is done by considering all possible Wick contractions using \eqref{eq:OPE_cano_ex}, and composed with the $\Tr$ represented by \eqref{eq:Tr_general}.
	
	For example, the Wick contraction $\wick{\Tr(:b c \c c:(z))\;\Tr(:\c b c c:(w))} $ gives us 
	\begin{equation*}
		\begin{aligned}
			\begin{tikzpicture}[scale = 0.6]
				\node (b1) at (3,0) {$\bb$};
				\node (b2) at (3.7,0) {$\ww$};
				\node (b3) at (4.4,0) {$\bb$};
				\node (b4) at (5.1,0) {$\ww$};
				\node (b5) at (5.8,0) {$\bb$};
				\node (b6) at (6.5,0) {$\ww$};
				\draw (b1) arc (180:354:1.75 and 0.7);
				\draw (b3) arc (360:190:0.35 and 0.3);
				\draw (b5) arc (360:190:0.35 and 0.3);
				\node (c1) at (7.8,0) {$\bb$};
				\node (c2) at (8.5,0) {$\ww$};
				\node (c3) at (9.2,0) {$\bb$};
				\node (c4) at (9.9,0) {$\ww$};
				\node (c5) at (10.6,0) {$\bb$};
				\node (c6) at (11.3,0) {$\ww$};
				\draw (c1) arc (180:354:1.75 and 0.7);
				\draw (c3) arc (360:190:0.35 and 0.3);
				\draw (c5) arc (360:190:0.35 and 0.3);
				\draw  (c1) arc (0:165:0.65 and 0.4);
				\draw (b5) arc (180:10:1.35 and 0.65);
			\end{tikzpicture}
			= 	&\frac{1}{z-w}	\begin{tikzpicture}[scale = 0.6]
				\node (b1) at (3,0) {$\bb$};
				\node (b2) at (3.7,0) {$\ww$};
				\node (b3) at (4.4,0) {$\bb$};
				\node (b4) at (5.1,0) {$\ww$};
				\node (b5) at (5.8,0) {$\bb$};
				\node (b6) at (6.5,0) {$\ww$};
				\node (b7) at (7.2,0) {$\bb$};
				\node (b8) at (7.9,0) {$\ww$};						
				\draw (b1) arc (180:354:2.45 and 0.7);
				\draw (b3) arc (360:190:0.35 and 0.3);
				\draw (b5) arc (360:190:0.35 and 0.3);
				\draw (b7) arc (360:190:0.35 and 0.3);
			\end{tikzpicture}\\
			=	& \frac{1}{z - w}\Tr(:bccc:(w))\,.
		\end{aligned}
	\end{equation*}
	We can check that the sum of all possible single Wick contractions cancel each other and gives us $0$. There are also contributions from double Wick contractions. For example, $\wick{\Tr(:\c2 b  c \c1 c:(z))\;\Tr(:\c1 b  c \c2 c:(w))} $ gives us 
	\begin{equation*}
		\begin{aligned}
			\begin{tikzpicture}[scale = 0.6]
				\node (b1) at (3,0) {$\bb$};
				\node (b2) at (3.7,0) {$\ww$};
				\node (b3) at (4.4,0) {$\bb$};
				\node (b4) at (5.1,0) {$\ww$};
				\node (b5) at (5.8,0) {$\bb$};
				\node (b6) at (6.5,0) {$\ww$};
				\draw (b1) arc (180:354:1.75 and 0.7);
				\draw (b3) arc (360:190:0.35 and 0.3);
				\draw (b5) arc (360:190:0.35 and 0.3);
				\node (c1) at (7.8,0) {$\bb$};
				\node (c2) at (8.5,0) {$\ww$};
				\node (c3) at (9.2,0) {$\bb$};
				\node (c4) at (9.9,0) {$\ww$};
				\node (c5) at (10.6,0) {$\bb$};
				\node (c6) at (11.3,0) {$\ww$};
				\draw (c1) arc (180:354:1.75 and 0.7);
				\draw (c3) arc (360:190:0.35 and 0.3);
				\draw (c5) arc (360:190:0.35 and 0.3);
				\draw  (c1) arc (0:165:0.65 and 0.4);
				\draw (b5) arc (180:10:1.35 and 0.65);
				\draw (c5) arc (0:170:3.45 and 0.8);
				\draw (b1) arc (180:10:4.15 and 1);
			\end{tikzpicture}
			= 	&\frac{N}{(z-w)^2}	\begin{tikzpicture}[scale = 0.6]
				\node (b1) at (3,0) {$\bb$};
				\node (b2) at (3.7,0) {$\ww$};
				\node (b3) at (4.4,0) {$\bb$};
				\node (b4) at (5.1,0) {$\ww$};
				\draw (b1) arc (180:354:1.05 and 0.7);
				\draw (b3) arc (360:190:0.35 and 0.3);
			\end{tikzpicture}\\
			=	& \frac{N}{(z - w)^2}\Tr(:c(z)c(w):)\,.
		\end{aligned}
	\end{equation*}
	On the other hand, $\wick{\Tr(:\c2 b  c \c1 c:(z))\;\Tr(:\c1  b \c2 c  c:(w))} $ gives us 
	\begin{equation*}
		\begin{aligned}
			\begin{tikzpicture}[scale = 0.6]
				\node (b1) at (3,0) {$\bb$};
				\node (b2) at (3.7,0) {$\ww$};
				\node (b3) at (4.4,0) {$\bb$};
				\node (b4) at (5.1,0) {$\ww$};
				\node (b5) at (5.8,0) {$\bb$};
				\node (b6) at (6.5,0) {$\ww$};
				\draw (b1) arc (180:354:1.75 and 0.7);
				\draw (b3) arc (360:190:0.35 and 0.3);
				\draw (b5) arc (360:190:0.35 and 0.3);
				\node (c1) at (7.8,0) {$\bb$};
				\node (c2) at (8.5,0) {$\ww$};
				\node (c3) at (9.2,0) {$\bb$};
				\node (c4) at (9.9,0) {$\ww$};
				\node (c5) at (10.6,0) {$\bb$};
				\node (c6) at (11.3,0) {$\ww$};
				\draw (c1) arc (180:354:1.75 and 0.7);
				\draw (c3) arc (360:190:0.35 and 0.3);
				\draw (c5) arc (360:190:0.35 and 0.3);
				\draw  (c1) arc (0:165:0.65 and 0.4);
				\draw (b5) arc (180:10:1.35 and 0.65);
				\draw (c3) arc (0:170:2.75 and 0.8);
				\draw (b1) arc (180:10:3.45 and 1);
			\end{tikzpicture}
			= -	&\frac{1}{(z-w)^2}	\begin{tikzpicture}[scale = 0.6]
				\node (b1) at (3,0) {$\bb$};
				\node (b2) at (3.7,0) {$\ww$};
				\node (b3) at (4.4,0) {$\bb$};
				\node (b4) at (5.1,0) {$\ww$};
				\draw (b1) arc (180:354:0.35 and 0.3);
				\  draw (b3) arc (180:354:0.35 and 0.3);
			\end{tikzpicture}\\
			=-	& \frac{1}{(z - w)^2}:\Tr(c(z))\Tr(c(w)):\,.
		\end{aligned}
	\end{equation*}
	Summing all contribution together, we find
	\begin{equation*}
		\Tr(:bcc:(z))\Tr(:bcc:(w)) \sim 2\frac{N\Tr(:c(z)c(w):)-:\Tr(c(z))\Tr(c(w)):}{(z-w)^2}\,.
	\end{equation*}
	Similarly, we can compute
	\begin{equation*}
		\begin{aligned}
			\Tr(:c[Z_1,Z_2]&:(z))\Tr(:c[Z_1,Z_2]:(w)) \\
			\sim &-2\frac{N\Tr(:c(z)c(w):)-:\Tr(c(z))\Tr(c(w)):}{(z-w)^2}\\
			& - 2\frac{\Tr(:cc[Z_1,Z_2]:(z))}{(z-w)}\,,
		\end{aligned}
	\end{equation*}
	and
	\begin{equation*}
		\Tr(:bcc:(z))\Tr(:c[Z_1,Z_2]:(w)) \sim \frac{\Tr(:cc[Z_1,Z_2]:(z))}{(z-w)}\,.
	\end{equation*}
	It follows that different contribution cancel, and we have $\EuScript{J}(z)\EuScript{J}(w) \sim 0$.
\end{proof}

As a consequence, $Q$ defines a differential on $\EuScript{A}_N$. The large $N$ vertex algebra studied in \cite{Costello:2018zrm} is then defined as the BRST cohomology, $H^{\sbullet}(\EuScript{A}_N,Q)$. We will explore this vertex algebra in more detail in the following.

\begin{remark}
The simplest compact $2d$ Calabi-Yau algebra is given by $\mathbb{C}[x]/(x^2)$, where $x$ is of degree $2$ and we define the trace map by $(x) = 1$. In this case, the corresponding vertex algebra is a $bc$ system $\{b(z),c(z)\}$, with the same OPE as in \eqref{eq:OPE_cano_ex}. The associated BRST operator is given by $Q = \oint dz \, \mathrm{Tr}(:b(z)c(z)c(z):)$. However, the computation in the above proof implies that $Q^2 \neq 0$. Hence, we cannot perform BRST reduction in this case.

On the other hand, there exist more complicated examples of $A$ such that $Q$ is a differential. One such example is discussed in \cite{Gaiotto:2024dwr}.
\end{remark}

\subsection{$\EuScript{N} = 4$ super-Virasoro symmetry}
In this section, we study a (small) $\EuScript{N} = 4$ super-Virasoro algebra in the vertex algebra $\EuScript{A}_N$. This property have already been studied in \cite{Costello:2018zrm,Gaiotto:2024dwr} for the large $N$ vertex algebra. It also hold for the finite $N$ vertex algebra studied in \cite{Beem:2013sza,Arakawa:2023cki}. Here, we provide a proof within our framework.
\begin{prop}
	\begin{enumerate}
		\item $\EuScript{A}_N$ is a conformal vertex algebra. The Virasoro generator is given by
		\begin{equation*}
			T(z) = \frac{1}{2}\Tr(-:Z_1(z)\partial Z_2(z) + Z_2(z) \partial Z_1(z) - 2b(z)\partial c(z):)
		\end{equation*}
		with central charge $-3N^2$.
		\item  $\EuScript{A}_N$ contains a $\mathfrak{sl}_2$ Kac-Moody vertex algebra $V_{-N^2/2}(\mathfrak{sl}_2)$ at level $-N^2/2$, which is generated by
		\begin{equation*}
			J^+(z) = \frac{1}{2}\Tr(:Z_1(z)^2:),\;\;J^0(z) = -\frac{1}{2}\Tr(:Z_1(z)Z_2(z):),\;\; J^-(z) = -\frac{1}{2}\Tr(:Z_2(z)^2:)\,.
		\end{equation*}
		Moreover, there are four fermionic fields
		\begin{equation*}
				G^{+}_i = \Tr(:b(z)Z_i(z):),\;\; G^{-}_i(z) = \Tr(:\partial c Z_{i}(z):),\;\; i = 1,2.
		\end{equation*}
		Together they form a (small) $\EuScript{N} = 4$ super-Virasoro algebra.
	\end{enumerate}
\end{prop}
\begin{proof}
	(1) To perform the computation, we use the OPEs \eqref{eq:OPE_cano_ex} together with following
	\begin{equation*}
			b(z)\circ(\mathrm{id}_{b_{-1}}\otimes \pa c(w)) = \frac{1}{(z-w)^2}(\begin{tikzpicture}[scale = 0.6]
				\node (a1) at (0,0) {$\bb$};
				\node (a2) at (1,0) {$\ww$};
				\node (a3) at (2,0) {$\bb$};
				\node (a4) at (3,0) {$\ww$};
				\draw (a1) arc (180:7: 1.5 and 1.1);
				\draw (a3) arc (0:164: 0.5 and 0.4);
			\end{tikzpicture}\otimes \mathrm{id}_{V_{\beta\gamma}}) + :b(z)\pa c(w):\,,
	\end{equation*}
	\begin{equation*}
			 \pa Z_1(z)\circ(\mathrm{id}_{Z_{1,-2}}\otimes Z_2(w))= \frac{-1}{(z-w)^2}(\begin{tikzpicture}[scale = 0.6]
				\node (a1) at (0,0) {$\bb$};
				\node (a2) at (1,0) {$\ww$};
				\node (a3) at (2,0) {$\bb$};
				\node (a4) at (3,0) {$\ww$};
				\draw (a1) arc (180:7: 1.5 and 1.1);
				\draw (a3) arc (0:164: 0.5 and 0.4);
			\end{tikzpicture}\otimes \mathrm{id}_{V_{\beta\gamma}}) + :\pa Z_1(z)Z_2(w):
	\end{equation*}
	that contain derivatives. As an example, the OPE between $\Tr(:b(z)\partial c(z):)$ and $\Tr(:b(w)\partial c(w):)$ contains single Wick contractions between $b$ and $\pa c$ and a double Wick contraction. One possible single Wick contraction looks like follows
		\begin{equation}\label{eq:TTbc_single1}
		\begin{aligned}
			\begin{tikzpicture}[scale = 0.6]
				\node (b3) at (4.4,0) {$\bb$};
				\node (b4) at (5.1,0) {$\ww$};
				\node (b5) at (5.8,0) {$\bb$};
				\node (b6) at (6.5,0) {$\ww$};
				\draw (b3) arc (180:350:1.05 and 0.5);
				\draw (b5) arc (360:190:0.35 and 0.3);
				\node (c1) at (7.8,0) {$\bb$};
				\node (c2) at (8.5,0) {$\ww$};
				\node (c3) at (9.2,0) {$\bb$};
				\node (c4) at (9.9,0) {$\ww$};
				\draw (c1) arc (180:354:1.05 and 0.5);
				\draw (c3) arc (360:190:0.35 and 0.3);
				\draw  (c1) arc (0:165:0.65 and 0.4);
				\draw (b5) arc (180:10:1.35 and 0.7);
			\end{tikzpicture}
			= &	\frac{-1}{(z - w)^2}	\begin{tikzpicture}[scale = 0.6]
				\node (b1) at (3,0) {$\bb$};
				\node (b2) at (3.7,0) {$\ww$};
				\node (b3) at (4.4,0) {$\bb$};
				\node (b4) at (5.1,0) {$\ww$};				
				\draw (b3) arc (360:190:0.35 and 0.3);
				\draw (b1) arc (180:350:1.05 and 0.5);
			\end{tikzpicture}
			= \frac{-1}{(z - w)^2}\Tr(:b(z)\pa c(w):)\\
			 = & \frac{-1}{(z - w)^2}\Tr(:b\pa c(w):) + \frac{-1}{(z - w)}\Tr(:\pa b\pa c(w):)\,.
		\end{aligned}
	\end{equation}
	The other single contraction gives
			\begin{equation}\label{eq:TTbc_single2}
		\begin{aligned}
			\begin{tikzpicture}[scale = 0.6]
				\node (b3) at (4.4,0) {$\bb$};
				\node (b4) at (5.1,0) {$\ww$};
				\node (b5) at (5.8,0) {$\bb$};
				\node (b6) at (6.5,0) {$\ww$};
				\draw (b3) arc (180:350:1.05 and 0.5);
				\draw (b5) arc (360:190:0.35 and 0.3);
				\node (c1) at (7.8,0) {$\bb$};
				\node (c2) at (8.5,0) {$\ww$};
				\node (c3) at (9.2,0) {$\bb$};
				\node (c4) at (9.9,0) {$\ww$};
				\draw (c1) arc (180:354:1.05 and 0.5);
				\draw (c3) arc (360:190:0.35 and 0.3);
				\draw  (c3) arc (0:170:2.05 and 0.9);
				\draw (b3) arc (180:5:2.75 and 1.2);
			\end{tikzpicture}
			= &	\frac{1}{(z - w)^2}	\begin{tikzpicture}[scale = 0.6]
				\node (b1) at (3,0) {$\bb$};
				\node (b2) at (3.7,0) {$\ww$};
				\node (b3) at (4.4,0) {$\bb$};
				\node (b4) at (5.1,0) {$\ww$};				
				\draw (b3) arc (360:190:0.35 and 0.3);
				\draw (b1) arc (180:350:1.05 and 0.5);
			\end{tikzpicture}
			= \frac{-1}{(z - w)^2}\Tr(:b(w)\pa c(z):)\\
			= & \frac{-1}{(z - w)^2}\Tr(:b\pa c(w):) + \frac{-1}{(z - w)}\Tr(: b\pa^2 c(w):)\,.
		\end{aligned}
	\end{equation}
	
	The double Wick contraction looks like follows
		\begin{equation}\label{eq:TTbc_double}
		\begin{tikzpicture}[scale = 0.6]
			\node (b3) at (4.4,0) {$\bb$};
			\node (b4) at (5.1,0) {$\ww$};
			\node (b5) at (5.8,0) {$\bb$};
			\node (b6) at (6.5,0) {$\ww$};
			\draw (b3) arc (180:350:1.05 and 0.5);
			\draw (b5) arc (360:190:0.35 and 0.3);
			\node (c1) at (7.8,0) {$\bb$};
			\node (c2) at (8.5,0) {$\ww$};
			\node (c3) at (9.2,0) {$\bb$};
			\node (c4) at (9.9,0) {$\ww$};
			\draw (c1) arc (180:354:1.05 and 0.5);
			\draw (c3) arc (360:190:0.35 and 0.3);
			\draw  (c1) arc (0:165:0.65 and 0.3);
			\draw (b5) arc (180:15:1.35 and 0.6);
			\draw  (c3) arc (0:170:2.05 and 0.9);
			\draw (b3) arc (180:5:2.75 and 1.2);
		\end{tikzpicture}=\frac{-N^2}{(z - w)^4}\,.
	\end{equation}
	Combining \eqref{eq:TTbc_single1},\eqref{eq:TTbc_single2},\eqref{eq:TTbc_double} gives us
	\begin{equation*}
		\Tr(:b \partial c(z):) \Tr(:b\partial c(w):)\sim \frac{-2\Tr(:b(w)\partial c(w):)}{(z-w)^2} + \frac{-\pa \Tr(:b(w)\partial c(w):)}{z-w} + \frac{-N^2}{(z - w)^4}\,.
	\end{equation*}
	Computation for the $\frac{-\epsilon^{ij}}{2}\Tr(:Z_i\pa Z_j:)$ is similar but more tedious. We find
	\begin{equation*}
		T(z)T(w) \sim \frac{2T(w)}{(z-w)^2} + \frac{\pa T(w)}{(z-w)} + \frac{-3N^2/2}{(z - w)^4}\,.
	\end{equation*}
	
	(2) The OPE between $J^0(z)$ and $J^+(w)$ is computed via a single contraction between $Z_1$ and $Z_2$
	\begin{equation}\label{eq:JJ_single}
		\begin{tikzpicture}[scale = 0.6]
			\node (b3) at (4.4,0) {$\bb$};
			\node (b4) at (5.1,0) {$\ww$};
			\node (b5) at (5.8,0) {$\bb$};
			\node (b6) at (6.5,0) {$\ww$};
			\draw (b3) arc (180:350:1.05 and 0.5);
			\draw (b5) arc (360:190:0.35 and 0.3);
			\node (c1) at (7.8,0) {$\bb$};
			\node (c2) at (8.5,0) {$\ww$};
			\node (c3) at (9.2,0) {$\bb$};
			\node (c4) at (9.9,0) {$\ww$};
			\draw (c1) arc (180:354:1.05 and 0.5);
			\draw (c3) arc (360:190:0.35 and 0.3);
			\draw  (c1) arc (0:165:0.65 and 0.4);
			\draw (b5) arc (180:10:1.35 and 0.7);
		\end{tikzpicture}
		= 	\frac{-1}{z-w}	\begin{tikzpicture}[scale = 0.6]
			\node (b1) at (3,0) {$\bb$};
			\node (b2) at (3.7,0) {$\ww$};
			\node (b3) at (4.4,0) {$\bb$};
			\node (b4) at (5.1,0) {$\ww$};				
			\draw (b3) arc (360:190:0.35 and 0.3);
			\draw (b1) arc (180:350:1.05 and 0.5);
		\end{tikzpicture}
		=	 \frac{-1}{z - w}\Tr(:Z_1^2:(w))\,.
\end{equation}
There are two such contractions, therefore 
\begin{equation*}
	J^0(z)J^+(w) \sim \frac{1}{z-w}J^+(w)\,.
\end{equation*}
Similarly, we have 
\begin{equation*}
	J^0(z)J^-(w) \sim \frac{-1}{z-w}J^-(w)\,.
\end{equation*}

The OPE between $J^+(z)$ and $J^-(w)$ contains both a single and a double contraction. The single contraction is similar to \eqref{eq:JJ_single}. The double contraction takes the following form:
	\begin{equation}\label{eq:JJ_double}
		\begin{tikzpicture}[scale = 0.6]
			\node (b3) at (4.4,0) {$\bb$};
			\node (b4) at (5.1,0) {$\ww$};
			\node (b5) at (5.8,0) {$\bb$};
			\node (b6) at (6.5,0) {$\ww$};
			\draw (b3) arc (180:350:1.05 and 0.5);
			\draw (b5) arc (360:190:0.35 and 0.3);
			\node (c1) at (7.8,0) {$\bb$};
			\node (c2) at (8.5,0) {$\ww$};
			\node (c3) at (9.2,0) {$\bb$};
			\node (c4) at (9.9,0) {$\ww$};
			\draw (c1) arc (180:354:1.05 and 0.5);
			\draw (c3) arc (360:190:0.35 and 0.3);
			\draw  (c1) arc (0:165:0.65 and 0.3);
			\draw (b5) arc (180:15:1.35 and 0.6);
			\draw  (c3) arc (0:170:2.05 and 0.9);
			\draw (b3) arc (180:5:2.75 and 1.2);
		\end{tikzpicture}=\frac{N^2}{(z - w)^2}\,.
\end{equation}
Counting the combinatorial factor, we find 
\begin{equation*}
	J^+(z)J^-(w) \sim \frac{2}{z-w}J^0(w) + \frac{-N^2/2}{(z - w)^2}\,.
\end{equation*}
The OPE between $J^0(z)$ and $J^{0}(w)$ only contain a double contraction and is similar to \eqref{eq:JJ_double}. We have
\begin{equation*}
	J^0(z)J^0(w) \sim \frac{-N^2/2}{2(z - w)^2}\,.
\end{equation*}
To summarize, $J^{\pm}(z),J^{0}$ generate the vertex algebra $V_{-N^2/2}(\mathfrak{sl}_2)$.

Computation involving the fermionic fields $G^{\pm}_i$ is similar. The only extra ingredient we need is the fact that $\Tr(:bb:) = 0$ and $\Tr(:\pa c\pa c:) = 0$ automatically as $b,c$ are in degree $\pm 1$. We find
\begin{equation*}
\begin{aligned}
		&G^+_i(z)G^{+}_j(w) \sim 0,\;\;G^-_i(z)G^{-}_j(w) \sim 0\,,\\
		&G^+_i(z)G^{-}_j(w) \sim \frac{-\epsilon_{ij} T(w) + \pa J_{ij}(w)}{z- w} + \frac{2 J_{ij}}{(z-w)^2}+\frac{-N^2}{(z-w)^3}\,,
\end{aligned}
\end{equation*}
where $J_{ij} = \Tr(:Z_iZ_j:)$. The OPEs between $T$ and $G$'s are given by
\begin{equation*}
	T(z)G^{\pm}_i(w) \sim \frac{\frac{3}{2}G^{\pm}_i(w)}{(z-w)^2} + \frac{\pa G^{\pm}_i(w)}{z - w}\,.
\end{equation*}
\end{proof}

\begin{remark}
	Simple computation also show that the above generators $T,J,G$ are also BRST $Q$ closed. In fact, this $\EuScript{N} = 4$ super-Virasoro algebra is a vertex subalgebra of $H^{\sbullet}(\EuScript{A}_N,Q)$ \cite{Costello:2018zrm}.
\end{remark}

\section{Vertex Poisson Structure}
\label{sec:chiralPoi}

\subsection{The first vertex Poisson structure}
In this paper, we will consider two vertex Poisson algebra limit of the vertex algebra $\EuScript{A}_N(A)$. We first consider the easier one, by exploring the fact that the $\beta\gamma$ vertex algebra is built on a Weyl algebra, which has a natural classical limit. 

Recall from section \ref{sec:ex_betaga} that we can introduce a parameter $\hbar$ and define the $\beta\gamma$ vertex algebra $V^{\beta\gamma}_{\hbar}(A)$, where the singular OPE of the fields $\phi^i(z)$ is multiplied by $\hbar$. We denote $\EuScript{A}_{N,\hbar}(A) = \mathrm{Hom}(\mathbb{1},V^{\beta\gamma}_{\hbar}(A))$, which is a vertex algebra over $\C[N,\hbar]$. In this section, we consider the classical limit with $\hbar \to 0$:  
\begin{equation*}
	\EuScript{A}_{N,\hbar = 0}(A) = \EuScript{A}_{N,\hbar}(A)/\hbar \EuScript{A}_{N,\hbar}(A)\,.
\end{equation*}

Since the $\hbar = 0$ specialization of $V^{\beta\gamma}_{\hbar}(A)$ is simply a commutative vertex algebra, $\EuScript{A}_{N,\hbar = 0}(A)$ is also a commutative vertex algebra. Then by \cite{frenkel2004vertex}, $\EuScript{A}_{N,\hbar = 0}(A)$ acquires the structure of a vertex Poisson algebra. Denote $Q_{0}$ as the differential on $\EuScript{A}_{N,\hbar = 0}(A)$ induced by $Q$. 

For simplicity, we will only consider the case when $A$ is a graded associative algebra in this section. Generalization to an $A_\infty$ algebra is straightforward but requires more involved computations. For an associative algebra $V$, recall that the space of cyclic maps $
CC^{\sbullet}(V) = \bigoplus_{n \geq 0} \mathrm{Hom}_{\mathbb{C}}(V[1]^{\otimes n}, \mathbb{C})^{\mathbb{Z}_n}[1]
$ is equipped with the standard Hochschild differential, given by (see e.g. \cite{loday2013cyclic})
\begin{equation*}
	b(f)(a_1,a_2,\dots,a_{n+1}) = \sum_{i = 1}^{n}(-1)^{i-1}f(a_1,\dots,a_ia_{i+1},\dots,a_{n+1}) + (-1)^{n}f(a_{n+1}a_1,a_2,\dots,a_{n})
\end{equation*}
for $f \in CC^{n-1}(V)$. The complex $(CC^{\sbullet}(V),b)$ is also called the Connes' complex.

We have the following proposition that improves the results of Lemma \ref{lem_isocyc}
\begin{prop}
	\begin{enumerate}
		\item $Q_0^2 = 0$.
		\item We have an isomorphism of chain complex
		\begin{equation*}
			(\EuScript{A}_{N,\hbar = 0}, Q_0) \cong (S(CC^{\sbullet}(A[[z]])[1])[N],b)\,.
		\end{equation*}
		where $b$ is the Hochschild differential on $CC^{\sbullet}(A[[z]])$ that compute the cyclic cohomology.
	\end{enumerate}
\end{prop}
\begin{proof}
	We only need to prove $(2)$ as the Hochschild differential is already known to satisfies $b^2 = 0$.
	
	As we have discussed, an element $f \in CC^{n-1}(A[[z]])$ correspond to an element of $\EuScript{A}_N(A)$ via \ref{eq:cyc_to_op}. It further correspond to a field via the state-field correspondence \eqref{eq:bg_ver_nor}. We denote $\Phi_f(z)$ the field that correspond to $f$. We have
	\begin{equation*}
		\Phi_f(z) = 	\sum_{i_l,k_l} f(e_{i_1}t^{k_1},\dots,e_{i_n}t^{k_n})\Tr(:\frac{1}{k_1!}\partial^{k_1}\phi^{i_1}(z) \frac{1}{k_2!}\partial^{k_2}\phi^{i_2}(z) \dots \frac{1}{k_n!}\partial^{k_n}\phi^{i_n}(z):)\,.
	\end{equation*}
	
	The associative product $\cdot$ of $A$, together with the pairing $(-,-)$ also gives us an element in $CC^{\sbullet}(A[[z]])$. It correspond to the field 
	\begin{equation*}
		\EuScript{J}(z) = \sum_{j_1,j_2,j_3}\frac{1}{3!}(e_{j_1},e_{j_2}e_{j_3})\Tr(:\phi^{j_1}(z)\phi^{j_2}(z)\phi^{j_3}(z):)\,.
	\end{equation*}
	
	By definition, the BRST differential is given by $\oint dz \EuScript{J}(z)$. Recall that the vertex Poisson algebra structure on $\EuScript{A}_{N,\hbar = 0}$ is obtained by extracting the order $\hbar^1$ terms in the vertex algebra $\EuScript{A}_{N,\hbar}$. Therefore, to compute $Q_0$, we compute order $\hbar^1$ term in the OPE between $\EuScript{J}(z)$ and $\Phi_f(w)$. By definition, these terms correspond to single Wick contraction. As an illustration, we consider the Wick contraction between $\phi^{j_1}(z)$ and $\phi^{i_2}(w)$. By differentiating \eqref{eq:OPE_gen_model}, we find
	\begin{equation*}
	\phi^i(z)\partial^k\phi^j(w) \sim \frac{k!\eta^{ij}}{(z - w)^{k+1}}(\begin{tikzpicture}[scale = 0.6]
		\node (a1) at (0,0) {$\bb$};
		\node (a2) at (1,0) {$\ww$};
		\node (a3) at (2,0) {$\bb$};
		\node (a4) at (3,0) {$\ww$};
		\draw (a1) arc (180:7: 1.5 and 1.1);
		\draw (a3) arc (0:164: 0.5 and 0.4);
	\end{tikzpicture}\otimes\mathrm{id}_{V}) \,.
	\end{equation*}
	Then we can compute that
	\begin{equation*}
\begin{aligned}
			&\EuScript{J}(z)\Phi_f(w) \overset{\phi^{j_1},\phi^{i_1} \text{ contraction}}{\sim}\\
			& \sum_{j_l,i_l,k_l}\frac{f(e_{j_2}e_{j_3}t^{k_1},\dots,e_{i_n}t^{k_n})}{3!(z - w)^{k_1+1}}\Tr(:\phi^{j_2}(z)\phi^{j_3}(z)\frac{1}{k_2!}\partial^{k_2}\phi^{i_2}(w)\dots \frac{1}{k_n}\partial^{k_n!}\phi^{i_n}(w):)\,.
\end{aligned}
	\end{equation*}
In the above, we used the identity \eqref{eq:eta_id}. To compute $Q_0$, we take the residue $\mathrm{Res}_{z = w}$ and find
	\begin{equation*}
		\sum_{j_l,i_l,k_l}\frac{1}{3!}f(e_{j_2}t^{k_0}e_{j_3}t^{k_1'},\dots,e_{i_n}t^{k_n})\Tr(:\frac{\partial^{k_0}\phi^{j_2}(w)}{k_0!}\frac{\partial^{k_1'}\phi^{j_3}(w)}{k_1'!}\frac{\partial^{k_2}\phi^{i_2}(w)}{k_2!}\dots \frac{\partial^{k_n}\phi^{i_n}(w)}{k_n!}:)\,.
	\end{equation*}
	The above field correspond to the map $$a_0(t)\otimes a_1(t) \otimes \dots a_n(t) \to f(a_0(t)a_1(t),a_2(t),\dots,a_n(t)).$$
	Then it is easy to see that after summing all possible Wick contraction, we have
	\begin{equation*}
			Q_0\Phi_f(w) = \Phi_{bf}(w)\,,
	\end{equation*}
	where $b$ is the Hochschild differential. Therefore, we have proved that $Q_0$ coincides with $b$ when restricted to $CC^{\sbullet}(A[[z]]) \subset S(CC^{\sbullet}(A[[z]]))$. The identification of $Q_0$ and $b$ on the entire $S(CC^{\sbullet}(A[[z]]))$ follows from the fact that $Q_0$ arises from a single Wick contraction, and thus, it must satisfies the Leibniz rule.
\end{proof}
\begin{remark}
In \cite{Costello:2018zrm}, a statement analogous to the above is obtained as a corollary of the Loday–Quillen–Tsygan (LQT) theorem \cite{Quillen1984,Tsygan_1983}. Roughly speaking, the inclusions of Lie algebras $\cdots \hookrightarrow \mathfrak{gl}_N \hookrightarrow \mathfrak{gl}_{N+1} \hookrightarrow \cdots$ induce a sequence of BRST cohomologies of the finite-$N$ vertex algebras, as vector spaces. The limit of this sequence can then be computed via the LQT theorem. However, the vertex algebra structures are not preserved under the morphisms induced by $\mathfrak{gl}_N \hookrightarrow \mathfrak{gl}_{N+1}$. Consequently, the limit construction in \cite{Costello:2018zrm} does not yield a definition of a vertex algebra.
\end{remark}

\subsection{The second vertex Poisson structure}
We proceed and consider a more complicated classical limit. To do this, we introduce another parameter $d$. We first consider $\EuScript{A}_{N,\hbar}(A)[d^{\pm\frac{1}{2}}]$ as a vertex algebra over $\mathbb{C}[N, \hbar, d^{\pm\frac{1}{2}}]$. Then we make the following re-parametrization
\begin{equation*}
	\EuScript{A}_{\lambda,d}(A) = \EuScript{A}_{N = \lambda/d^{\frac{1}{2}},\hbar = d^{\frac{1}{2}}}(A)[ d^{\pm\frac{1}{2}}]\,.
\end{equation*}
By construction, $\EuScript{A}_{\lambda,d}(A)$ is vertex algebra over $\C[\lambda,d^{\pm \frac{1}{2}}]$. As a $\C[\lambda,d^{\pm \frac{1}{2}}]$ module, we have the identification $\EuScript{A}_{\lambda,d}(A) = S(CC^{\sbullet}(A[[z]])[-1])[\lambda,d^{\pm \frac{1}{2}}]$. We consider the following subspace, via a Rees-type construction
\begin{equation*}
	\widetilde{\EuScript{A}}_{\lambda,d}(A) = \bigoplus_{n\geq 0} d^{\frac{n}{2}}S^{ n}(CC^{\sbullet}(A[[z]])[-1])[\lambda,d]\subset \EuScript{A}_{\lambda,d}(A)\,.
\end{equation*}

A priori, $\widetilde{\EuScript{A}}_{\lambda,d}(A)$ is only a module over $\C[d,\lambda]$. We show that it has the structure of vertex algebra over $\C[d,\lambda]$ inherited from $\EuScript{A}_{\lambda,d}(A)$.
\begin{theorem}
	$\widetilde{\EuScript{A}}_{\lambda,d}(A)$ is a vertex algebra over $\C[d,\lambda]$. Moreover, 
	\begin{equation*}
		\widetilde{\EuScript{A}}_{\lambda,d = 0}(A) =\widetilde{\EuScript{A}}_{\lambda,d}(A)/d\widetilde{\EuScript{A}}_{\lambda,d}(A)
	\end{equation*}
	is a commutative vertex algebra. Therefore, there is a vertex Poisson algebra structure on $\widetilde{\EuScript{A}}_{\lambda,d = 0}(A)$.
\end{theorem}
\begin{proof}
	For the first statement to hold, we need to show that the OPE for any two fields in $\widetilde{\EuScript{A}}_{\lambda,d}(A)$ is still in $\widetilde{\EuScript{A}}_{\lambda,d}(A)$. In other words, if we perform the computation in $\EuScript{A}_{N,\hbar}(A)[d^{\pm\frac{1}{2}}]$, the OPE coefficients
	\begin{equation}\label{eqn:OPEco_general}
			d^{\frac{a}{2}}\hbar^bN^{c} \overset{N = \lambda/d^{\frac{1}{2}},\hbar = d^{\frac{1}{2}}}{\longrightarrow} d^{\frac{a+b-c}{2}}\lambda^c
	\end{equation}
	must satisfy the conditions
	\begin{equation*}
		\begin{aligned}
			&c \geq 0,\\
			&a + b - c \text{ is a non-negative even integer.}
		\end{aligned}
	\end{equation*}
	For $\widetilde{\EuScript{A}}_{\lambda,d = 0}(A)$ to be a commutative vertex algebra, we also require that $a + b - c$ to be a positive even integer in the singular part of the OPE. 
	
	Since the regular part of the OPE is given by normally ordered product, the corresponding OPE coefficient is of order $1$. Hence we focus on the singular part, which is computed via Wick contraction. We prove by induction on the number of Wick contractions that every coefficient satisfies
		\begin{equation}\label{eq:ind_cond}
		\begin{aligned}
			&c \geq 0,\\
			&b - c \geq  0,\\
			&a + b - c \text{ is a positive even integer.}
		\end{aligned}
	\end{equation}
	Recall from Lemma \ref{lem_isocyc} that $\EuScript{A}_{N,\hbar}(A)$ has a basis given by multi-trace fields of the form $:\Tr(\phi^{i_1}\phi^{i_2}\dots) \Tr(\phi^{j_1}\phi^{j_2}\dots)\dots:$. By construction, $\widetilde{\EuScript{A}}_{\lambda,d}(A)$ have a $\C[d,\lambda]$ basis given by re-scaled multi-trace fields of the form $:d^{\frac{1}{2}}\Tr(\phi^{i_1}\phi^{i_2}\dots) d^{\frac{1}{2}}\Tr(\phi^{j_1}\phi^{j_2}\dots)\dots:$

	When there is a single Wick contraction, it must connects two single trace and produce one single trace field. There are two possibilities, depending on the length of the two traces. If at least one trace has length greater than $1$, the contraction take the following schematic form
	\begin{equation}\label{eq:wick_two_1}
		\wick{d^{\frac{1}{2}}\Tr(\dots\c\phi\dots)\dots d^{\frac{1}{2}}\Tr(\dots\c\phi\dots) \dots}\sim \frac{d^{\frac{1}{2}}\hbar}{z}d^{\frac{1}{2}}\Tr(...)\dots\,,
	\end{equation}
	where $\dots$ represents fields not contracted.
	The corresponding OPE coefficient is $d^{\frac{1}{2}}\hbar$, which satisfies the conditions \eqref{eq:ind_cond}.
	In the other case, we have
	\begin{equation}\label{eq:wick_two_2}
		\wick{d^{\frac{1}{2}}\Tr(\c\phi)\dots d^{\frac{1}{2}}\Tr(\c\phi) \dots}\sim \frac{d\hbar N}{z}\dots
	\end{equation}
	This OPE coefficient also satisfies the conditions \eqref{eq:ind_cond}.
	
Now suppose we already have $n$ Wick contractions and the OPE coefficient $d^{\frac{a}{2}}\hbar^b N^{c}$ satisfies the conditions \eqref{eq:ind_cond}. Adding one more Wick contraction always introduces an extra $\hbar$ factor. It might also introduce an additional factor of $N$ or $d^{\pm \frac{1}{2}}$. We discuss each possibility in more detail.

The extra Wick contraction either connects two distinct traces or connects two fields inside a single trace (after previous contraction). In the first case, we still have either \eqref{eq:wick_two_1} or  \eqref{eq:wick_two_2}, which produces the coefficient $d^{\frac{a+1}{2}}\hbar^{b+1} N^{c}$ or $d^{\frac{a+2}{2}}\hbar^{b+1} N^{c+1}$. We can check that the conditions in \eqref{eq:ind_cond} are still satisfied. In the second case when the Wick contraction connects two fields inside a single trace, we have three possibilities. We could have
	\begin{equation}\label{eq:wick_one_1}
	\wick{d^{\frac{1}{2}}\Tr(\dots\c\phi\dots\c\phi\dots)}\dots\sim \frac{d^{-\frac{1}{2}}\hbar}{z}d^{\frac{1}{2}}\Tr(...)d^{\frac{1}{2}}\Tr(...)\dots
\end{equation}
The corresponding OPE coefficient becomes $d^{\frac{a-1}{2}}\hbar^{b+1} N^{c}$.
The second possibility is 
	\begin{equation}\label{eq:wick_one_2}
	\wick{d^{\frac{1}{2}}\Tr(\dots\c\phi\c\phi\dots)}\dots\sim \frac{\hbar N}{z}d^{\frac{1}{2}}\Tr(...)\dots
\end{equation}
with OPE coefficient $d^{\frac{a}{2}}\hbar^{b+1} N^{c+1}$. The third possibility is
	\begin{equation}\label{eq:wick_one_3}
	\wick{d^{\frac{1}{2}}\Tr(\c\phi \c\phi)}\dots\sim \frac{d^{\frac{1}{2}} \hbar N^2}{z}\dots
\end{equation} 
with OPE coefficient $d^{\frac{a+1}{2}}\hbar^{b+1} N^{c+2}$. We can check that in all of the above three cases \eqref{eq:wick_one_1}, \eqref{eq:wick_one_2}, \eqref{eq:wick_one_3}, the conditions in \eqref{eq:ind_cond} are satisfied. This finishes the proof.
\end{proof}

The above theorem simply means that the OPE for any two fields in $\widetilde{\EuScript{A}}_{\lambda,d}(A)$ have coefficients in $\mathbb{C}[d,\lambda]$. Moreover, the order $d^0$ part only consists of the regular part of the OPE. The order $d^1$ part becomes the vertex Poisson structure by definition. We briefly describe this vertex Poisson structure.

From the above proof, the first Wick contraction always produces a factor of $d$. Hence, we are restricted to subsequent Wick contractions that do not contribute any additional $d$. From the proof, these correspond to Wick contractions within a single trace (after previous contraction). It is easy to see that under this condition, these subsequent contractions do not cross any previous contraction. We illustrate a generic contraction between two single-trace fields as follows:
\begin{equation*}
	\wick{d^{\frac{1}{2}}\Tr(\dots \c4 \phi\dots \c3\phi\c2\phi \dots \c1\phi \dots)d^{\frac{1}{2}}\Tr(\dots \c1\phi\dots \c2\phi\c3\phi \dots \c4\phi \dots)}
\end{equation*}

The above Wick contraction can be drawn as diagram in the Deligne category, which becomes ribbon graph. We can see that the OPEs that contribute to the vertex Poisson structure correspond precisely to those planar Wick contractions. This vertex Poisson algebra is also referred to as the planar algebra in \cite{Gaiotto:2024dwr}. The full algebra $\widetilde{\EuScript{A}}_{\lambda,d}(A)$ is thus a deformation quantization of the planar vertex Poisson algebra $\widetilde{\EuScript{A}}_{\lambda,d = 0}(A)$.

\noindent\textbf{Acknowledgments}
K.Z. would like to thank Kevin Costello, Davide Gaiotto, Sam Raskin, Wenjun Niu for helpful comments and conversations. The author is also grateful to the anonymous referee for many valuable suggestions that improved the quality of this paper. Part of this work was carried out while K.Z. was at Perimeter Institute. Research at Perimeter Institute is supported in part by the Government of Canada through the Department of Innovation, Science and Economic Development Canada and by the Province of Ontario through the Ministry of Colleges and Universities. K.Z. is also supported by Harvard University, CMSA.

\bibliographystyle{amsplain}
\providecommand{\href}[2]{#2}
\bibliography{DeligneVOA}
\end{document}